\documentclass[11pt,reqno]{amsart}

\usepackage[utf8]{inputenc}
\usepackage[margin=1.25in]{geometry}
\usepackage{hyperref}
\usepackage{appendix}
\usepackage{amsfonts}
\usepackage{amsthm}
\usepackage{amssymb}
\usepackage{stmaryrd} 
\usepackage{amsmath}
\usepackage{amsthm}
\usepackage[dvipsnames]{xcolor}
\usepackage{mathrsfs}
\usepackage{lipsum} 

\theoremstyle{plain}
\newtheorem{theorem}{Theorem}[section]

\newtheorem{lemma}[theorem]{Lemma}
\newtheorem{Proposition}[theorem]{Proposition}

\newtheorem{Definition}[theorem]{Definition}

\theoremstyle{remark}
\newtheorem{example}[theorem]{Example}

\usepackage{biblatex} 
\numberwithin{equation}{section}
\title[Spectral radius concentration for random matrices]{Spectral radius concentration for inhomogeneous random matrices with independent entries}

\author{Yi HAN}
\address{Department of Mathematics, Massachusetts Institute of Technology, Cambridge MA
}
\email{hanyi16@mit.edu}

\begin{document}

\begin{abstract}

Let $A$ be a square random matrix of size $n$, with mean zero, independent but not identically distributed entries, with variance profile $S$. When entries are i.i.d. with unit variance, the spectral radius of $n^{-1/2}A$ converges to $1$ whereas the operator norm converges to 2. Motivated by recent interest in inhomogeneous random matrices, in particular non-Hermitian random band matrices, we formulate general upper bounds for $\rho(A)$, the spectral radius of $A$, in terms of the variance $S$. We prove (1) after suitable normalization $\rho(A)$ is bounded by $1+\epsilon$ up to the optimal sparsity $\sigma_*\gg (\log n)^{-1/2}$ where $\sigma_*$ is the largest standard deviation of an individual entry; (2) a small deviation inequality for $\rho(A)$ capturing fluctuation beyond the optimal scale $\sigma_*^{-1}$; (3) a large deviation inequality for $\rho(A)$ with Gaussian entries and doubly stochastic variance; and (4) boundedness of $\rho(A)$ in certain heavy-tailed regimes with only $2+\epsilon$ finite moments and inhomogeneous variance profile $S$. The proof relies heavily on the trace moment method.

\end{abstract}

\maketitle

\section{Introduction}

Let $A=(a_{ij})$ be an $n$ by $n$ random matrix with independent, mean zero entries $a_{ij}$. Let $S:=(\mathbb{E}[|a_{ij}|^2])_{i,j}$ denote the variance profile of $A$. The aim of this paper is to obtain high probability upper bounds for $\rho(A)$, the spectral radius of $A$, for a wide class of variance profiles $S$ up to high precisions that capture both sparsity of the matrix $A$ and fluctuations at the spectral edge at the almost optimal scale.

The spectral radius $\rho(A)$ for a random matrix has important practical implications. Consider a simplified model of a system of linear ODEs driven by a random matrix $X$
\begin{equation}\label{solutionut}
\frac{d}{dt}u_t=-gu_t+Xu_t,\end{equation}
 where $g$ is a coupling constant, the spectral radius $\rho(X)$ controls the long time behavior of its solutions with large real part. It has since played an important part in analyzing the interplay of ecological stability and complexity, originating from the seminal paper of May \cite{may1972will}, see also \cite{allesina2015stability}. More recent applications concern neural networks \cite{sompolinsky1988chaos}, \cite{grela2017drives}, \cite{hennequin2014optimal}, \cite{rajan2006eigenvalue}, \cite{muir2015eigenspectrum}, \cite{gudowska2020synaptic} having a structured inhomogeneous variance profile $S$, which could be very sparse with large bulks of zero entries. In addition to the almost sure limit of $\rho(X)$, deriving finer asymptotics of $\rho(X)$ around its deterministic limit for large $n$ can possibly allow us to compute the solution to $u_t$ \eqref{solutionut} at much larger time scales at the critical coupling $g=\sqrt{\rho(X)}$, see \cite{chalker1998eigenvector}, \cite{erdos2018power}, \cite{erdHos2023randomly} for computing asymptotics of ODE systems at the critical coupling. Note however that in \cite{chalker1998eigenvector}, \cite{erdos2018power}, \cite{erdHos2023randomly} the variance profile $S$ is uniformly bounded from below or uniformly primitive in the terminology of \cite{erdos2018power}, yet these non-degeneracy assumptions on $S$ are not assumed throughout in this paper.

The homogeneous case, i.e. when $S=\frac{1}{n}\mathbf{1}\mathbf{1}^t$ where $\mathbf{1}$ is the all-ones vector of dimension $n$, has been much studied in the literature. The convergence of $\rho(A)$ to the deterministic limit 1 has been obtained under a fourth moment condition in \cite{Bai1986LimitingBO}, see also \cite{Geman1986THESR} where the same result was proven under stronger moment assumptions. When $A$ has Gaussian entries, the precise distribution of $\rho(A)$, namely:
\begin{equation}\label{fluctuationedge}
\rho(A)=1+\sqrt{\frac{\gamma_n}{4n}}+\frac{\mathcal{G}_n}{\sqrt{4n\gamma_n}},\quad \gamma_n:=\log n-2\log\log n-\log 2\pi,
\end{equation} where $\mathcal{G}_n$ converges to a Gumbel variable, 
has been uncovered in \cite{MR1986426}, \cite{MR3211006} via the exact integrable structure of the Gaussian matrix. Recently \cite{cipolloni2023universality} showed that the asymptotic expansion \eqref{fluctuationedge} is universal fora large class of random matrices with i.i.d. distributions.

Sparse i.i.d. matrices also received significant recent attention. Take for example the adjacency matrix of Erdös-Renyi digraph where we set $a_{ij}\sim \operatorname{Ber}(\frac{p_n}{n})$ independently for each $i,j$. Despite the sparsity when $p_n\ll n$, the variance profile is still very homogeneous, and high trace method has been applied in \cite{benaych2020spectral} for $p_n\gg\log n$ and the method of characteristic functions has been applied in \cite{coste2023sparse} to determine the spectral outliers even for constant $p_n\geq 1$. For $p_n\geq n^\epsilon$, finer edge statistics have been obtained in \cite{he2023edge}.

In this work we aim to get a unified understanding for $\rho(A)$ without any non-degeneracy or flatness condition on the variance $S$.
This is highly relevant for the theory of inhomogeneous structured random matrices (see \cite{van2017structured}, \cite{adamczak2024norms}) as previous works mostly consider only the largest singular value. In addition to the neural network models discussed before where sparsity and inhomogeneity are key model features, another guiding example is non-Hermitian random band matrices with large chunks of zero entries. Consider for example the block band matrix
\begin{equation}\label{blockcanonicalform}
A=\begin{pmatrix} D_1&U_2&&&T_m\\T_1&D_2&U_3&&\\&T_2&D_3&\ddots&\\&&\ddots&\ddots&U_m \\ U_1&&&T_{m-1}&D_m \end{pmatrix}
\end{equation}
where each $D_i,T_i,U_i$ are some $b_n\times b_n$ matrix with i.i.d. entries. Another important model is the periodic band matrix (with bandwidth $d_n$) where we set $a_{ij}=0$ for any $\frac{d_n-1}{2}<|i-j|<n-\frac{d_n-1}{2}$, and $a_{ij}$ are i.i.d. copies of $\frac{1}{\sqrt{d_n}}\xi$ if $|i-j|\leq \frac{d_n-1}{2}$ or $|i-j|\geq n-\frac{d_n-1}{2}$.  
Hermitian band matrix models have received very high attention in recent years, see the survey \cite{bourgade2018random} and references therein. The non-Hermitian case is less studied, see \cite{MR3857860}, \cite{han2024outliers}, \cite{latala2018dimension}, \cite{han2024circular} for some recent progress.

What would be a correct bound for $\rho(A)$? The closely related problem of bounding $\|A\|$ has been much studied. By \cite{bandeira2016sharp}, Theorem 3.1, in the case $a_{ij}\sim N(0,b_{ij}^2)$ the real Gaussian distribution with mean zero and variance $b_{ij}^2$, we have for any $\epsilon\in(0,\frac{1}{2})$,
\begin{equation}\label{whatistracywidom}
\mathbb{E}\|A\|\leq (1+\epsilon)\left(\max_i\sqrt{\sum_j b_{ij}^2}+\max_j\sqrt{\sum_i b_{ij}^2}\right)+\frac{5(1+\epsilon)}{\sqrt{\log(1+\epsilon)}}\max_{ij}|b_{ij}|\sqrt{\log n}.
\end{equation}
Recently \cite{brailovskaya2024extremal} improved this bound to capture fluctuations at the Tracy-Widom scale. Despite the obvious bound $\rho(A)\leq \|A\|$, the bound \eqref{whatistracywidom} does not effectively capture the bound on $\rho(A)$: assume $S$ is doubly stochastic, then there is a multiplicative constant $2(1+\epsilon)$ in the bound \eqref{whatistracywidom}, but at least for the homogeneous case we have $\rho(A)\leq 1+\epsilon$ by \cite{Bai1986LimitingBO},\cite{Geman1986THESR}. This discrepancy is due to the fact that $A$ is far from a normal matrix, so that controlling $\rho(A)$ is a fundamentally non-Hermitian problem. Despite this discrepancy, we can still learn some lessons from the Hermitian setting: a natural parameter for sparsity could still be \begin{equation}\sigma_*:=\sup_{i,j}\sqrt{\mathbb{E}[|a_{ij}|^2]},
\end{equation} and the maximal row / column sum of $S$, \begin{equation}\label{maximalrowcolumn}\sigma:=\max\left(\max_i\sqrt{\sum_j \mathbb{E}[|a_{ij}|^2]},\quad \max_j\sqrt{\sum_i \mathbb{E}[|a_{ij}|^2]}\right),\end{equation} would still be a good upper bound for $\rho(A)$ in many circumstances when $A$ has a good structure.

These two parameters $\sigma$ and $\sigma_*$ are good candidates to bound $\rho(A)$ when $A$ is the random band matrix \eqref{blockcanonicalform}, and more generally (informally speaking) when $S$ has some regular patterns. However they will be less precise when the pattern of $S$ is highly irregular. This is yet another fundamentally non-Hermitian phenomena, which to our best knowledge has no analogue in the Hermitian context of bounding large singular values. An important example was studied in \cite{alt2021spectral} where it assumed that, for some $0<c<C$ fixed,
\begin{equation}\label{entryvariance}
\frac{c}{n}\leq\mathbb{E}[|a_{ij}|^2]\leq \frac{C}{n},
\end{equation} and assuming sufficient moments on the entry of $A$, then for any $\epsilon>0$ we have with high probability
\begin{equation}\label{whatboundrhoas}
\rho(A)\leq \rho(S)+n^{-1/2+\epsilon}.
\end{equation}
We will make a mild improvement to this bound in Theorem \eqref{spectralradiuslargedeviation}. From this result one may be tempted to bound $\rho(A)$ by $\rho(S)$ for an arbitrary variance $S$ without assuming \eqref{entryvariance}: this is clearly a too general statement which may not be provable without any structure on $S$. Instead, we introduce
another terminology, the long time control of $S$ by $\sigma$ in Definition \eqref{longtimeshorttime}, which provides a much better control of $\rho(A)$ when $S$ is far from being doubly stochastic.

Many other related results are proved along the lines. We summarize the main results of this paper into the following four categories:

\begin{enumerate}
    \item In Theorem \ref{convergencespectral} we show that whenever $\sigma_*\sqrt{\log n}\to 0$ then $\rho(A)$ can be bounded by $\sigma(1+\epsilon)$ with high probability. This can be compared with \eqref{whatistracywidom} which shows in the same range of $\sigma_*$, we have $\mathbb{E}\|A\|\leq 2(1+\epsilon)\sigma$. That is, we remove the factor 2 on the bound of $\rho(A)$ from the bound on $\|A\|$ all the way down to the optimal scale $\sigma_*\sqrt{\log n}\to 0$ (optimality is shown in Theorem \ref{unboundedspectralradius}).
    \item In Theorem \ref{Theorem1.6a12} we prove small deviation inequalities for $\rho(A)$ with a precision close to the optimal scale as predicted in \eqref{fluctuationedge}: we prove fluctuation up to the scale $\sigma_*\log n$ and the bound is effective whenever $\sigma_*\log n\to 0$. A better parameter is invented in Definition \ref{longtimeshorttime} to bound $\rho(A)$ in leading order, replacing $\sigma$ in \eqref{whatboundrhoas} to better capture non-normal fluctuations. Further, when $S$ satisfies \eqref{entryvariance} we prove in Theorem \ref{spectralradiuslargedeviation} a small deviation inequality that improves the scale of \eqref{whatboundrhoas}.

    \item When the variance $S$ is doubly stochastic and $A$ has Gaussian entries, we prove in Theorem \ref{theorem1.5subgaussianconcentration} a large deviation inequality for $\rho(A)$. This appears to be the first large deviation inequality for $\rho(A)$ capturing the correct dependence on $\sigma_*$.
    \item An important property of random matrix with i.i.d. entries \cite{WOS:000435416700013}, \cite{bordenave2021convergence} is that its spectral radius remains bounded even when the entries do not have a fourth moment. In Theorem \ref{upperboundsecondmoment} we generalize this to show, when $\sigma_*=O(n^{-1/2}),$ then $\rho(A)$ is also bounded by $\sigma$ (or $\rho(S)$) with high probability when entries of $A$ are symmetric with only $2+\epsilon$ finite moments.
\end{enumerate}

In the following we state the main results of this paper, divided into four different categories.

\subsection{Convergence of spectral radius under optimal sparsity}
The first result in this paper states that $\rho(A)$ can be well controlled by  $(1+\epsilon)\sigma$ whenever $\sigma_*\ll(\log n)^{-1/2}$. Compared to \eqref{whatistracywidom} which shows $\rho(A)\leq \|A\|\leq 2(1+\epsilon)\sigma$ whenever $\sigma_*\ll(\log n)^{-1/2}$, the main interest is we remove the factor 2 and capture non-commutativity at optimal sparsity.

\begin{theorem}\label{convergencespectral}(Convergence of spectral radius) Let $A=(b_{ij}g_{ij})$ be an $n\times n$ random matrix with independent entries. Assume that $b_{ij}$ are some fixed scalars and $g_{ij}$ are i.i.d. random variables that are either real Gaussian with distribution $\mathcal{N}_\mathbb{R}(0,1)$, or complex Gaussian with distribution $\mathcal{N}_\mathbb{C}(0,1)$. 
 Denote by $$\sigma:=\max\left(\sup_{i\in[n]}\sum_{j=1}^nb_{ij}^2,\sup_{j\in[n]}\sum_{i=1}^n b_{ij}^2\right),\quad \sigma_*=\sup_{i,j}b_{ij}$$
 and assume $\sigma/\sigma_*\leq n$.
Then there is a universal constant $C>0$ such that, for any $\epsilon>0$, we have the non-asymptotic bound
$$
\mathbb{E}[\rho(A)]
\leq (1+\epsilon)(\sigma+\frac{6\sigma_*}{\sqrt{\log(1+\epsilon})}\sqrt{\log n})\left(1+C\frac{\sqrt{6\frac{\log n}{\log(1+\epsilon)}}}{\sqrt{(\frac{\sigma}{\sigma_*})^2+6\frac{\log n}{\log(1+\epsilon)}}}\right).
$$
In particular, when $\sigma_*=o(\frac{1}{\sqrt{\log n}})$, then almost surely the outliers are absent:
  $$
\mathbb{P}(\lim\sup_{n\to\infty}\rho(A)>\sigma)=0.
  $$
  More generally, all the statements are true whenever the distributions of $g_{ij}$  satisfy
  
  \begin{enumerate}
    \item $\mathbb{E}[g_{ij}]=0, \quad \mathbb{E}[|g_{ij}|^2]=1$, $g_{ij}$ has a symmetric distribution:$$g_{ij}\overset{\text{law}}{\equiv}-g_{ij}.$$
    \item (Sub-Gaussian tail) For each $m\in\mathbb{N}_+$, \begin{equation}\label{gaussianmomentsrealgaussian}\mathbb{E}[|g_{ij}|^{2m}]\leq (\text{const}\cdot m)^m.\end{equation}
\item When $g_{ij}$ are complex-valued, we further require that the law of $g_{ij}$ is rotational invariant, i.e. $g_{ij}$ has the same distribution as $e^{2\pi i\mathcal{U}}g_{ij}$ where $\mathcal{U}\sim\operatorname{Unif}([0,1])$. 
\end{enumerate}
\end{theorem}

Here we briefly elaborate on what we mean by outliers of $A$, and why we assume that the bulk of the eigenvalues should lie in the ball $B(0,\sigma)$ in the complex plane. When A has i.i.d. entries, it is known \cite{WOS:000281425000010} that the empirical spectral measure of $A$ converges to the circular law, the uniform measure on $B(0,\sigma)$. For sparse i.i.d. matrices, the circular law was also proven in \cite{rudelson2019sparse}, improving \cite{wood2012universality} and \cite{basak2019circular}. It is a well-known conjecture to prove that even when $A$ has inhomogeneous entries, we also have convergence of the ESD of $A$ to the circular law under reasonable assumptions. No general solution to this conjecture is known up to date, for lack of sufficient technical machinery to control $\sigma_{min}(A)$, the smallest singular value of $A$. However, it is still widely believed that the ESD converges to the circular law in the setting of Theorem \ref{convergencespectral} when the variance profile of $A$ is doubly stochastic ($\sum_j b_{ij}^2=\sum_i b_{ij}^2=1$) \cite{tikhomirov2023pseudospectrum}, and thus we consider eigenvalues outside $B(0,\sigma)$ as outlying eigenvalues. Of course, we conclude by mentioning that convergence of ESDs does not imply absence of outliers, although absence of outliers is much easier to prove than convergence of ESDs for non-Hermitian matrices.

An associated problem with Theorem \ref{convergencespectral} is whether we can prove an analogous lower bound for $\mathbb{E}[\rho(A)]$ of the form $\mathbb{E}[\rho(A)]\geq c\sigma$ for some $c>0$. Unfortunately we have not been able to do so, and to our best knowledge the only way to establish such a lower bound is to prove convergence of ESDs of A to a limiting measure (see the last paragraph). This is in stark contrast to the case of a symmetric (Hermitian) matrix $B$ where by the trace moment method we have $\|B\|\leq (\operatorname{tr}(B^p))^\frac{1}{p}\leq n^\frac{1}{p}\|B\|$, so that the computation of $\operatorname{tr}(B^n)$ automatically gives rise to a lower bound on $\|B\|$. Still, we claim that $\sigma>0$ is the natural parameter in Theorem \ref{convergencespectral} for $\rho(A)$ because the circular law heuristic suggests so.

We note that a similar result was proven in \cite{benaych2020spectral}, Theorem 2.11 which also captures the optimal scale on $\sigma_*$ and does not require $g_{ij}$ to have a symmetric law, but made a flatness assumption in \cite{benaych2020spectral}, equation (2.3)
which rules out random band matrices \eqref{blockcanonicalform} and all examples when the variance of some entry is $\Omega(n^{-\epsilon})$ for some $\epsilon\in(0,\frac{1}{10})$. In other words, our method captures the most extreme case where the variance is concentrated on a few entries (band matrices, etc.), while the method in \cite{benaych2020spectral} is more useful for the random graph model where the variance profile is homogeneous.

In Appendix \ref{appendix1} we show in Theorem \ref{unboundedspectralradius} that our assumption $\sigma_*=o(\frac{1}{\sqrt{\log n}})$ is sharp for convergence of $\rho(A)$ by a diagonal Gaussian matrix example.

\subsection{Small deviation bounds}
Theorem \ref{convergencespectral} captures convergence of $\rho(A)$ under optimal sparsity, at the order $O(1)$, but does not capture the fluctuations at a smaller $o(1)$ scale. Now we further develop small deviation inequalities for the spectral radius up to the almost optimal scale predicted by \eqref{fluctuationedge}, represented in terms of $\sigma_*$. We consider slightly larger $\sigma_*$, and thus weaken the sub-Gaussian tail assumption by a sub-exponential tail, and remove the requirement that the random variables have a symmetric distribution.

\subsubsection{Upper bound via row sums and long-time averages}

\begin{Definition}\label{generalmatrixmodel}(General matrix model)
Let $X=(x_{ij})$ be a $n\times n$ square matrix with independent entries, and each $x_{ij}$ satisfies
\begin{enumerate}
    \item $\mathbb{E}[x_{ij}]=0,\quad \mathbb{E}[|x_{ij}|^2]=1$;
    \item (Sub-exponential tail) For a fixed constant independent of $i,j\in [n]$ we have, for each $m\in\mathbb{N}_+$, \begin{equation}\label{gaussianmoments}\mathbb{E}[|x_{ij}|^{m}]\leq (\text{const}\cdot m)^{m}.\end{equation}
\end{enumerate}
 Let $(b_{ij})_{i,j=1}^n$ be non-negative scalars, and define the inhomogeneous matrix $$A=(b_{ij}x_{ij})_{i,j=1}^n.$$ Define the variance matrix $S:=(b_{ij}^2)_{i,j=1}^n$ and define the sparsity parameter
 $$
\sigma_*=\sigma_*(A)=\sup_{i,j}b_{ij}.
 $$
\end{Definition}

We introduce a parameter, in place of $\sigma$ \eqref{maximalrowcolumn}, that controls the spectral radius:

\begin{Definition}(Long-time control)
\label{longtimeshorttime} Let $S$ be an $n\times n$ square matrix with non-negative entries.
We say $S$ is long-time controlled by a parameter $\sigma>0$ [$\sigma$ can be $n$-dependent but we assume $\sigma=O(1)$] if there exists a fixed (independent of $n$) constant $C=C(
\sigma)>0$ such that
$$ \sup_i
\sum_{j=1}^n[S^k]_{ij}\leq C\sigma^{2k},\quad \sup_i
\sum_{j=1}^n[(S^*)^k]_{ij}\leq C\sigma^{2k},\quad\text{ for all } k\geq 1,
$$
where $[S^k]_{ij}$ denotes the $(i,j)$-th entry of the product matrix $S^k$, and $S^*$ is the transpose of the matrix $S$.    
\end{Definition}
The notion of Definition \ref{longtimeshorttime} is more flexible than the $\ell^2$ norm of columns and rows of $S$, as shown in the following most simple example:
\begin{example}\label{bandmatrix}(Doubly stochastic) When $A$ has doubly stochastic variance profile, so that 
\begin{equation}\label{doublestochastic}\sum_{j=1}^n b_{ij}^2\leq 1 \quad \text{for all } i\in [n]; \quad \sum_{i=1}^n b_{ij}^2\leq 1\quad \text{for all } j\in[n],
\end{equation} $S$ is long-time controlled by 1, as we can simply take $\sigma=1$ and $C=1$ in Definition \eqref{longtimeshorttime}.
\end{example}

\begin{example}\label{example206}(Heterogeneous block matrix) Consider two independent $n\times n$ random matrices $T_1,T_2$ with mean zero, unit variance sub-Gaussian entries. Fix $\lambda_1,\lambda_2>0$ and consider the $(2n\times 2n)$ matrix 
$$
T:=n^{-1/2}\begin{pmatrix}
    0&\lambda_1T_1\\\lambda_2T_2&0
\end{pmatrix}.
$$ The largest $\ell^2$ norm of  any row (or column) of the variance profile of $T$ is $\max(\lambda_1,\lambda_2)$, but it is easy to see the variance profile of $T$ is long-time controlled by $\sqrt{\lambda_1\lambda_2}$, so we expect to control $\rho(T)$ by $\sqrt{\lambda_1\lambda_2}$, which is much smaller than $\sup(\lambda_1,\lambda_2)$ when $\lambda_1,\lambda_2$ are not close. 
    
\end{example}

A further example is deferred to \ref{example12}. These examples illustrate that the long-term control in Definition \ref{longtimeshorttime} can better capture $\rho(A)$ than $\sigma$ in \eqref{maximalrowcolumn}. 
The underlying reason is that for a  non-Hermitian random matrix $A$, its spectral radius $\rho(A)$ is less sensitive to the maximum $\ell^2$ norm of each single row, compared to the operator norm $\|A\|$.

Our main theorem in this direction can be stated as:

\begin{theorem}(Small deviations)\label{Theorem1.6a12} In the setting of Definition \ref{generalmatrixmodel} and \ref{longtimeshorttime}, 
let $\underline{\sigma}>0$ be such that $S$ is long-time controlled by $\underline{\sigma}$. Assume that $\sigma_*\ll (\log n)^{-1}$, then we have the following estimate: for any $t>0$ we have 
$$
\mathbb{P}\left(\rho(A)\geq \underline{\sigma}(1+t\sigma_*)\right)\leq C_0ne^{-Ct},
$$ where $C_0,C>0$ are two positive constants depending only on the sub-exponential moments of $x_{ij}$ in Definition \ref{generalmatrixmodel} and the constant $C(\underline{\sigma})$ in Definition \ref{longtimeshorttime}.

 In particular, whenever $A$ satisfy Definition \ref{generalmatrixmodel} (so that $x_{ij}$ has sub-exponential tails and not necessarily symmetric distribution), and $\sigma_*\ll(\log n)^{-1}$, then $$\mathbb{P}(\lim\sup_{n\to\infty}\rho(A)\leq\underline{\sigma})=1.
    $$
\end{theorem}

Theorem \ref{Theorem1.6a12} captures the fluctuation of $\rho(A)$ on an almost optimal scale. Taking, for example, $\underline{\sigma}=1$ and $\sigma_*=n^{-1/2}$, then Theorem \ref{Theorem1.6a12} implies with high probability $\rho(A)\leq 1+\frac{(\log n)^{1+c}}{\sqrt{n}}$ for any $c>0$ and the possibility that $\rho(A)\geq 1+n^{-d}$ is $O(ne^{-n^{1/2-d}})$ for any $0<d<1/2$. Compared to the  three terms expansion \eqref{fluctuationedge}, these estimates are valid on an almost optimal scale, and are generalized to the most general inhomogeneous model in Theorem \ref{Theorem1.6a12} by adjusting $\underline{\sigma}$ and $\sigma_*$ accordingly. However, our method is unlikely to recover the exact form of the three-term expansion \eqref{fluctuationedge} due to the moment method itself. When we apply the moment method, we bound $$\rho(A)^m=\rho(A^m)\leq \|A^m\|=\|A^m(A^*)^m\|^{1/2}\leq\operatorname{tr}(A^m(A^*)^m)^{1/2}$$ for suitable $m\in\mathbb{N}_+$, but the first and last inequalities are not exact and the inaccuracy in these inequalities will prevent us from the sharp asymptotic in \eqref{fluctuationedge}. In particular, the first inequality $\rho(A^m)\leq\|A^m\|$ is far from exact, and we are not aware of a method to make it into a more accurate form.

 In the last statement of Theorem \ref{Theorem1.6a12} we proved almost sure boundedness of $\rho(A)$ up to $\sigma_*\ll(\log n)^{-1}$ without assuming a symmetric entry distribution, but slightly missed the range $\sigma_*\ll(\log n)^{-1/2}$ in Theorem \ref{convergencespectral}. (As we only require a sub-exponential moment, $\sigma_*\ll(\log n)^{-1}$ is likely to be an optimal condition here). For adjacency matrix of homogeneous directed regular graphs, the optimal scale $\sigma_*\ll (\log n)^{-1/2}$ was reached in \cite{benaych2020spectral}, Theorem 2.11 under the assumption that the variance profile is relatively flat. For symmetric matrices an analogous result was previously derived in \cite{latala2018dimension}, Example 4.10.

\subsubsection{Upper bound via spectral radius}
To further illustrate a bound on $\rho(A)$ taking into account the non-Hermitian structure of its variance matrix $S$, we show in the following that $\rho(A)$ can also be well-controlled by $\sqrt{\rho(S)}$ given that the variance $S$ is relatively flat, that is, if $S$ admits a two-sided bound \eqref{twosidedbounds}. This was previously proven in \cite{alt2021spectral} and here we provide a different perspective and slight strengthening.

\begin{theorem}\label{spectralradiuslargedeviation} Let $X$ be the random matrix in Definition \ref{generalmatrixmodel} with variance profile $S$.
    Assume that we can find fixed $0<c<C$ such that 
    \begin{equation}\label{twosidedbounds}
\frac{c}{n}\leq b_{ij}^2\leq \frac{C}{n}\end{equation}
     for each $i,j\in [n]$, and we let $\rho(S)$ denote the spectral radius of $S$. Then we have the small deviation estimate:
for any $t>0$ we have 
$$
\mathbb{P}\left(\rho(A)\geq \sqrt{\rho(S)}+\frac{t}{\sqrt{n}})\right)\leq C_0n^2e^{-C_1t},
$$ where $C_0,C_1>0$ are two constants depending only on sub-exponential moments of $x_{ij}$ in Definition \ref{generalmatrixmodel} and the constants $c,C$ in equation \eqref{twosidedbounds}.

\end{theorem}

In particular, we show that $\rho(A)\leq\sqrt{\rho(S)}+\frac{(\log n)^{1+c}}{\sqrt{n}}$ with high probability for any $c>0$. Similar results were first proven in \cite{alt2021spectral} which shows that $\rho(A)\leq\sqrt{\rho(S)}+n^{\epsilon-1/2}$ with probability at least $1-c_{D,\epsilon}/n^D$ for any $D\in\mathbb{N}_+$. In \cite{alt2021spectral} it was only assumed that the entries have finite moments of all order, whereas we assume $x_{ij}$ has sub-exponential tails. Under this assumption we improve the bound in \cite{alt2021spectral} in two directions: first we reach a much smaller scale close to $\sqrt{\log n/n}$ as predicted by \eqref{fluctuationedge}, 
and second we have a super-polynomial bound $e^{-n^\epsilon}$ for the error of small deviation by taking $t=n^\epsilon$, improving the polynomial bound \cite{alt2021spectral}.

The assumption \ref{twosidedbounds} plays a crucial role in the proof. Adapting the proof argument a bit, \ref{twosidedbounds} could possibly be weakened to $n^{-1-\tau}\leq b_{ij}^2\leq n^{-1+\tau}$ for some sufficiently small $\tau>0$, though we do not pursue this improvement in this paper.

\subsubsection{Suppressed fluctuations: Examples}

Whereas Theorem \ref{Theorem1.6a12} provides high precision small-deviation estimates when $S$ has a good structure, we may have vastly overestimated the fluctuation (and even the spectral radius itself) when the matrix has large chunks of zero dispersed in a particular pattern. This is illustrated by the following important examples that arises in the study of product matrix process and product circular law.

\begin{theorem}\label{examplegrowthbounds}
    For each integer $n$ let $p_n\in[1,n]$ be an integer dividing $n$, with $n=p_nq_n$. Let $X_1,\cdots,X_{p_n}$ be independent $q_n\times q_n$ random matrices with i.i.d. entries having distribution $\frac{1}{\sqrt{q_n}}\xi$, with $\mathbb{E}[\xi]=0$ and $\mathbb{E}[|\xi|^2]=1$ (we do not need $\xi$ to have higher than two moments).
Consider the following $n\times n$ matrix $\mathcal{L}$
\begin{equation}\label{welinearizeit}
\mathcal{L}=\begin{pmatrix}
&X_2&&&\\&&X_3&&\\&&&\ddots&\\&&&&X_{p_n}\\X_1&&&&
\end{pmatrix}.\end{equation}Then for any $q_n\log q_n\ll  n$, for any $\epsilon>0$, with high probability we have $\rho(\mathcal{L})\leq 1+\epsilon$ for $n$ large. We also have a small deviation bound: for any $t>0$,
    $$
\mathbb{P}(\rho(\mathcal{L})\geq 1+\frac{t}{p_n})\leq q_n e^{-2t}.
    $$
\end{theorem}
Note that both claims are true if $p_n$ is so large that $q_n=o(\log n)$: in this case $\sigma_*(\mathcal{L})\gg (\log n)^{-1/2}$, so that Theorem \ref{convergencespectral} and \ref{Theorem1.6a12} are not applicable. But even in this case, we have $\rho(\mathcal{L})\to 1$ and $\rho(\mathcal{L})\leq 1+n^{-1+\epsilon}$ with high probability for any $\epsilon>0$: this scale of fluctuation is much smaller than the Gaussian matrix \eqref{fluctuationedge}. 

The matrix model \eqref{examplegrowthbounds} plays a fundamental role in the study of spectrum and outlier of product random matrices, see for example \cite{o2015products}, \cite{coston2020outliers}, \cite{han2024outliers}, \cite{han2024circular}. For large $p_n\gg\sqrt{n}$ the system exhibits an ergodic behavior, and in particular for $q_n=2$, a law of large numbers is proved for spectral radius of product matrices with only a second moment assumption \cite{MR4270263}. 

\begin{example}(Nilpotent matrices)
    Another important example where  Theorem \ref{convergencespectral} and \ref{Theorem1.6a12} are less precise is when the variance profile $S$ of $A$ is itself a nilpotent matrix. For example we can take $S$ the square matrix with entry one all along its upper diagonal and zero elsewhere. Then $A$ itself is nilpotent and $\rho(A)=0:$ this can be checked via computing high moments of $A$, even if Theorem \ref{convergencespectral} and \ref{Theorem1.6a12} are not designed to detect such degenerate circumstances. 
\end{example}

It is hard to formulate general conditions capturing pedagogically defined matrices with large chunks of zero, and the focus of the current paper is more leaning on inhomogeneous non-Hermitian matrices having a near mean-field structure. Thus we do not develop this direction further and illustrate suppressed fluctuations just via the example in Theorem \ref{examplegrowthbounds}.

\subsection{Large deviation estimates of spectral radius}

An important result in high dimensional probability states that the operator norm of a random matrix with sub-Gaussian entries concentrates around a fixed value with sub-Gaussian tails of deviation. More precisely, (see \cite{vershynin2018high}, Theorem 4.3.5) Let $A$ be an $m\times n$ random matrix with independent mean zero sub-Gaussian entries, then for any $t>0$,
$$ \|A\|\leq CK(\sqrt{m}+\sqrt{n}+t)
$$ with probability at least $1-2\exp(-t^2)$. Much more refined sub-Gaussian concentration inequalities have since been derived, see \cite{van2017spectral},\cite{bandeira2023matrix} and references therein. When $A$ has Gaussian entries, the operator norm $\|A\|$ is a Gaussian process where powerful dimension free techniques are available, see for example \cite{bandeira2016sharp}, \cite{bandeira2023matrix}.

In contrast, little seems to be known for sub-Gaussian concentration of $\rho(A)$, the spectral radius of an i.i.d. random matrix. Observe that using $\rho(A)\leq \|A\|$ and that large deviation inequalities for $\|A\|$ are well-known, the interest in proving large deviation of $\rho(A)$ mainly lies in the range $\rho(A)\leq t\leq \|A\|$. The spectral radius does not admit a variational characterization, hence the standard arguments for the operator norm are not applicable. Via computing very high moments as in the proof of Theorem \ref{convergencespectral}, we can derive large deviation estimates for $\rho(A)$, but which may be suboptimal in the power of $\sigma_*$. We propose an alternative method in the special case when $X$ has a doubly stochastic variance profile and has Gaussian entries, with optimal dependence on $\sigma_*$.

\begin{theorem}\label{theorem1.5subgaussianconcentration}(Spectral radius sub-Gaussian concentration) 
   Let $A=(a_{ij})$ be an $n\times n$ random matrix with independent mean zero entries, the variance profile $S=(s_{ij})=(\mathbb{E}[|a_{ij}|^2])_{i,j}$ is doubly stochastic . Assume that $A$ has jointly Gaussian entries and $\sqrt{\sigma_*}(\log n)^\frac{3}{4}\to 0$ as $n\to\infty$, then for any $t>0$, whenever $\sqrt{\sigma_*}(\log n)^{\frac{3}{4}}\leq \frac{t^{1.5}}{500}$ we have
    $$
\mathbb{P}(\rho(A)\geq 1+t)\leq  (1+t^{-3}) e^{-C_0\frac{\min\{t^2,t^3\}}{\sigma_*^2}}
    $$ where $C_0$ is some universal constant.
\end{theorem}

At this point, the reader may be bothered as to why we need Theorem \ref{theorem1.5subgaussianconcentration} after all. For the Gaussian case, we may expect the following simple argument to work: we have the computation $\mathbb{E}\rho(A)$ in Theorem \ref{convergencespectral} and since the entries are Gaussian we expect $\rho(A)$ to concentrate around $\mathbb{E}\rho(A)$. This argument works pretty well when one controls the operator norm $\|A\|$ (see \cite{bandeira2016sharp}) but does not seem to work for the spectral radius $\rho(A)$ because, in general, the spectral radius $\rho(A)$ is not a Lipschitz continuous function of all the entries of $A$, unlike the operator norm. Thus, one cannot directly apply Gaussian concentration of measure to $\rho(A)$. Instead, we will consider the resolvent of the Hermitization of $A$ and use a concentration of measure argument for the resolvent to deduce the result. 

Compared to the Hermitian case (see for example  \cite{brailovskaya2024extremal}, \cite{bandeira2016sharp}), we have reached the expected (sharp) dependence of the large deviation on $\sigma_*^{-2}$. To see why $\exp(-t^2/\sigma_*^2)$ is expected to the asymptotics for the large deviation probability, suppose (only as a heuristic argument) that we can regard $\rho(A)$ as a 1-Lipschitz continuous function of each of its entries, then the fact that the entries are scaled by a factor $\sigma_*$ would imply a Gaussian large deviation probability $\exp(-t^2/\sigma_*^2)$. However, we do not have a proof that this large deviation probability $\exp(-t^2/\sigma_*^2)$ is the sharpest one for all models. Compared to this optimal estimate, we need a higher power $t^3$ instead of $t^2$ for $t$ small, and we have an additional entropy cost factor $(1+t^{-3})$ contributing for some $t$. These additional factors all arise from the proof technique and may be removed via other methods, but these additional factors are tame enough for large $n$ and do not contribute significantly to practical applications.

The extension of Theorem \ref{theorem1.5subgaussianconcentration} to general subGaussian distributions should be possible, but the current method yields suboptimal dependence on $t$ and $\sigma_*$ in the resulting estimate, so we do not pursue it here. The restriction of a doubly stochastic variance profile is made so that the matrix Dyson equation can easily be solved. For an arbitrary variance profile, the matrix Dyson equation \eqref{matrixdysomequationsfag} may not be explicitly solvable, which makes quantitative results much harder to get.

\subsection{Upper bound in the heavy-tailed case} In the main theorems of this paper we have considered random matrices with sub-Gaussian, or sub-exponential entries. In this last section we show the theme of this work continues to hold in the heavy-tailed case, when entries of $A$ have only $2+\epsilon$ moments. The fact that $\rho(A)$ is bounded without fourth moment assumption is a peculiar feature for non-Hermitian matrices uncovered in \cite{WOS:000435416700013}, \cite{bordenave2021convergence}. So far this fact is proven only for the i.i.d. case, and in the next theorem we show the parameters $\sigma$ and $\rho(S)$ can still be a good upper bound for $\rho(A)$ in this case without a fourth moment.

\begin{theorem}\label{upperboundsecondmoment}
Let $A=(b_{ij}x_{ij})$ be an $n\times n$ random matrix where $(x_{ij})$ are i.i.d. with distribution $\xi$ such that $\xi$ is symmetric $\xi\overset{\text{law}}{\equiv}-\xi$, $\mathbb{E}[\xi]=0,\mathbb{E}[|\xi|^2]=1$ and $\mathbb{E}[|\xi|^{2+\epsilon}]<\infty$ for a given $\epsilon>0$. Denote by $S=(s_{ij})=(b_{ij}^2)$. Assume that the entries of $S$ are uniformly upper bounded: we can find a constant $C>0$ such that 
\begin{equation} \label{howdoesupperbounds}s_{ij}\leq\frac{C}{n}\text{ for each }i,j\in[n].
\end{equation}

Then the spectral radius $\rho(A)$ can be well-controlled in the following cases
\begin{enumerate}
    \item (Row/ column sum) Let $\sigma$ be defined as in \eqref{maximalrowcolumn}.
Then with probability $1-o(1)$,
$$\limsup_{n\to\infty}\rho(A)\leq \sigma.$$

    \item (Spectral radius) Assume moreover that $S$ has a lower bound: for some fixed $c>0$,
    $$
s_{ij}\geq \frac{c}{n}\text{ for each } i,j\in[n].
    $$ Then for any $\epsilon>0$, we have with probability $1-o(1)$,
   $$\rho(A)\leq \sqrt{\rho(S)}+\epsilon.$$
\end{enumerate}
\end{theorem}
The proof will essentially be an adaptation of the moment method in \cite{WOS:000435416700013}. In case (1), it is not clear if we can replace $\sigma$ by the long-time control bound in Definition \eqref{longtimeshorttime}. In case (2), we should be able to prove a two-sided inequality, that is, we also have $\rho(A)\geq\sqrt{\rho(s)}-\epsilon$, via showing an inhomogeneous circular law holds for the ESD of A (see \cite{MR3770875}, \cite{MR3857860}, \cite{cook1article}) and that the limiting law is supported on $B(0,\sqrt{\rho(S)})$. The proof of this step is standard by now (only the $2+\epsilon$ moment condition needs to be taken care of) and is omitted.

The uniform upper bound \eqref{howdoesupperbounds} plays a fundamental role and cannot be removed. As entries of $X$ only have $2+\epsilon$ finite moments, if \eqref{howdoesupperbounds} were replaced by a weaker assumption, we can easily construct a band matrix in which the diagonal entries are with high probability unbounded, while the rest of the matrix (with diagonal entries removed) has a bounded spectral radius. This is likely to produce a matrix with unbounded spectral radius; for more details of this band matrix example, see Example \ref{examplebandmatrix}. That is, a weakening of \eqref{howdoesupperbounds} should be followed by a strengthening of the moment assumption on $\xi$, but proving such am interpolation result is beyond the scope of this work.

\subsection{Methods and further discussions}

To prove most theorems in this paper (except Theorem \ref{theorem1.5subgaussianconcentration}) we use the method of moments to bound the spectral radius of $A$, and this idea dates back to the earliest works \cite{Geman1986THESR}, \cite{Bai1986LimitingBO} and the more recent ones
\cite{WOS:000435416700013}, \cite{benaych2020spectral}. The basic idea is very easy to illustrate: if $\lambda$ is an eigenvalue of $A$ then $\lambda^p$ is an eigenvalue of $A^p$ for any $p$, so that we have $|\lambda|\leq \|A^p\|^\frac{1}{p}$ and the problem boils down to computing upper bounds of $\mathbb{E}[\|A^p\|]^\frac{1}{p}$. The accuracy of estimates on $\rho(A)$ depends on the range on which $p$ can simultaneously grow with $n$. Compared to previous works, the key novelty of this paper is to improve the range of $p$ where such a bound can be proven.

The idea of using high moment computations and a compression argument to bound the operator norm of an inhomogeneous random matrix dates back to \cite{erdHos2011quantum}, \cite{bandeira2016sharp}, and we adapt them to this non-Hermitian setting. In contrast to the situation there, the high moment computations $\mathbb{E}[\|A^p\|]$ for a Gaussian model $A$ does not seem available, and this will occupy the bulk of Section \ref{section2} of this paper.

For Theorem \ref{convergencespectral}, we need to consider any $p=O(n)$ in computing $\mathbb{E}[\|A^p\|]^\frac{1}{p}$. This is in general very difficult, and we will take ideas from computing non-backtracking paths and precise topological enumeration of all possible paths, which is first done in \cite{feldheim2010universality} for Wigner and Wishart matrices. The independent assumption on $A$ essentially guarantees the paths are non-backtracking, so many computations in \cite{feldheim2010universality} can be used here.

For Theorem \ref{Theorem1.6a12}, we need to consider any $p\ll \sqrt{n}$ in bounding $\mathbb{E}[\|A^p]\|$. We do not consider any larger scale of $p$ here because typically $\rho(A)\geq \sigma+n^{-1/2}$ by \eqref{fluctuationedge}. The upper bound of $\mathbb{E}[\|A\|^p]$ is reminiscent of the Dyck path enumerations in proving edge universality for Wigner matrices in \cite{MR1620151}, \cite{MR1647832}, \cite{soshnikov1999universality}, and we will take inspirations for these works to prove Theorem \ref{Theorem1.6a12}. In this range $p\ll\sqrt{n}$, the main contribution to $\mathbb{E}\|A\|^p$ comes from double edges without self-intersection, and this is why only the long-time control in Definition \ref{longtimeshorttime} contributes to the first order limit of $\rho(A)$.

Theorem \ref{theorem1.5subgaussianconcentration} will follow from using the machinery of \cite{bandeira2023matrix} and a covering argument, leading to an entropy cost $(1+t^{-3})$.
Theorem \ref{upperboundsecondmoment} will follow from adapting the high moment computations from \cite{WOS:000435416700013}.

The high moment method in this paper requires $\mathbb{E}[A]=0$. Recently, \cite{campbell2024spectral} proposed a method to prove absence of outliers for $A$ having nonzero expectation and a general variance profile satisfying \ref{twosidedbounds}.

\section{Moments expansion in the homogeneous case}\label{section2}

To bound the spectral radius we will develop moments expansion of $\mathbb{E}[\operatorname{tr}(A^\ell (A^*)^\ell)]$ for moderately large $\ell$ up to $\ell=o(\sqrt{n})$ and $\ell=o(n)$. Unlike Wigner and Wishart matrices, such estimates cannot easily be found in the literature even in the special case where $A$ is the Ginibre ensemble. In \cite{Geman1986THESR} a computation was done up to $\ell=n^c$ for some fixed small $c>0$ and in  \cite{benaych2020spectral} a computation up to $\ell=n^{\frac{1}{3}-\epsilon}$, but they are not enough for our purpose.

Therefore we have to do moment computation by ourselves even in the Gaussian case. We will present in this section the computations when the variance profile is homogeneous, and then adapt the computations to inhomogeneous variance profiles in Section \ref{section3}. The computations do not use the Gaussian structure and works for general distributions. 

At this point we may wonder if, using Wick formula, we can significantly simplify the computation of $\mathbb{E}[\operatorname{tr}(A^\ell(A^*)^\ell)]$ for the Gaussian case. Relevant computations for the mixed moments of complex Ginibre ensembles can be found in references in \cite{byun2022progress}. In particular, it is well-known that certain mixed moments of the complex Ginibre ensemble converge to the Fuss-Catalan numbers \cite{halmagyi2020mixed} \cite{kemp2011enumeration}. However, these results are computed in the regime when the matrix size $n$ grows to infinity but $\ell$ is fixed. When $\ell$ is growing with $n$, it is not clear if the use of Wick formula leads to any simplification in the combinatorial argument: we have to enumerate all the pairings of vertices of a chain of length $2\ell$ where two different pairings share the same vertex, and it is unclear how to correctly enumerate such pairings especially near the critical case $\ell=n^{1/2-\epsilon}$.

This section is purely technical and can be skipped for a first reading. In Section \ref{section3}, references will be made on the results and notations stated in this section.

We will carry out two different types of moment computations. In Section \ref{momentestimate1}, which is the preparation for proving Theorem \ref{Theorem1.6a12} and \ref{spectralradiuslargedeviation}, we will compute moments up to $\ell=o(n^\frac{1}{2})$ to derive small deviation estimates. The choice of $\ell$ is because the spectral edge of a Ginibre matrix satisfies the three terms expansion \eqref{fluctuationedge}. It turns out that in the regime $\ell=o(n^\frac{1}{2})$ the combinatorics in the trace expansion $\mathbb{E}[\operatorname{tr}(A^\ell (A^*)^\ell)]$ is relatively simple, and the only dominant contribution comes from simple paths with no self-intersection. We will partly follow and make considerable adaptations to the moment computation in \cite{MR1620151}  for Wigner matrices up to $\ell=o(n^\frac{1}{2})$.

In Section \ref{momentestimate2}, which is the preparation for proving Theorem \ref{convergencespectral}, we compute moments up to $\ell=O(n)$. Motivated by \cite{feldheim2010universality}, the combinatorics involved is much more difficult but still tractable by the graph counting techniques.

\subsection{Moment estimate I}\label{momentestimate1}

To highlight the main ideas we first do computation for the complex Ginibre case in Lemma \ref{lemma2.111}, then the sub-exponential symmetric case in Theorem \ref{momenttheorem2.2}, and finally remove the symmetric distribution assumption in Theorem \ref{nosymmetricdistribution}. The underlying theme is the same despite increased generality.

\begin{lemma}\label{lemma2.111}
    Assume that $A$ is the complex Ginibre ensemble, i.e. $V=\frac{1}{n}\mathbf{1}\mathbf{1}^t$ and entries of $X$ are i.i.d. with law $\mathcal{N}_\mathbb{C}(0,1)$. Then whenever $p\ll \sqrt{n}$, we have 
    $$
\mathbb{E}\operatorname{Tr}[A^p(A^*)^p]=n(1+o(1)).
    $$
\end{lemma}

\begin{proof}
    The property we use about the distribution $\xi\sim \mathcal{N}_\mathbb{C}(0,1)$ is its rotational invariance, so that for any $m,n\in\mathbb{N}_+$, $\mathbb{E}[\xi^m\overline{\xi}^n]=0$ unless $m=n$.

\textbf{Breakdown the expectation computation: classifying vertices by the number of visits by the first half.}
We have the expansion 
$$
\mathbb{E}\operatorname{tr}(A^p(A^*)^p)=\frac{1}{n^p}\sum_{i_0,\cdots,i_{p}=1}^n \sum_{j_0,\cdots,j_{p-2}=1}^n \mathbb{E}[x_{i_0i_1}x_{i_1i_2}\cdots x_{i_{p-1}i_{p}}\bar{x}_{j_{p-2}i_p}\bar{x}_{j_{p-3}j_{p-2}}\cdots\bar{x}_{i_0j_0}]. 
$$  From the above property of complex Gaussian distribution, we see that once $(i_0,\cdots,i_p)$ are fixed, then $(j_0,\cdots,j_{p-2})$ must be chosen from $(i_0,\cdots,i_p)$ and no other possibility is allowed. This is because if some additional index is chosen then we have a term $\bar{x}_{i_sj_s}^{m_s}$ in the expectation with no corresponding pairs of $x_{i_sj_s}$, so the expectation must vanish. Likewise, each directed edge $(i_0,i_1),\cdots,(i_{p-1},{i_p})$
counting multiplicity must be traveled by the $\bar{x}$ elements in exactly the same number of times as they are traveled by the $x$ elements.

Therefore we only need to enumerate over the first half of the paths, and we later show how the contributions from the second half of paths can be bounded from the first half. We first fix the initial $i_0\in[n]$ and then consider paths $P=(i_1,\cdots,i_p)$  and we let $n_k$ be the number of indices in $[n]$ that are visited $k$ times by $P$. Clearly we have $\sum_{k\geq 0}n_k=n$ and $\sum_{k\geq 0}kn_k=p$ (note that $P$ has $p$ indices, and each vertex visited $k$ times by $P$ should appear in $i_1,\cdots,i_p$ with multiplicity $k$, hence the equality). We let $\mathcal{P}_k\subset[n]$ denote the set of vertices visited $k$ times by $P$, for each $k$.

\textbf{Contribution from the dominant part.}
We claim that the dominant contribution comes from paths where $n_1=p$ and $n_k=0$ for all $k\geq 2$. Once such a path is chosen, so that $P$ has no self-intersection, there is a unique choice of $j_{p-2},\cdots,j_{1}$ so as to have a nonzero expectation: we must take $j_{p-2}=i_{p-1},\cdots,j_0=i_1$. The contribution of all such terms to $\mathbb{E}\operatorname{tr}(A^p(A^*)^p)$ is at most
\begin{equation}\label{dominantsums}
\frac{1}{n^p}\cdot n\cdot (n-1)\cdots(n-p)= n(1+o(1)),
\end{equation} where we simply need to choose $p$ different indices from $[n]$, in addition to the initial vertex $i_0$, and we also use $p\ll\sqrt{n}$.

\textbf{Other contributions: Determining the second half of path from the first half}. Then we show contributions from terms with nonzero $n_k,k\geq 2$ is negligible. The procedure is to first choose the initial $i_0$ and then choose the path $P$. Once $P$ is chosen proceed as follows: we set up a procedure which, for every $1\leq s\leq p-2$ suppose that $j_{s}$ is chosen, outputs a choice of $j_{s-1}$. We first show how we determine $j_{p-2}$ from $i_p$: if $i_p$ is visited only once in $P$ by $x_{i_{p-1}i_p}$ then we must take $j_{p-2}=i_{p-1}$ with no other option, and we have a $|x_{i_{p-1}i_p}|^2$ term in the expectation. If $i_p$ is visited more than once in $P$, say it is visited $k$ times, then we choose any of these $k$ visits and set $j_{p-2}$ to be the vertex in $P$ immediately to the left of this visit (the restriction of complex Ginibre ensembles forces that $j_{p-2}$ must be exactly one of those $k$ preimages to guarantee a nonzero expectation). In this scenario $|x_{i_{p-1}j_{p-2}}|$ may appear at most $2k$ times in the expansion of $\mathbb{E}[A^p(A^*)^p]$.
This determines the choice of $j_{p-2}$, and from this we can determine the choice of $j_{p-3}$ depending on the number of times $j_{p-2}$ appears in $P$, and then determine $j_{p-4}$, all the way up until we determine $j_0$. We are faced with the global constraint that each directed edge in $P$ must be traveled the same number of times by the elements with complex conjugate, the $\bar{x}_{ij}$'s.

In the above procedure, the determination of $j_{p-s-1}$ from $j_{p-s}$ is not unique, but if $j_{p-s}$ is visited by $P$ for $k$ times, then there is at most $k$ possible values for $j_{p-s-1}$. Thus for a fixed first half of path $P$, the total number of choices of the second half is at most $\prod_{k=2}^p k^{kn_k}$. We prove this upper bound as follows: 

\textbf{Proving the upper bound for the number of choices for the second half of the path.}
To justify the upper bound in the last paragraph, we can carry out an inductive proof: let $P'$ denote the second half of the path, and suppose we have known the first $x\leq p$ vertices of $P'$, so we know $P\cup (P')_{[1,x]}$ where the latter is the restriction of $P'$ on the first $x$ vertices. By the constraint of complex Gaussians, we have determined the exact vertices that $P'_{[x+1,p]}$ should visit, and the exact number of times each vertex should be visited (but have not determined the order where they are visited).

Now we denote, for each $k$, $n_k'$ (and  $n_k''$) be the number of indices of $(P')_{[1,x]}$ (and respectively  $(P')_{[x+1,p]}$) that lie in the subset $\mathcal{P}_k$, counting the multiplicity of each appearance. Note that by the previous paragraph, $n_k,n_k''$ are uniquely determined by $(P')_{[1,x]}$ and not by the choice of $(P')_{[x+1,p]}$. Then we have $kn_k=n_k'+n_k''$: be careful that in the definition of $n_k$ we only count indices in $[n]$ without the number of visits, but in defining $n_k'$ and $n_k''$ we count each single visit, and thus there is an overall multiplicity $k$ for each index in $\mathcal{P}_k$. For the known path $(P')_{[1,x]}$ we claim the number of possible choices for $(P')_{[x+1,p]}$ is at most $\prod_{k=2}^pk^{n_k''}$. This claim is clearly true for $x=p-2$, and when the claim is true for $x$, it is not hard to prove it is true for $x-1$ by determining the possible next vertex once $P'_{x-1}$ is chosen, see the next paragraph:

Indeed, if we use $n_k'(x)$ and $n_k''(x)$ to denote the numbers $n_k',n_k''$ with respect to $P'_{[1,x]}$, then assume that we are given $(P')_{[1,x-1]}$, and the last vertex $P'_{x-1}\in\mathcal{P}_{k_0}$ for some $k_0$, then there are at most $k_0$ choices for the vertex $P'_{x}$, and, for each \textbf{admissible} choice of $P_x'$ (to be explained later), we must have $n_l''(x)=n_l''(x-1)$ for all $l\neq k_0$ and $n_{k_0}''(x)=n_{k_0}''(x-1)-1$. Then by indiction hypothesis, the contribution of $(P')_{[x+1,p]}$ for each fixed admissible $P_x'$ is at most $\prod_{k=2}^p k^{n_k''(x-1)-\mathbf{1}_{k=k_0}}$. Multiplying by $k_0$ verifies the indiction hypothesis for $x-1$. Then we inductively prove this upper bound for the whole path $P'$, since when $x=0$ we have $n_k''=kn_k$.

Finally, we explain what \textbf{admissible} means in the previous paragraph. Although for a given $P'_{x-1}$ we may have at most $k_0$ possible choices of $P_{x}'$, not all of them will make a non-vanishing contribution to the contribution: we have to make sure the constraint on complex Ginibre matrices is satisfied, so that the edges of $P'$ should be an exact repetition of the edges of $P$. As such, we only consider those $P_x'$ such that the remaining paths can be constructed to satisfy this condition, and thus we can check that $n_l''(x)=n_l(x)-1_{l=k_0}$ for each such choice of $P_x'$.

\textbf{Combining all the upper bounds.} For each choice of $n_1,n_2,n_3,\cdots n_p$ the total contribution to $\mathbb{E}[\operatorname{tr}(A^p(A^*)^p]$ by subpaths with free index $i_0$ and first half part $P$ has index $n_0,\cdots,n_p$ is bounded by
\begin{equation}\label{contributiongeneral}
n\frac{1}{n^p}\frac{n!}{n_0!n_1!\cdots n_p!}\frac{p!}{\prod_{k=2}^p(k!)^{n_k}}\prod_{k=2}^pk^{kn_k}\prod_{k=2}^p(\text{const}\cdot k)^{kn_k}.
\end{equation} The construction here is highly reminiscent of \cite{MR1620151} in the Wigner case.
The leading factor $n$ is the choice of $i_0\in[n]$.
The factor $\frac{n!}{(n_0!\cdots n_p!)}$ arises from dividing $n$ integers into subsets of sizes $n_0,\cdots,n_p$. The factor $\frac{p!}{\prod_{k=2}^p(k!)^{n_k}}$ is the number of ways to write down all such paths $P$ of length $p$ with fixed $n_0,\cdots,n_p$ and fixed vertex groups that appear in $P$: indeed we have $p!/\prod_{k=1}^p(kn_k)!$ ways to divide the path length $p$ into subsets of size $kn_k$, $k=1,\cdots,p$, and we have $(kn_k)!/(k!)^{n_k}$ ways to write down entries in each class $P_k$. The factor $k^{kn_k}$ is explained in the above paragraph. Finally, the last exponential factor $(\text{const}\cdot k)^{kn_k}$ bounds the subGaussian moments from paths visited multiple times, by assumption \eqref{gaussianmomentsrealgaussian}. Indeed the estimate \eqref{contributiongeneral} is very loose and involves a lot of overcounting, but it is good enough in the sub-Gaussian case. We will refine them in Theorem \ref{momenttheorem2.2} for the sub-exponential case.

\textbf{Upper bound estimate for the summation.}
Using $\frac{n!}{n^pn_0!}\leq \frac{n(n-1)\cdots(n-n_0+1)}{n^p}\leq n^{n-n_0-p}=n^{-\sum_{k=2}^p(k-1)n_k}$, using also $p!\leq (n_1)!p^{p-n_1}$ and $p-n_1=\sum_{k\geq 2}kn_k$, we can bound the right hand side of \eqref{contributiongeneral} by 

$$\begin{aligned}
&n\cdot n^{-\sum_{k\geq 2}(k-1)n_k} \frac{1}{n_2!\cdots n_p!}\frac{p^{\sum_{k\geq 2}kn_k}}{\prod_{k=2}^s(ke^{-1})^{kn_k}}\prod_{k=2}^p(2\text{const}\cdot k^2)^{kn_k}
\\&\quad \leq n\prod_{k=2}^p\frac{1}{n_k!} \left[\frac{(2e\cdot\text{const}\cdot k\cdot p)^k}{n^{k-1}}\right]^{n_k}
.\end{aligned}$$
Taking the summation over all $0<\sum_{k\geq 2}kn_k\leq p$, we see that the total sum is bounded by 
$$
n\left[\exp\left(\sum_{k=2}^p\frac{(2e\cdot\text{const}\cdot k\cdot p)^k}{n^{k-1}}\right)-1\right].
$$
    For any $p\ll n^\frac{1}{2}$ we have $$\begin{aligned}\sum_{k=2}^p\frac{(kp)^k}{n^{k-1}}&=\sum_{k=0}^{p+2}\frac{((k+2)p)^{k+2}}{n^{k+1}}\leq  \frac{p^2}{n}(4+\sum_{k= 1}^{p+2}\frac{(k+2)^2((k+2)p)^{k}}{n^k})\\& \leq \frac{p^2}{n}(4+\sum_{k= 1}^{p+2}(\frac{\text{const}\cdot p^2}{n})^k)=\frac{p^2}{n}\cdot(3+\frac{1}{1-\text{const}\cdot p^2/n})=o(1),\end{aligned},$$ where for the first inequality in the second line we use $k\leq p+2$ and $\log_k k+2\leq 2$. Thus we verified the contribution from summations with $\sum_{k\geq 2}kn_k\geq 1$ is negligible.

    The last computation highlights why $p\sim \sqrt{n}$ is the threshold value for the proof.
\end{proof}

Exactly the same limit holds in the sub-exponential case with a symmetric law.

\begin{theorem}\label{momenttheorem2.2}
    Let $X$ satisfy the assumptions in Definition \ref{generalmatrixmodel}, and we assume further that the distribution of $x_{ij}$ is symmetric. We assume $V=\frac{1}{n}\mathbf{1}\mathbf{1}^t$. Then whenever $p\ll \sqrt{n}$, 
    $$
\mathbb{E}\left[\operatorname{Tr}(A^p(A^*)^p)\right]=n(1+o(1)).
    $$ We can also find constants $C_0,C_1>0$ depending only on the constant in the sub-exponential moments $x_{ij}$ \eqref{gaussianmoments} such that whenever $p^2\leq C_0n$, we have $$\mathbb{E}\left[\operatorname{Tr}(A^p(A^*)^p)\right]\leq C_1n.
    $$
\end{theorem}

\begin{proof}
    We follow most parts of the proof of Lemma \ref{lemma2.111}, but now $x_{ij}$ are not rotationally symmetric. We do still assume $x_{ij}$ has a symmetric distribution, so that 
\begin{equation}\label{evenproducts}  \mathbb{E}[x_{ij}^m\overline{x}_{ij}^n]\neq 0
 \text{  only if  } m+n  \text{ is even}.\end{equation}
 Consider a typical term (we use a different parametrization from Lemma \ref{lemma2.111}) \begin{equation}\label{typicalproducts}x_{i_0i_1}x_{i_1i_2}\cdots x_{i_{p-1}i_p}\bar{x}_{i_{p+1}i_p}\cdots\bar{x}_{i_0i_{2p-1}}\end{equation} in the expansion $\operatorname{Tr}(A^p(A^*)^p)$ with free indices $i_0,\cdots,i_{2p-1}$ to be summed over $[n]^{2p}$. 

In this proof, we call the first half of the path associated to the product in \eqref{typicalproducts} to be the subset with indices $i_0,i_1,\cdots,i_p$ and we call the second half of the path to be the subset with indices $i_{p+1},\cdots,i_{2p-1},i_0$.

 \textbf{Classifying each visit by the parity of visit times.} We make labeling construction similar to the first step of the proof of Lemma \ref{lemma2.111} (but instead of labeling only the first half of the visit, we label the visit of the whole path with a counting mod 2). 
 We say that the index $i_t$, $1\leq t\leq p$ is an odd index if there are an even number of terms in \eqref{typicalproducts} to the left of $x_{i_{t-1}i_t}$ that equals $x_{i_{t-1}i_t}$. We say an index $i_t$, $p+1\leq t\leq 2p-1$ is an odd index if there are an even number of terms to the left of $\bar{x}_{i_{t}i_{t-1}}$ that equals $x_{i_{t}i_{t-1}}$ or $\bar{x}_{i_{t}i_{t-1}}$. Here we use $x$ and $\bar{x}$ as abstract variables rather than the value they may take.
 This definition is stated to count the number of times a directed edge $i_{t-1}\mapsto i_t$ appears in the generated path \eqref{typicalproducts} before and until instance $i_t$.

 Likewise, we say $i_t$, $1\leq t\leq p$ (resp. $p+1\leq t\leq 2p-1$) is an even index if an odd number of terms to the left of $x_{i_{t-1}i_t}$ (resp. $\bar{x}_{i_{t}i_{t-1}}$)  equals $x_{i_{t-1}i_t}$ (resp. equals $x_{i_{t}i_{t-1}}$ or $\bar{x}_{i_{t}i_{t-1}}$).  Evidently, if $i_1,\cdots,i_t$ are mutually distinct up to time $t$, then $i_t$ is an odd element. 

 Finally, we say the last instance $i_0$ in the path expansion \eqref{typicalproducts} is odd if to the left of the term $\bar{x}_{i_0i_{2p-1}}$ there are an even number of terms that equal
 $\bar{x}_{i_0i_{2p-1}}$ or $x_{i_0i_{2p-1}}$, and we say $i_0$ is even otherwise. 
 
 This definition of odd and even index is very similar to that in \cite{MR1620151}, Definition 1 and 2 but with the crucial difference that we consider directed edges (as the entries are i.i.d.) whereas \cite{MR1620151} considers undirected edge: this leads to an overall difference of a factor of 2.

 For this term \eqref{typicalproducts} to have nonzero expectation, each term $x_{i_si_{s+1}}$ must be coupled with $x_{i_si_{s+1}}$ or its complex conjugate, and the total number that $x_{i_si_{s+1}}$ and its conjugate appears must be even, due to \eqref{evenproducts}. Thus the directed path generated by \eqref{typicalproducts} must be an even closed path, and the last appearance of $i_0$ in \eqref{typicalproducts}  must be an even index. Altogether, there are $p$ odd indices and $p$ even indices among $i_0,\cdots,i_{2p-1}$ (we do not consider the first appearance of $i_0$ in \eqref{typicalproducts}). We can reconstruct an even closed path as follows.

\textbf{Step 1: Assignment of vertices of the path with the correct multiplicity (parts of the path with odd visit times)}. We choose $p$ indices from $[n]$(with repetition) such that $n_1$ of them appears once (we denote this set by $\mathcal{P}_1$), $n_2$ of them appears twice (denoted by $\mathcal{P}_2$), $n_k$ of them appears $k$ times (denoted by $\mathcal{P}_k$), etc. Then $\sum_{i=0}^p n_i=n$ and $\sum_{i=1}^p in_i=p$. Then we give an ordering $\mathcal{L}$ of $\cup_{i\geq 1}i\mathcal{P}_i$ to decide the order they appear in $[p]$, where we use the notation $i\mathcal{P}_i$ to denote the multiset in which every element of $\mathcal{P}_i$ appears $i$ times. These $p$ vertices are candidates for the odd instances in the path \eqref{typicalproducts} , with overall combinatorial factor
$$n\frac{n!}{n_0!n_1!\cdots n_p!}\frac{p!}{\prod_{k=2}^p(k!)^{n_k}}.$$

\textbf{Step 2: inserting even paths to the already constructed odd paths I: the first half}. In the previous step, we have essentially constructed half of the path, having the edges that are visited an odd number of times. The next step is to insert new edges in the path that correspond to visits a number of times. We introduce an algorithm to insert these edges of the visit at even number of times. Before proceeding, we highlight how this proof mirrors the proof of Lemma \ref{lemma2.111}: here we introduce subsets $\mathcal{P}_1,\mathcal{P}_2,\cdots$: These subsets play a similar role as the subsets introduced in the proof of Lemma \ref{lemma2.111} with the same notation. The remaining steps are to insert the edges that are visited at the even number of times to the already constructed part, and to show that this leads to an overall combinatorial factor of at most $\prod_{k\geq 2} (\text{const}\cdot k)^{kn_k}$. We have to design a much more involved algorithm to insert in the edges visited at the even number instances and verify this combinatorial upper bound: this algorithm is much more sophisticated than the one in Lemma \ref{lemma2.111} and is presented as follows.

By Step 1, we first select $i_1,\cdots,i_p$ to be the vertices selected in Step 1 in the order specified in Step 1, and we assign that all the index $i_1,\cdots,i_p$ are odd index. Then we start a scanning and insertion procedure as follows. We start to scan from the second index $i_1$ which must be odd and we do nothing with it. (a) Suppose we have scanned the path up to index $i_t$ $(1\leq t<p)$ and that $i_t\in\mathcal{P}_1$ is an odd index, then $i_{t+1}$ must also be an odd index as $x_{i_ti_{t+1}}$ hasn't appeared before to its left. (b) On the other hand, if $i_t\in\mathcal{P}_k$ is an odd index for some $k\geq 2$, then in the finally constructed path \eqref{typicalproducts} there is a freedom for $i_{t+1}$ to be (b1) an odd index (and we take no action) or (b2) be an even vertex (we will insert new elements between $i_t$ and the present $i_{t+1}$ so that after insertion, $x_{i_ti_{t+1}}$ repeats an element appearing before it). In other words, we have pre-requested that $i_{t+1}$ is an odd index in the initial construction, but it can now freely change its odd/even nature via a further insertion process.

In the latter case (b2), we insert vertices between $i_t$ and $i_{t+1}$. There are no more than $2k$ choices of previously existing terms to repeat, hence there are at most $2k$ possible choices for the next even index $i_{t+1}$.  Once a new vertex from these $2k$ choices is selected and inserted with label
$i_{t+1}$ in the path (after insertion we relabel the subscript of previous indices $i_{t+k}$ by $i_{t+k+1}$ for all $k\geq 1$), (I) if $i_{t+1}\in\mathcal{P}_1$ then the possible choice of $i_{t+2}$ is unique: it must be the same index as the index that appears after the first appearance of $i_{t+1}$ (because this edge has to be traversed at least twice; if $i_{t+2}$ is chosen differently then we have to return to this edge later and thus arrive at $i_{t+1}$ for the third time, contradicting $i_{t+1}\in\mathcal{P}_1$). Thus we have a unique choice of $i_{t+2},\cdots$ until we reach an index in $\cup_{i\geq 2}\mathcal{P}_i$ and the last step is via an edge used an even number of times. Then we move on to the next un-scanned index in the path, which has an odd index, and we restart the whole procedure of Step 2 to scan this new vertex (if this point is in $\mathcal{P}_1$ then do nothing, otherwise decide whether this index is odd or even and proceed accordingly). (II) If $i_{t+1}\in\cup_{k\geq 2}\mathcal{P}_k$, then move on to the unread index in the path (this is an odd index, and recall that we have fixed $p$ vertices and its order, already in Step I) and restart the procedure in Step 2 afresh and scan $i_{t+1}$. The construction proceeds until we have scanned through the whole path which now has length at least $p+1$, and we have possibly inserted some new vertices to it.

A graphical understanding of Step 2 would be more helpful: case (a) correspond to the case where $i_t$ is not a point of self-intersection, and case (b) corresponds to a possible instance of self-intersection. In case (b1) we simply continue without repeating previously used paths, while in case (b2) we switch to previously used paths and form repeated edges. In case (b2), once we decide to insert repeated edges, there is a unique choice of continuation until we reach a point in $\cup_{k\geq 2}\mathcal{P}_k$, where we stop the insertion and move on to the next un-scanned index. After the whole procedure, we may have inserted some new vertices in the originally length $p+1$ path, and we have already scanned through the path up to time $p$ (we may have scanned the path to a time larger than $p$).

The geometric meaning of Step 2 is to construct all the repeated cycles that both lie in the first half, or the second half, part of the path. This will be useful in Step 3. As a consequence of this observation, in Step 2 we have fixed all the odd indices, and what remains in Step 3 is to fill in the even indices. This task will be settled in Step 3.

\textbf{Step 3: inserting even paths to the already constructed odd paths II: the second half.} If $i_p\in\mathcal{P}_1$ then we assign $i_{p+1}=i_{p-1}$; otherwise we apply Step 2 case (b2) to construct $i_{p+1}$. Now we consider the remaining times $t\geq p+1$. In Step 3 we initiate a new scanning procedure from $i_{p+1}$ onward. Suppose $i_t$ is the un-scanned index with the smallest $t\geq p+1$. If (a) $i_t\in\mathcal{P}_1$, then there are two cases either (a1) the index $i_t$ has been visited before the instant $t$ \footnote{For $t\leq p$, any occurrence of $i_t\in\mathcal{P}_1$ with $i_t$ visited before time $t$ correspond to a doubled edge, and the instance of the smallest $t$ that this happens belong to $\cup_{k\geq 2}\mathcal{P}_k$, which has been constructed via applying Step 2 (b2). Similarly, any $i_t\in\mathcal{P}_1$ visited twice after time $p$ lead to doubled edge and has been constructed via applying Step 2 (b2).}, or (a2) the index $i_t$ hasn't been visited before time $t$. In case (a1) there is a unique $l< t$ with $i_l=i_t$,  (a11) if $l\leq p$ then we insert $i_{t+1}:=i_{l-1}$ into the path and increase the subscript of indices after $i_{t}$, so that $x_{i_{l-1}i_{l}}$ is traversed twice in the resulting path, and (a12) the case $l> p$ will not happen since $\bar{x}_{i_li_{i-1}}$ does not pair with $\bar{x}_{i_ti_{t-1}}$ by independence, unless $i_{t-1}=i_{l-1}$. The latter case corresponds to a self-loop which is constructed in Step 2. In case (a2) then $i_t$ is an odd index, and we do nothing but move on to scan the next index in the path (there must be un-scanned elements remaining). Then if (b) $i_t\in\cup_{i\geq 2}\mathcal{P}_i$, say for example $i_t\in\mathcal{P}_k$, then we can either assign $i_{t+1}$ to be odd or even. If(b1) we assign $i_{t+1}$ to be odd (this must be the case if $i_t$ is not the last index in the current path), then we do no insertion and simply move on to $i_{t+1}$.
Or, (b2) we assign $i_{t+1}$ to be even, then $\bar{x}_{i_{t+1}i_t}$ is a repetition of elements appearing before it, then there are  $k$ choices of $i_{t+1}$, ad we insert this value to the path with label $i_{t+1}$ and increase the label of indices after $i_t$. Then we proceed to scan the next element.

In Step 2 and 3 we have designed two scanning and insertion procedures, and both steps lead to additional combinatorial factors. The important take-away of our construction is that the scanning procedure in Step 2 amounts to the construction of repeated loops that are either both in the first, or second, half of the path whereas the scanning procedure in Step 3 amounts to the pairings of two elements with one in the first half, and the other in the second half, of the path. We have assigned the even vertices in $\cup_{k\geq 2}\mathcal{P}_k$ a label $\{0,1\}$ to determine whether they are used in Step 2 or Step 3, and regardless of this label, every even vertex in $\cup_{k\geq 2}\mathcal{P}_k$ contribute a factor of $2k$ by the number of choice for possible returns.

\textbf{Combining all the combinatorial factors.}
Combining Steps 2 and 3, we have introduced the following combinatorial factor: at each odd instance in $\cup_{k\geq 2}\mathcal{P}_k$, we have two choices to determine whether the next index is odd or even. For an even index there are $k$ choices of return (and there are $kn_k$ vertices in $\mathcal{P}_k$ with even index in the path) while for an odd index there is no additional choice. Thus the total combinatorial factor is 
$$\prod_{k\geq 2}(8k)^{kn_k}.
$$
\textbf{Explanation of this combinatorial factor.} We now explain how this $\prod_{k\geq 2}(\text{const}\cdot k)^{kn_k}$ factor arises: the reasoning here is much the same as the paragraph  `` proving the upper bound for the number of choices for the second half of the path" in the proof of Lemma \ref{lemma2.111}, so we only give a sketch here. The essential reason is that not all the methods of the choice of the next vertex starting from a vertex in $\cup_{k\geq 2}\mathcal{P}_k$ is allowed: only those choices that obey the overall constraint that each such edges appear $2k$ times is contributing to the expectation. Then one can prove by induction, as in Lemma \ref{lemma2.111}, that we get this quantity $\prod_{k\geq 2}(8k)^{kn_k}$ for the possible ways to insert the even edges.

Now we bound $\mathbb{E}[\operatorname{Tr}(A^p(A^*)^p)]$ via this path reconstruction procedure. Again the dominant contribution comes from the case where $n_1=p$ and $n_k=0$ for any $k\geq 2$, and the contribution of this term is as in \eqref{dominantsums}.  For general $n_1,n_2,\cdots n_k$ the contribution is bounded by 
\begin{equation}\label{contributiongeneralspecific}
n\frac{1}{n^p}\frac{n!}{n_0!n_1!\cdots n_p!}\frac{p!}{\prod_{k=2}^p(k!)^{n_k}}\prod_{k=2}^p(2k)^{kn_k}\prod_{k=2}^s(\text{const}\cdot k)^{kn_k} \cdot 4^{\sum_{k\geq 2}kn_k}.
\end{equation}
To explain each term, we first divide $n$ into subsets of size $n_1,\cdots,n_p$, and choose the number of ways in which elements appear in the first half of this product term up to index $p$. At this point we assume all the indices thus appeared up to time $p$ are odd indices, but we have freedom to switch indices in $\mathcal{P}_2,\mathcal{P}_3,\cdots$ to even indices: this leads to the factor $4^{\sum_{k\geq 2}kn_k}.$ The term $(2k)^{kn_k}$ again represents the choice of the next index if the current index lies in $\mathcal{P}_t,t\geq 2$ and is an even index. The term $(\text{const}\cdot k)^{kn_k}$ are the sub-Gaussian moments.

\textbf{Sub-exponential entries in place of subgaussian entries.} An important observation is that \eqref{contributiongeneralspecific} is true even if $x_{ij}$ has only sub-exponential moments \eqref{gaussianmoments}: this is an observation made already in \cite{MR1647832}, Lemma 1. The reason is that the factor $k^{kn_k}$ comes from choosing the vertices of return in every visit to a point in $\cup_{k\geq 2}\mathcal{P}_k$, and if an edge $e=(e_i,e_j)$ is chosen for return for $\ell\geq 2$ times, this would contribute to the sub-exponential moment of $x_{ij}$ by a factor $(\text{const}\cdot \ell)^{2\ell}$, but in this case we need to divide by a factor $\ell!$ as the choice of a returning edge from these $\ell$ identical edges are identical. (In other words, the number of possible returns and the factor in the sub-exponential moments cannot be too large simultaneously). Then the same reasoning as in \cite{MR1647832}, Lemma 1 shows \eqref{contributiongeneral} is true under the weaker sub-exponential moment condition \eqref{gaussianmoments}, which we assume henceforth. More concretely, we are enumerating over all even closed paths $\mathcal{P}$ of length $2p$. Let $(i,j)$ be any directed edge in $\mathcal{P}$, we denote by $2l((i,j))$ the number of occurrence of this directed edge in $\mathcal{P}$, then under the sub-exponential moment condition we need to bound, accounting for the $\ell!$ overcounting:
$$
\prod_{(i,j)}'(\text{const}\cdot l(i,j))^{2l(i,j)}\frac{1}{l((i,j))!}\text{       in place of         }\prod_{(i,j)}'(\text{const}\cdot l(i,j))^{l(i,j)}
$$ where $\prod'$ is the product over all edges with $l(i,j)>1$, and the right hand side is the sub-Gaussian bound without dividing the overcounting. As the left hand side is bounded by the right by a suitable choice of constants, we see that switching to sub-exponential tails does not affect the upper bound \eqref{contributiongeneralspecific}. 

\textbf{Final bound.} Taking the summation over all indices such that $\sum_{k\geq 2}kn_k>0$, we get (as in the proof of Lemma \ref{lemma2.111}, with only change of constants) that the sum is bounded by \begin{equation}\label{lastboundsssa}
n\left[\exp\left(\sum_{k=2}^p\frac{(32e\cdot\text{const}\cdot k\cdot p)^k}{n^{k-1}}\right)-1\right],
\end{equation}which is bounded by $n\frac{p^2}{n}\frac{\text{const}}{1-\text{const}\cdot p^2/n}$
and thus is
$n\cdot O(1)$ whenever $p^2\leq C_0n$ for some constant $C_0$ depending on the sub-exponential moments of $x_{ij}$, and is $o(n)$ for $p\ll n^{\frac{1}{2}}.$  
\end{proof}

Next we show that we can remove the assumption that $x_{ij}$ has a symmetric distribution in Theorem \ref{momenttheorem2.2}. 

\begin{theorem}\label{nosymmetricdistribution} Theorem \ref{momenttheorem2.2} is valid without assuming that $x_{ij}$ has a symmetric distribution. 
    
\end{theorem}

The proof of Theorem \ref{nosymmetricdistribution} follows somewhat standard arguments, but is too lengthy to be presented here so we defer the proof to Appendix \ref{appendixC}. The idea is to show that contributions to the expectation from non-even paths can be bounded via considering adding odd edges to even paths, and we show such even paths with added odd edges has negligible contribution. Motivated by \cite{peche2007wigner} in the Wigner case, the key component of the proof is a method to construct non-even paths from even paths.

\subsection{Moment estimate II}\label{momentestimate2}
In this section we outline necessary moment computations for the proof of Theorem \ref{convergencespectral}. In contrast to Theorem \ref{momenttheorem2.2} considering moments $p\ll \sqrt{n}$, we now work with much higher moments $p\sim n$ and new combinatorial ingredients are necessary.

This type of results was first obtained for Wigner matrices in \cite{feldheim2010universality} via refined combinatorial techniques. In \cite{feldheim2010universality}, equation (I.5.6) they prove that
\begin{theorem}\label{theorem2.333}
Let $W_n=(w_{ij})$ be an $n\times n$ Wigner matrix with entry distribution satisfying assumptions (1),(2) of Theorem \ref{convergencespectral}. In the complex-valued case we assume further that $\mathbb{E}[w_{ij}^2]=0$ for each $i\neq j$. Then we can find $C>0$ such that for any $k\in\mathbb{N}_+$,
$$\mathbb{E}\left[\operatorname{Tr}\left(\frac{W_n}{2\sqrt{n-2}}\right)^k\right]
\leq C\frac{n}{k\sqrt{k}}e^{C\frac{k^3}{n^2}}.
$$
\end{theorem}
We note that for Gaussian random matrices, one can use the theory of Gaussian processes to derive a similar bound as in Theorem \ref{theorem2.333} without too much difficulty; see \cite{bandeira2016sharp}, Lemma 2.2. 

In the non-Hermitian setting, there are no results similar to Theorem \ref{theorem2.333} concerning traces of high powers, to the author's best knowledge. For Gaussian matrices (the complex/real Ginibre ensemble), the situation does not appear to be simpler, as the spectral radius of a Gaussian matrix is not a Gaussian process (in contrast to the spectral norm) and the considered quantity \eqref{consideredquantity} does not admit a closed form expression in terms of integrable formulas of real/complex Ginibre ensemble. We will derive the necessary moment bounds by ourselves, adapting the combinatorial tools in \cite{feldheim2010universality} to the non-Hermitian setting. The main result is
\begin{theorem}\label{mainresultveryhighbounds}
    Let $G_n=(\frac{1}{\sqrt{n}}g_{ij})$ be an $n\times n$ matrix with independent entries, where $g_{ij}$ are i.i.d. random variables satisfying assumptions (1),(2),(3) of Theorem \ref{convergencespectral}. Then we can find $C>0$ depending only on the constant in equation \eqref{gaussianmomentsrealgaussian}  such that for any $k\in\mathbb{N}_+,k\leq n$
    \begin{equation}\label{consideredquantity}
\mathbb{E}[\operatorname{Tr}G_n^k(G_n^*)^k]\leq Ck^2n^6e^{C[\frac{k^{3/2}}{n^{1/2}}+\frac{k^2}{n}]}.
    \end{equation}
\end{theorem}
The leading factor $n^6$ is possibly unnecessary, but this does not lead to any suboptimality in the proof of Theorem \ref{convergencespectral}. This is because we will set $k\gg\log n$, so that $n^\frac{6}{k}\to 1.$

\subsubsection{Recap of combinatorial terminologies}

Our proof of Theorem \ref{mainresultveryhighbounds} uses many combinatorial terminologies from \cite{feldheim2010universality}. We outline these combinatorial constructions. 

In our combinatorial construction we differentiate between two cases: the real case $\beta=1$ and the complex case $\beta=2$. The estimates are slightly different in both cases.

Fix $n\in\mathbb{N}_+$. Consider a path $p_{2k}=u_0u_1\cdots u_{2k}$ on $\{1,2,\cdots,n\}$ satisfying
\begin{enumerate}
    \item 
 $u_j\neq u_{j-1}$ (no self-loops) for each $j$; \item $u_j\neq u_{j-2}$ (non-backtracking) for each $j$; \item $u_0=u_{2k}$ (closed path); \item \begin{enumerate}
     \item  In the real case for any $u\neq v\in[n]$, the number of $j\in[2k]$ such that $u_j=u,u_{j+1}=v$ and the  number of $j\in[2k]$ such that $u_j=v,u_{j+1}=u$ are equal mod 2. \item In the complex case for any $u\neq v$, the number of $j$ such that $u_j=u,u_{j+1}=v$ equals the number of $j$ such that $u_j=v,u_{j+1}=u$.
\end{enumerate}\end{enumerate}

\begin{Definition}\label{directmatchings}(Directed multigraphs and matching)
From this path $p_{2k}$ we construct a directed multigraph $G=(V,E_{\text{dir}})$ where $V\subset\{1,\cdots,n\}$ consists of all vertices $u_j,0\leq j\leq 2k$, and $E_{\text{dir}}$ is the collection of all directed edges $u_{j-1}\to u_j$ (counting multiplicity). We define a matching of $p_{2k}$ as an involution of $\{0,1,\cdots,2k-1\}$ that satisfies (i) in the real case, an edge $(u,v)$ is matched to $(u,v)$ or $(v,u)$; or (ii) in the complex case, each edge $(u,v)$ is matched to $(v,u)$.
\end{Definition}
In Definition \ref{directmatchings}, we use the multigraph $G$ to encode the path $p_{2k}$. To have a heuristic understanding of our definition of matchings (a.k.a., pairings), simply note that these pairings will arise when we apply Wick formula for Gaussian matrices: the number of matchings is exactly the combinatorial factor of the number of pairings showing up in Wick formula.   

We call a path $p_{2k}$ with a matching as a matched path. To effectively estimate the number of matched paths, in \cite{feldheim2010universality} the authors introduced an algorithm of combinatorial reduction, to reduce a matched path to the following simpler object:

\begin{Definition}
\cite{feldheim2010universality} For $\beta=1,2$, we define a diagram of class $\beta$ as an undirected multigraph $\bar{G}=(\bar{V},\bar{E})$ combined with a circuit $\bar{p}=\bar{u}_0\bar{u}_1\cdots\bar{u}_0$ on $\bar{G}$ satisfying that 
\begin{enumerate}
\item $\bar{p}$ is not backtracking, meaning that $\bar{u}_j\neq\bar{u}_{j-2}$ for each $j$. \item For any edge $(\bar{u},\bar{v})\in\bar{E}$, in the real case the number of $j$ with $\bar{u}_j=\bar{u},\bar{u}_{j+1}=\bar{v}$ plus the number of $j$ with $\bar{u}_j=\bar{v},\bar{u}_{j+1}=\bar{u}$ equals 2; and in the complex case both quantities are one. \item The degree of $\bar{u}_0$ in $\bar{G}$ is 1 and the degree of any other vertices is $3$.\end{enumerate} Finally, we add a weight function $\bar{w}:\bar{E}\to \{-1,0,1,2,\cdots\}$ defined in the graph $\bar{G}$.

As noted in \cite{sodin2010spectral}, Remark 2.4, $\bar{G}$ is a multigraph where the coinciding edges are distinguished, so the circuit is not uniquely defined by the vertices through which it passes.
\end{Definition}

The algorithm of reducing matched paths to diagrams is as follows:

\begin{Proposition}\label{prop2.8new}
(\cite{feldheim2010universality}, Claim II.1.4) There exists a mapping from the set of matched paths $p_{2k}$ to a diagram $(\bar{G},\bar{p},\bar{w})$. By \cite{feldheim2010universality}, for each weighted diagram $(\bar{G},\bar{p},\bar{w})$ there are no more than $n^{|\bar{V}|+\sum_{\bar{e}}\bar{w}(\bar{e})}$ matched paths that are mapped to this weighted diagram $(\bar{G},\bar{p},\bar{w})$.
\end{Proposition}

The final step is to find an upper bound for the number of diagrams. In \cite{feldheim2010universality}, Section II.2, an algorithm was designed for this purpose:

\begin{Proposition}\label{prop2.9new}(\cite{feldheim2010universality}, Section II.2) There exist an automaton generating all possible  diagrams.
In the complex case $\beta=2$, each diagram is generated in $s$ steps for some $s\in\mathbb{N}$, and for a diagram generated in $s$ steps the diagram has $|\bar{E}|=3s-1$ edges and $|\bar{V}|=2s$ vertices. In the real case $\beta=1$, each diagram can be generated in some $s\in\mathbb{N}$ steps with $|\bar{E}|=3s-1$ edges and $|\bar{V}|=2s$ vertices. In \cite{feldheim2010universality}, The number of diagrams with parameter $\beta\in\{1,2\}$ and generated in $s$ steps has cardinality $D_\beta(s)$ which satisfy, for some universal constant $C>0$,
\begin{equation}\label{dbetas}
    (s/C)^s\leq D_\beta(s)\leq (Cs)^s.
\end{equation}
\end{Proposition}

\subsubsection{An illustrative explanation for the proof of Theorem \ref{mainresultveryhighbounds}}
We outline how the proof of Theorem \ref{theorem2.333} from \cite{feldheim2010universality} can be adapted to the proof of Theorem \ref{mainresultveryhighbounds}, at a non-rigorous level. Only the main ideas are highlighted here and the details are presented later.

The key estimate in the proof of Theorem \ref{theorem2.333}  is in \cite{feldheim2010universality}, Theorem I.5.3, which states that:

Let $U_n$ denote the $n$-th Chebyshev polynomial $U_n(\cos\theta)=\frac{\sin((n+1)\theta)}{\sin\theta}$, then $$\mathbb{E}[\operatorname{Tr}U_k(\frac{W_n}{2\sqrt{n-2}})]\leq Ck\exp(Ck^{3/2}/n^{1/2}).$$ It can be shown that the composition by Chebyshev polynomial essentially amounts to counting only the non-backtracking paths in the trace expansion of $(W_n)^k$. Also note that $U_k(x)\sim (2x)^k$ when $x$ is large, which cancels a factor $2$ in the denominator of $\frac{W_n}{2\sqrt{n-2}}$.
In the non-Hermitian case, because the non-backtracking condition on path $p_{2k}$ is automatically imposed on the expansion $\operatorname{Tr}(G_n^k(G_n^*)^k)$ thanks to independence of entries of $G_n$, one expects the same computation should also generalize. There do exist paths that are backtracking and contribute to the expectation, but the contribution of these backtracking paths will be negligible compared to the non-backtracking paths, as we shall prove later.

\subsubsection{The proof of of Theorem \ref{mainresultveryhighbounds}}
\begin{proof}[\proofname\ of Theorem \ref{mainresultveryhighbounds}, the real case $\beta=1$] Fix $i,j\in [n]$. Let $P$ be a path from $i$ to $j$ of length $k$, that is, the path $P$ can be denoted by $P:i=p_0\mapsto p_1\mapsto \cdots\mapsto p_k=j$. For this path $P$ we can define its weight $w(P)$ by $$w(P)=\prod_{t=0}^{k-1}g_{p_tp_{t+1}}, 
$$  and we denote by $\bar{w}(P)$ the complex conjugate of $w(P)$. We can expand the trace via
    $$ 
\operatorname{Tr}\left(G_n^k(G_n^*)^k\right)=\sum_{i,j}\sum_{P_1,P_2:i\mapsto j}w(P_1)\bar{w}(P_2),
    $$ where $P_1,P_2$ are directed paths from $i$ to $j$ of length $k$.

    \textbf{Step 1: Reduction to path counting.}
    We can construct from the paths $(P_1,P_2)$ a path $P$ of length $2k+2$ as follows: we start from $i$, then follow $P_1$ to vertex $j$, then take the edge $(j\mapsto i)$, then follow $P_2$ to get to vertex $j$ again, and finally return to vertex $i$ via edge $(j,i)$. For fixed $i,j$ this defines a unique mapping from pairs of paths $(P_1,P_2)$ from $i$ to $j$ of length $k$, and a cycle of length $2k+2$ rooted at edge $(j,i)$ (we say the path of length $2k+2$ is rooted at $(j,i)$ if the $k+1$-th and the $2k+2$-th step of the path both take the edge $j\to i$).  Let $\mathcal{P}_{(j,i)}(2k+2)$ denote the set of even rooted paths of length $2k+2$ rooted at $(j,i)$, where a rooted path is called an even rooted path if each directed edge has an even multiplicity. Since $g_{ij}$ are symmetric random variables, to get a nonzero expectation for $\mathbb{E}[w(P_1)\bar{w}(P_2)]$, the resulting path $P$ of length $2k+2$ must have each edge an even multiplicity. Then 
    \begin{equation}\label{expectationterms}\mathbb{E}\left[\operatorname{Tr}\left(G_n^k(G_n^*)^k\right)\right]
 =\sum_{i,j}\sum_{P\in\mathcal{P}_{(i,j)}(2k+2)}\frac{1}{n^k} \mathbb{E}[W_r(P)],  \end{equation} where for an even rooted path $P=(p_0,p_1,\cdots,p_{2k+2})$ of length $2k+2$ and root $(j,i),$ we define 
$$
W_r(P)=\prod_{t=0}^{k-1}g_{p_tp_{t+1}}\cdot \prod_{t=k+1}^{2k}\bar{g}_{p_tp_{t+1}},
$$ note that we have taken complex conjugates in the second part of the path and we have omitted elements $g_{p_kp_{k+1}}$ and $g_{p_{2k+1}p_{2k+2}}$ in the product defining $W_r(P)$.

We take a decomposition of $\mathcal{P}_{(i,j)}(2k+2)$ into two subsets 
$$\mathcal{P}_{(i,j)}(2k+2)=\mathcal{P}^1_{(i,j)}(2k+2)\cup \mathcal{P}^2_{(i,j)}(2k+2)$$
where $\mathcal{P}^1_{(i,j)}(2k+2)$
consists of paths that are non-backtracking and has no edges $(e,e)$ in it (so that $p_i\neq p_{i+1},p_{i+2}$ for each $i=0,1,\cdots,2k$.) Let  $\mathcal{P}^2_{(i,j)}(2k+2)$ consist of other paths that do not satisfy these conditions. The collection of paths in $\mathcal{P}^1_{(i,j)}(2k+2)$ make the primary contribution to the expectation in \eqref{expectationterms}, which we now estimate. 

\textbf{Step 2: Non-backtracking paths via Wick formula in the Gaussian case.}
First we assume that $g_{ij}$ are standard real Gaussian variables, then by Wick formula, 
\begin{equation}
\label{wickformulagood}
\sum_{P\in\mathcal{P}^1_{(i,j)}(2k+2)}\mathbb{E}[W_r(P)]\leq\Omega_{(i,j)}(2k+2), 
\end{equation} where $\Omega_{(i,j)}(2k+2)$ is the number of doubled matched paths (a path whose each directed edge has even multiplicity and where a matching has been assigned to this path, see Definition \ref{directmatchings}) $P$ of length $2k+2$, containing the edge $(j,i)$, satisfying (a) $p_i\neq p_{i+1}$ for each $i,$  (b) $p_i\neq p_{i+2}$ for each $i$, (c) $p_0=p_{2k+2}$. and (d) are matched in the following sense: each edge $(u,v)$ is matched either to $(u,v)$ or to $(v,u)$. Since the entries $g_{ij}$ are symmetric, for the expectation to be non-vanishing we require that (e) for any $u,v$, the number of times each edge $(u,v)$ and $(v,u)$ appear in $P$ are equal mod 2.
These five constraints are the same as the definition of matched paths in Definition \ref{directmatchings}.

We give more explanation to the validity of \eqref{wickformulagood} for why it corresponds to paths with a matching: by Wick formula (\cite{nica2006lectures}, Theorem 22.3 for the real case and Remark 22.5 for the complex case) the expectation of $W_r(P)$ is given by numbers of pairings without pairing the roots: the $(k,k+1)$ and $(2k+1,2k+2)$-th edge of $P$. These two edges $j\to i$ can be further paired together resulting in the counting $\Omega_{(i,j)}(2k+2)$: note that for this quantity $\Omega_{(i,j)}(2k+2)$ we not only count the number of paths, but also count the number of possible matchings associated to this path. This gives rise to Wick formula computations.

\textbf{Step 3: Counting non-backtracking paths.}
In the real case $\beta=1$, we can estimate the quantity $\Omega_{(i,j)}(2k)$ for each $k\in\mathbb{N}$ just as in \cite{feldheim2010universality}. \footnote{For future use, the fact that such path should go through $(j,i)$ is not very important and we shall ignore this restriction: this leads to an overcounting of at most $n^2$.}
By Proposition \ref{prop2.8new} and \ref{prop2.9new}, each such matched path is mapped to a weighted diagram $(\bar{V},\bar{E},\bar{p})$ in $s$ steps for some $1\leq s\leq k$, with $|\bar{V}|=2s$, $|\bar{E}|=3s-1$, and with a set of weight functions $\bar{w}(e)$ satisfying that $\sum_{\bar{e}}\bar{w}(\bar{e})=k-|\bar{E}|=k-3s+1$. The number of diagrams $(\bar{V},\bar{E},\bar{p})$ generated in $s$ steps is denoted by $D_\beta(s)$, which satisfies the upper bound \eqref{dbetas}. The number of ways to assign the weights $\bar{w}$ on this path is $\binom{k+3s-2}{3s-2}$ and the number of ways to choose vertices is $n^{2s+k-3s+1}=n^{k-s+1}$. Therefore
\begin{equation}\label{omegaj2s}\begin{aligned}
\Omega_{(i,j)}(2k)&\leq \sum_{1\leq s\leq k}D_\beta(s)n^{k-s+1}\frac{(k+3s-2)^{3s-2}}{(3s-2)!}\\&\leq\sum_{1\leq s\leq k}C^{s}s^sn^{k-s+1}\frac{(k+3s-2)^{3s-2}}{(3s-2)!}\\&\leq kn^k\sum_{s\geq 1}\frac{(C_1k^3/n)^{s-1}}{(2(s-1))!}\leq kn^k\exp(C_2k^{3/2}/n^{1/2}),
\end{aligned}\end{equation} 
where $C_1,C_2>0$ are some universal constants.

Plugging in this upper bound for $\Omega_{(i,j)}(2k)$ into \eqref{wickformulagood} and summing over $i,j$, we get an upper bound for the trace moment with a magnitude almost as large as the one claimed in Theorem \ref{mainresultveryhighbounds}. We are now left with the task to show that the contribution from the remaining paths that do not belong to $\Omega_{(i,j)}(2k+2)$, has at most the same magnitude.

Before proceeding with this task, we first show that the above computation works also for non-Gaussian distribution with only a minor modification:

\textbf{Step 3.1: non-Gaussian distribution.}
The same estimate also holds (with a different constant $C_2>0$) when $g_{ij}$ are general sub-Gaussian random variables satisfying assumptions (1),(2),(3) in Theorem \ref{convergencespectral}. Here we follow the ideas of \cite{feldheim2010universality}, Section III.2. Instead of using Wick formula we use the sub-Gaussian moment bound \eqref{gaussianmomentsrealgaussian} which leads to a product of $(\text{const}\cdot k)^k$ terms, and these terms are absorbed in the $s^s$ factor. More precisely, for a constant $C>0$ we consider the quantity
$$
\Omega_{(i,j)}^C(2k):=\sum_{P\in \Omega_{(i,j)}(2k)}C^{k-|E(P)|}
$$ where $|E(P)|$ is the number of distinct edges in $P$ and we are summing over matched paths $P$. Using the sub-Gaussian moment condition and elementary combinatorics, we can find some $C_3>0$ depending on the sub-Gaussian moment in \eqref{gaussianmomentsrealgaussian} such that 
$$
\sum_{P\in\mathcal{P}^1_{(i,j)}(2k+2)}\mathbb{E}[W_r(P)]\leq \Omega_{(i,j)}^{C_3}(2k+2).
$$ To derive this inequality we do not use the Wick formula, but directly evaluate the expectation counting the multiplicity of each entry $g_{ij}$. There is a $(k/c)^k$ factor coming from the sub-Gaussian moment \eqref{gaussianmomentsrealgaussian} but this factor is absorbed in the number of pairings, modulo a constant left out from each repeated edge. The number of repeated edges for a path $P$ of length $k$ is $k-|E(P)|$, hence the computation.

We now let $b$ be the number of edges $\bar{e}$ such that $\bar{w}(\bar{e})=-1$, where we recall $\bar{w}$ is the weight function defined on a diagram of class $\beta=1$ generated from $P$. Then for a path $P$ of length $2k+2,$ $|k+1-E(P)|\leq b$. Therefore 
$$\begin{aligned}
\sum_{P\in\mathcal{P}^1_{(i,j)}(2k+2)}\mathbb{E}[W_r(P)]&\leq \sum_{1\leq s\leq k+1}D_1(s)n^{k-s+1}\sum_{b\geq 0} (C_3)^b\binom{k+3s-2}{3s-2-b}\binom{3s-1}{b}\\&\leq \sum_{1\leq s\leq k+1}C^{s-1}s^sn^{k+2-s}\binom{k+3s-1}{3s-2}\sum_{b\geq 0}\frac{(C_4s/k)^b}{b!}
\end{aligned}$$ 
for some constant $C_4>0$ depending on $C_3$. Since $\sum_{b\geq 0}\frac{(C_4s/k)^b}{b!}=O(1)$, we can bound the other terms as in \eqref{omegaj2s}  and conclude that, for some $C_5>0$ depending only on the sub-Gaussian constant of $g_{ij}$ that 
$$
\sum_{P\in\mathcal{P}^1_{(i,j)}(2k+2)}\mathbb{E}[W_r(P)]\leq C_5(k+1)n^{k+1}\exp(C_5k^{3/2}/n^{1/2}).
$$

\textbf{Step 4: Counting backtracking paths.}
The remaining task is to bound the contribution from paths not counted in \textbf{Step 3}, which by definition are those that are either not backtracking or has two adjacent vertices that are equal. We will manually remove the backtracking parts and parts where adjacent indices are equal, and reduce the count of these paths to a path of shorter length. Recall that we construct, from two paths $P_1,P_2:i\to j$ of length $k$, a path $P$ of length $2k+2$
via first going through $P_1$, then the edge $j\to i$, then $P_2$, then $j\to i$. The path $P$ may be backtracking, but since the entries $g_{ij}$ are i.i.d., for $\mathbb{E}[W_r(P)]\neq 0$, we must have that for each backtrack $p_{i-1}\mapsto p_{i}\mapsto p_{i+1}=p_{i-1}$, unless $p_{i-1}=p_i=p_{i+1}$, we must not match $(p_{i-1},p_i)$ with $(p_{i},p_{i+1})$, but we match these two directed edges with other edges in the path $P$.

\textbf{Reduction to only one type of backtracking.}
There are two possible cases where this backtrack will be matched with other elements of $P$ occurring later. Case (A), we can find some $j>i$ such that there is another backtracking $p_{j-1}\mapsto p_j\mapsto p_{j+1}=p_{j-1}$ and $(p_{i-1},p_i)=(p_{j-1},p_j)$ or $(p_j,p_{j-1})$. This is the most generic case of backtracking. Case (B), we can find some $j\neq k$ with $j,k>i$ such that $(p_{i-1},p_i)=(p_{j-1},p_j)$, $(p_{i},p_{i+1})=(p_{k-1},p_k)$. This case is more exotic, as we cannot simultaneously remove the two edges $p_{i-1}\mapsto p_i\mapsto p_{i+1}$ and $p_{j-1}\mapsto p_j,p_{k-1}\mapsto p_k$ as the remaining path is disconnected, but we can convert case (B) into case (A) as follows. For such a path $P$ of length $2k+2$ in case (B), we identify the two vertices of the edge $p_{j-1}\mapsto p_j$ and collapse this edge to one single vertex, and do the same operation to the edge $p_{k-1}\mapsto p_k$, so we get a unique closed even path $\tilde{P}$ of length $2k-2$, with two backtracks attached to it that are matched to each other, which is case (A).
Therefore, we only need to bound backtracking of class (A), from which we can bound contributions in case (B).

\textbf{Decomposing a backtracking path into forests.} Now we bound contributions from case (A). We adapt the arguments from \cite{feldheim2010universality}, Section III.3. For a closed even path $P$ of length $2k+2$ not belonging to case (B), we can (for some integer $m\in\mathbb{N}$) decompose $P$ as the union of a closed even path $q_{2(k+1-m)}:=\mathfrak{C}(P)$ with no backtracking and no self loops, a forest $f_{2m}=\mathfrak{F}(P)$ and the remaining self-loops. By assumption of the path $P$ not being in case (B), each leaf of $f_{2m}$ should appear at another place on $P$. The contribution from $q_{2(k+1-m)}$ is well-understood in the previous steps, and in the following we study the structure of each connected component of $f_{2m}$ in detail.

From \cite{feldheim2010universality}, Definition III.3.1 we define the notion of a tree diagram (abbreviated as $t$-diagram) to be a rooted planar binary tree. We assign a weight function $\bar{w}$ to this $t$-diagram which takes values in $\{-1,0,1,2,\cdots\}$ on each edge of the tree. By \cite{feldheim2010universality}, Lemma III.3.2 (which follows from elementary combinatorial counting), the cardinality $D^t(\ell)$ of $t$-diagrams with $\ell$ leaves is bounded by $4^\ell$.

As each leaf of $f_{2m}$ should appear somewhere else on $P$, we suppose that, for some $0\leq\ell_1\leq\ell/2$, we have $\ell_1\leq \ell/2$ coinciding pairs of leaves and the remaining $\ell-2\ell_1$ leaves coincide with some other vertices of $P$ which are not leaves.

Then we upper bound the number of ways to form a weighted $t$-diagram $t_{m'}$ for some $m'\in\mathbb{N}$. The total number of trees with $\ell$ leaves is bounded by, using $k\leq n$ and $\ell\leq k$, 
$$
\sum_{0\leq\ell_1\leq\ell/2}\binom{\ell}{2\ell_1}k^{\ell-2\ell_1}n^{2\ell_1}n^{m'-2\ell}\frac{(2\ell_1)!}{2^{\ell_1}\ell_1!}\leq n^{m'-\ell}(C\ell)^{\ell},
$$ where the first binomial coefficient gives the choice of coinciding pairs of leaves; the last binomial factor the number of ways to pair these coinciding leaves; the power of $k$ is the number of ways these leaves coincide with other vertices of $P$, and the two powers of $n$ are the way to choose these vertices from $[n]$. 

This $t$-diagram $t_m'$ with $\ell$ leaves has $2\ell-1$ edges. In order for the summation of weights on the $t$-diagram to not exceed $2m'$, the count of ways to assign the weight on the $t$-diagram is bounded by, for some universal constant $C_6>0$, $$
\binom{2m'+2\ell-1+2\ell-2}{2\ell-2}\leq(C_6m'/\ell)^{2\ell-2}.
$$
Finally we take the summation over $\ell$ as follows: the number of $t$-diagrams $t_{m'}$ is at most
$$
\sum_{\ell\geq 1}(C_6m'/\ell)^{2\ell-2} (C\ell)^{\ell}n^{m'-\ell}\leq n^{m'-1}\exp(C_7(m')^2/n).
$$
A more careful computation as in \textbf{Step 3.1} allows us to upper bound the weight coming from the forest $f_{2m}$, which we omit for simplicity. 

\textbf{Total contribution from forests: summing everything up.}
Now we consider all the connected components of $f_{2m}$. Then the contributions from all $t$-tuples of trees of $m_i$ edges each, with $t\in\mathbb{N}_+$ and $m_1+\cdots+m_t=m$, is upper bounded by 
$$
\sum_{m_1+\cdots+m_t=m}\prod_{j=1}^tn^{m_j-1}\exp(C_7(m_j)^2/n)\leq n^{m-t}\exp(C_7m^2/n)\binom{m-1}{t-1}.
$$Finally, the total contribution is bounded as follows: summing over $t$ first and then over $m$, 
$$\sum_{1\leq t\leq m,t,m\in\mathbb{N}}
n^{m-t}\binom{k+1-m+1}{t}n^{k+1-m}\binom{m-1}{t-1}\exp(C_7m^2/n)\Sigma_1(2(k+1-m))
$$
where $\Sigma_1(2(k+1-m))$ is the contribution from non-backtracking paths computed in \textbf{Step 3.1}.
Note that for each fixed $1\leq m\leq k+1$,
$$
n^{-t}\binom{k-m+2}{t}\binom{m-1}{t-1}\leq (\frac{2km}{n})^t\cdot\frac{1}{t!},
$$ which is summable over $t$. Then we may simply deduce that the total contribution of case (A) is bounded by 
$$ (k+1)^2n^{k+1}\exp(C_8k^2/n+C_8k^{3/2}/n^{1/2}).
$$
Summing over $i,j$, the upper bound has the same magnitude as the one claimed in \eqref{consideredquantity}.

As discussed previously, the contribution from case (B) can be bounded by that from case (A), up to a change of constant, say, for example, replacing $(C\ell)^\ell$ by $(2C\ell)^\ell$. The remaining computations follow the same line.

\textbf{Self-loops.}
Finally, when there is a self-loop, then each self-loop must also appear an even number of times for us to get a non-zero expectation. Say, the self-loop repeats $2k$ times, then we remove the self-loops, resulting in a path with no self-loops. This leads to a loss of $n^k$ free vertices to choose form for this path. The subgaussian moment generated from the self loop is $(Ck)^k$. This leads to a factor $(Ck/n)^k$ to the overall contribution. It suffices to count the multiplicity of all self-loops and remove these loops.

\end{proof}

\begin{proof}[\proofname\ of Theorem \ref{mainresultveryhighbounds}, the complex case $\beta=2$]. We essentially follow the proof in the real case, but take a different method to form a path $P$ from $P_1,P_2$ in \textbf{Step 1}. Suppose we are given two paths $P_1,P_2:i\mapsto j$ of length $k$ each, then we form a path $P$ of length $2k+12$ via the following procedure. Select four distinct vertices $T_1,T_2,T_3,T_4\in[n]$, we follow the path $P_1$ until it gets to $j$, then follow $j\mapsto T_1\mapsto T_2\mapsto j\mapsto T_3\mapsto T_4\mapsto j\mapsto T_2\mapsto T_1\mapsto j\mapsto T_4\mapsto T_3\mapsto j$, and then follow path $P_2$ reversed in order, from $j$ back to $i$. This defines a unique map of paths: $(P_1,P_2)\mapsto P$. The matching on $P$ is defined such that $(u,v)$ is matched only to $(v,u)$ but not $(u,v)$ itself: this is consistent with the definition of $\beta=2$ matching in \cite{feldheim2010universality} and consistent with the assumption $\mathbb{E}[g_{ij}^2]=0$. We further require that the $k+1$ to $k+12$-th steps of $P$ are (uniquely) matched to itself: these are the steps we just appended to the original path. Then by definition, $P$ has no backtracking in the $k+1$-th to $k+12$-th steps, and by independence of entries of $g_{ij}$, for any possible backtracking of $P$ of the form $p_{i-1}\mapsto p_i\mapsto p_{i+1}=p_{i-1}$, we cannot match $(p_{i-1},p_i)$ with $(p_i,p_{i+1})$ (unless $p_{i-1}=p_i=p_{i+1}$) so they both have to be matched with other paths in $P$.

Then we are back to almost the same setting as in the real case $\beta=1$, only that we have a path of length $2k+12$ and we use the $\beta=2$ automaton in \cite{feldheim2010universality}. We just need to enumerate all closed even paths $P$ of length $2k+12$, matched in the $\beta=2$ type (so that for any two $u,v\in[n]$, the edge $(u,v)$ is traveled in $P$ the same number of times as the edge $(v,u)$), and such that all backtracking parts in $P$ are not matched to itself.
The rest of the proof is the same and omitted. 

\end{proof}

\section{Convergence and small deviation: the inhomogeneous case}\label{section3} 

\subsection{Convergence of spectral radius under optimal sparsity}

In this section we prove Theorem \ref{convergencespectral}.
The proof uses in a crucial way the moment computations in Theorem \ref{mainresultveryhighbounds} and an idea of comparison, dating back to \cite{bandeira2016sharp}, Proposition 2.1.

\begin{proof}[\proofname\ of Theorem \ref{convergencespectral}, the real case $\beta=1$]
For this proof we need some notations of paths and cycles from \cite{bandeira2016sharp}. We essentially follow verbatim the steps in \cite{bandeira2016sharp} although certain constructions are slightly different: in particular the meaning of $n_i(\mathbf{u})$ and $\mathcal{S}_{2p}$ in the following are different from those in \cite{bandeira2016sharp}.

Given a path $\mathbf{u}=(u_1,u_2,\cdots,u_{2p})\in[n]^{2p}$ of length $2p$, we identify it with a cycle $\mathbf{u}=(u_1,\cdots,u_{2p},u_1)\in[n]^{2p+1}$. We denote by $\mathcal{S}(\mathbf{u})$ the shape of $\mathbf{u}$ obtained by reordering the vertices via the order of first appearance. (For example, the cycle $7\mapsto 3\mapsto 5\mapsto 4\mapsto 3\mapsto 5\mapsto 4\mapsto 3\mapsto 7$ has its shape $1\mapsto 2\mapsto 3\mapsto 4\mapsto 2\mapsto 3\mapsto 4\mapsto 2\mapsto 1$). We say $\mathbf{u}$ is a \textbf{pair-admissible even cycle} if the following two statements (1) and (2) hold: 

(1) for each directed edge $(u_i\mapsto u_{i+1})$, $1\leq i\leq p$,  
\begin{equation}\label{536win}
\#| j\leq p:(u_j, u_{j+1})=(u_i,u_{i+1})|=\#|j\geq p+1:(u_{j+1},u_{j})=(u_i,u_{i+1})|\mod 2,
\end{equation} we then say the edge $u_i\mapsto u_{i+1}$ is equivalent to the edges in these two sets (edges ($u_j,u_{j+1})$ in the first set and $(u_{j+1},u_j)$ in the second), and we define the \textbf{multiplicity} of $(u_i\mapsto u_{i+1})$ to be the sum of the cardinality of the two sets in \eqref{536win}.

(2) And for each directed edge $(u_i\mapsto u_{i+1})$, $p+1\leq i\leq 2p$, 
\begin{equation}\label{536lose}
\#| j\leq p:(u_{j+1}, u_{j})=(u_i,u_{i+1})|=\#|j\geq p+1:(u_j,u_{j+1})=(u_i,u_{i+1})|\mod 2,
\end{equation} we then say the edge $u_i\mapsto u_{i+1}$ is equivalent to the edges in these two sets (edges ($u_{j+1},u_{j})$ in the first set and $(u_{j},u_{j+1})$ in the second), and we define the \textbf{multiplicity} of $(u_i\mapsto u_{i+1})$ to be the sum of the cardinality of the two sets in \eqref{536lose}. 

(In the above, by $(a,b)=(c,d)$ we mean $a=c,b=d$, and we set $u_{2p+1}=u_1$). In other words, we require that $\mathbf{u}$ can be paired according to the independent entries of $A$ taking into account the direction of edges, and we have reversed the order of edges in the second half of $\mathbf{u}$.
This definition is designed for the non-Hermitian case, while in the symmetric case we only need $\mathbf{u}$ be an even cycle without recording the direction an edge is visited \cite{bandeira2016sharp}.

Define $$\mathcal{S}_{2p}:=\{\mathcal{S}(u):u\text{ is a pair-admissible even cycle of length }2p\}.$$
Let $n_i(\mathbf{u})$ be the number of directed edges with multiplicity $i$ in the path $\mathbf{u}$, then $n_i(\mathbf{u})=n_i(\mathcal{S}(\mathbf{u}))$.

We denote, for each initial vertex $u\in[n]$ and each shape $\mathbf{s}\in\mathcal{S}_{2p}$, the following set of collections of path with shape $\mathbf{s}$ and initial vertex $u$: 
$$\Gamma_{\mathbf{s},u}:=\{\mathbf{u}\in[n]^{2p}:\mathcal{S}(\mathbf{u})=\mathbf{s},u_1=u\}.$$
As we compute the trace $\operatorname{Tr}(A^p(A^*)^p),$ we are summing over paths $P_1,P_2:i\mapsto j$, and we merge $P_1,P_2$ into a cycle as follows: we first start from $i$ and go through $P_1$, then we go through $P_2$ in the reversed direction ($j\mapsto i$).
Then we have 
\begin{equation}\label{tracemoment1s}
\mathbb{E}[\operatorname{Tr}(A^p(A^*)^p)]=\sum_{u\in[n]}\sum_{\mathbf{s}\in \mathcal{S}_{2p}}
\prod_{i\geq 1}\mathbb{E}[g^i]^{n_i(\mathbf{s})}\sum_{\mathbf{u}\in\Gamma_{\mathbf{s},u}}b_{u_1u_2}\cdots b_{u_{p-1}u_p}b^*_{u_pu_{p+1}}\cdots b^*_{u_{2p}u_{1}},
\end{equation} where we use the notation $(b^*)_{ij}=b_{ji}$ for all $i,j\in[n]$. Here $g$ is a copy of the i.i.d. entries $g_{ij}$ given in Theorem \ref{convergencespectral}. As we are in the real case $\beta=1,$ we have $g=\overline{g}.$

For a shape $\mathbf{s}=(s_1,\cdots,s_{2p})$, denote by $m(\mathbf{s})=\max_i s_i$ the total number of distinct vertices in it. (As the shape is independent of the labeling, we shall set $s_1=1$, and the next visit to a new vertex will have label 2, a further visit to a new vertex has label 3, etc. so $\max_i s_i$ is the total number of distinct vertices). Let $i_k:=\inf\{j\geq 1:s_j=k\}$ for all $1\leq k\leq m(\mathbf{s})$, so $i_k$ is the time when the $k$-th new vertex is first visited along the path. As the cycles are even, $(u_{i_{k-1}},u_{i_k})$ should be visited at least twice along the whole path, but $(u_{i_{k}-1},u_{i_k})$ should also be distinct from $(u_{i_\ell-1},u_{i_\ell})_{\ell<k}$ by definition of $i_k$.

In the following we \textbf{take the normalization $\sigma_*=1$}, so we have an upper bound of the right hand side of \eqref{tracemoment1s} by
$$
\sum_{\mathbf{u}\in\Gamma_{\mathbf{s},u}}b_{u_1u_2}\cdots b^*_{u_{2p}u_1}\leq \sum_{v_2\neq\cdots\neq v_{m(\mathbf{s})}\in[n]}b_{v_{s_{i_2}-1}v_2}^2b_{v_{s_{i_3}-1}v_3}^2\cdots (b^\cdot)^2_{v_{s_{i_{m(\mathbf{s})}-1}}v_{m(\mathbf{s})}}\leq \sigma^{2(m(\mathbf{s})-1)},
$$ where we denote by $(b^\cdot)_{u_iu_{i+1}}=b_{u_iu_{i+1}}$ if $i\leq p$ and  $(b^\cdot)_{u_iu_{i+1}}=b_{u_{i+1}u_{i}}$ if $i> p$, and where $v_{s_{i_2}-1}=u$. That is, we have proven
\begin{equation}\label{gaussiana}
\mathbb{E}[\operatorname{Tr}(A^p(A^*)^p)]\leq n\sum_{\mathbf{s}\in\mathcal{S}_{2p}}\sigma^{2(m(\mathbf{s})-1))}
\prod_{i\geq 1}\mathbb{E}[g^i]^{n_i(\mathbf{s})}.
\end{equation}

The case of a homogeneous variance can be easily computed. Let $Y_r$ be an $r\times r$ square matrix with i.i.d. entries following the same law as $g_{ij}$, then for any $r>p$,

\begin{equation}\label{gaussianb}
\mathbb{E}\operatorname{Tr}[Y_r^{p}(Y_r^*)^p]=r\sum_{\mathbf{s}\in \mathcal{S}_{2p}}(r-1)\cdots (r-m(\mathbf{s})+1)\prod_{i\geq 1}\mathbb{E}[g^i]^{n_i(\mathbf{s})}.
\end{equation}
This computation follows as the cycle $\mathbf{u}$ with a given shape $\mathbf{s(u)}$ is determined solely by its $m(\mathbf{s})$ free vertices.

Taking $r=\lceil \sigma^2\rceil+p$, from the elementary inequality and $p\geq m(\mathbf{s})$,
$$
(r-1)(r-2)\cdots (r-m(\mathbf{s})+1)\geq \sigma^{2(m(\mathbf{s})-1)}.
$$

Comparing \eqref{gaussiana} and \eqref{gaussianb} and substituting the value $r=\lceil\sigma^2+p\rceil$ into \eqref{gaussianb}, we immediately deduce that, by \eqref{gaussiana} and \eqref{gaussianb},
$$
\mathbb{E}\left[\operatorname{Tr}(A^p(A^*)^p)\right]\leq \frac{n}{\lceil \sigma^2\rceil+p}\mathbb{E}\operatorname{Tr}\left[Y_{\lceil \sigma^2\rceil+p}^p (Y_{\lceil \sigma^2\rceil+p}^*)^p\right].
$$
    Now we complete the proof of Theorem \ref{convergencespectral}. Applying Theorem \ref{mainresultveryhighbounds}, we get 
    $$\begin{aligned}
\mathbb{E}[\rho(A)]&\leq \mathbb{E}[\operatorname{Tr}|A^p(A^*)^p\|]^\frac{1}{2p}\leq n^\frac{1}{2p}\mathbb{E}\operatorname{Tr}\left[Y_{\lceil \sigma^2\rceil+p}^p (Y_{\lceil \sigma^2\rceil+p}^*)^p\right]^\frac{1}{2p}
\\&\leq n^\frac{1}{2p} (\lceil \sigma^2+p\rceil)^\frac{6}{2p}\exp(C\frac{p^{1/2}}{(\lceil\sigma^2\rceil+p)^{1/2}})(\lceil \sigma^2+p\rceil)^\frac{1}{2}.\end{aligned}$$
Now we take $p=\alpha\log n$ for some $\alpha>0$. By our assumption $\sigma^2/\sigma_*^2\gg \log n,$ and the normalization $\sigma_*=1$, we have $\sigma^2+p\leq n+p$  and 
$$
\mathbb{E}[\rho(A)]\leq (C(n+p))^{\frac{7}{2\alpha\log n}}p^\frac{2}{p}(\sigma+2\sqrt{\alpha\log n)}\exp(C\frac{\sqrt{\alpha\log n}}{\sqrt{\sigma^2+\alpha\log n}}). 
$$That is, for any $\epsilon>0$, we choose $\alpha$ such that $(2Cn)^\frac{4}{\alpha\log n}\leq \sqrt{1+\epsilon}$, $p^\frac{2}{p}\leq\sqrt{1+\epsilon}$ and get
$$
\mathbb{E}[\rho(A)]
\leq (1+\epsilon)(\sigma+\frac{6\sigma_*}{\sqrt{\log(1+\epsilon})}\sqrt{\log n})\exp\left(C\frac{\sqrt{6\frac{\log n}{\log(1+\epsilon)}}}{\sqrt{(\frac{\sigma}{\sigma_*})^2+6\frac{\log n}{\log(1+\epsilon)}}}\right).
$$

\textbf{Absence of outliers}
We can finally derive the tail bounds for $\rho(A)$ via Markov's inequality. First suppose that $\frac{\sigma}{\sigma_*}\gg\sqrt{\log n}$, and we again normalize by setting $\sigma_*=1$, then for any $\delta>0$,
$$\begin{aligned}
\mathbb{P}(\rho(A)\geq (1+\delta)\sigma)&\leq\mathbb{P}(\operatorname{Tr}(A^p(A^*)^p)\geq (1+\delta)^{2p}\sigma^{2p})\\&\leq Cn^6(1+\delta)^{-2p}\sigma^{-2p}(\sigma^2+p)^{p+2}\exp\left(C\frac{p^{1/2}}{(\lceil\sigma^2\rceil+p)^{1/2}}\right)^p 
.\end{aligned}$$By our choice $p=\alpha\log n$ and the assumption $\sigma^2\gg\log n$, it is easy to see the right hand side is summable in $n$. Hence by Borel-Cantelli lemma we have almost surely  $\rho(A)\leq(1+\delta)\sigma$ for all $n$ sufficiently large.

In the regime where $C_1\sqrt{\log n}\leq \frac{\sigma}{\sigma_*}\leq C_2\sqrt{\log n}$, again using Markov's inequality we can show there exists a constant ${D}=D(C_1,C_2)>0$ depending on $C_1,C_2$ such that 
$$\mathbb{P}(\lim\sup_{n\to\infty}\rho(A)\geq D(\sigma+\sqrt{\log n}
\sigma_*))=0.$$This statement can however be alternatively deduced from the elementary fact that $\rho(A)\leq\|A\|$ and the non-asymptotic bound on $\|A\|$ proved in \cite{bandeira2016sharp}.

\end{proof}

\begin{proof}[\proofname\ of Theorem \ref{convergencespectral}, the complex case $\beta=2$]The proof is essentially identical to the real $\beta=1$ case, but the extra assumption that $g_{ij}\sim g_{ij}\mathcal{U}$ for $\mathcal{U}$ the uniform distribution on the unit circle ensures that if a term $\prod_{i\in 2\mathbb{N}_+}\mathbb{E}[(gg^*)^{i/2}]^{n_i(\mathbf{s})}$ is nonzero (since $n_i(\mathbf{s})$ must be even in the complex case to get a nonzero expectation), then it must be a positive real number so the comparison procedure in the proof of real $\beta=1$ case still applies here. The details are omitted.
\end{proof}

\subsection{Small deviation bounds via long-time control}

In this section we prove Theorem \ref{Theorem1.6a12}. The main idea is to closely follow the proof of Theorem \ref{momenttheorem2.2}, \ref{nosymmetricdistribution}, but replace the coefficients $\frac{1}{\sqrt{n}}$ there by coefficients $b_{ij}$. The magnitude of $\sigma_*$ determines the largest range of $p$ where moment estimates are valid.

We first prove Theorem \ref{Theorem1.6a12} when the distribution of $x_{ij}$ are symmetric.
\begin{proof}[\proofname\ of Theorem \ref{Theorem1.6a12}, symmetric distribution]
 We use frequently the notations in the proof of Theorem \ref{momenttheorem2.2}. We compute the high trace moments $\mathbb{E}[\operatorname{Tr}[A^p(A^*)^p]$. For a typical term 
\begin{equation}\label{typicalexpansion1}
x_{i_0i_i}b_{i_0i_1}x_{i_1i_2}b_{i_1i_2}\cdots \bar{x}_{i_{p+1}i_p}b_{i_{p+1}i_p}
\cdots \bar{x}_{i_0i_{2p-1}}b_{i_0i_{2p-1}}
\end{equation} in the expansion, we claim that the leading contribution comes from vertices with distinct $i_0,\cdots,i_p$. Let $Z_1(p)$ be the sum of terms in $\operatorname{Tr}[A^p(A^*)^p]$ of the form \eqref{typicalexpansion1} such that $\operatorname{Card}|i_0,i_1,\cdots,i_p|=p+1$. We estimate its contribution by 
\begin{equation}\label{whatthezip}\begin{aligned}
\mathbb{E}[Z_1(p)]&\leq 
\sum_{i_0,\cdots,i_p} b_{i_0i_1}^2\cdots b_{i_{p}i_{p+1}}^2\\&
\leq \sum_{i_0=1}^n\sum_{i_1,\cdots,i_{p-1}}b^2_{i_0i_1}\cdots b^2_{i_{p-2}i_{p-1}}\sum_{i_p}b^2_{i_{p-1}i_p}
\leq \cdots\leq Cn\sigma^{2p},\end{aligned}\end{equation} where we first fix $i_0,\cdots,i_{p-1}$ and sum over $i_p$, then fix $i_0,\cdots,i_{p-2}$ and sum over $i_{p-1}$, etc. We
use that each $x_{i_sj_s}$ must appear exactly twice in the expansion, so that $\{i_1,\cdots,i_{p-1}\}=\{i_{p+1},\cdots,i_{2p-1}\}$ and that the entries $b_{ij}$ need to be squared, and we finally use that $S$ is long-time controlled by $\sigma$, see Definition \ref{longtimeshorttime}.

Then we show the remaining terms, summed together, is also $O(n\sigma^{2p})$ provided $p$ is not too large. We follow the main steps of the proof of Theorem \ref{momenttheorem2.2},

As in Step 1 of the proof of Theorem \ref{momenttheorem2.2}, we consider a sum of terms \eqref{typicalexpansion1} over all possible $p$ indices $i_1,\cdots,i_p$ from $[n]$ such that $n_1$ of these indices appear once, $n_2$ of them appear twice, and $n_k$ of them appear $k$ times. An ordering of their appearance in $[p]$ is also determined. Then as in Step 2 and Step 3 of the proof of Theorem \ref{momenttheorem2.2}, we construct the whole path $P$ from $i_1,\cdots,i_p$. Assume that such a configuration has been fixed (the vertices are not yet chosen), and we first consider the product of coefficients along the path.
\begin{equation}\label{sumproductpaths}b_{i_0i_1}b_{i_1i_2}\cdots b_{i_{p-1}i_p}b_{i_{p+1}i_p}\cdots b_{i_0i_{2p-1}}.\end{equation} By symmetry of entry distribution, to have a nonzero expectation, each $b_{ij}$ should appear an even number of times in this product. Since there are $2(\sum_{k\geq 2}kn_k)$ vertices appearing four times or more and $n_1$ vertices appearing twice, we can 
decompose the product \eqref{sumproductpaths} into the product of entries of $S$, along $J+1\geq 2$ (possibly disconnected) paths $P^0,P^1,\cdots,P^J$, with $J=\sum_{k\geq 2}(k-1)n_k$.

\textbf{The method of summing over coefficients.} The precise method to decompose the product \eqref{sumproductpaths} is as follows. Recall the notation $\mathcal{P}_k$ introduced in the proof of Theorem \ref{momenttheorem2.2}. For a given path $P$ and each $k$ we can alternatively understand the multiset $\mathcal{P}_k$ associated to $P$ as the collection of vertices in the path $P$ where the path visits $2k$ times. Also, each point in $\mathcal{P}_k$ is repeated with multiplicity $k$ in the multiset $\mathcal{P}_k$. (Note: for a path $i_0,\cdots,i_p,\cdots,i_{2p-1}$ we regard the middle value $i_p$ as visited twice).

 Starting from $i_0$ and walk through the path, at the first instance when we arrive at some index $i_{t_1}=x\in\cup_{k\geq 2}\mathcal{P}_k$ the second time, then we stop the path $P^1$ at $x$ with $x$ the last index of $P^1$. (If no such $t_1<p$ exists then we stop $P^1$ at $i_p$ and finish the construction: by the pairing rule of the random matrix this means all the vertices appear only twice and $P$ has no self-intersection). Then starting from $i_{t_1}$, if the edge $i_{t_1}\mapsto i_{t_1+1}$ has been visited an odd number of times before $i_{t_1}$ (so this visit is the 2,4,6, etc. th visit), then we skip this edge $i_{t_1}\mapsto i_{t_1+1}$ and move on to $i_{t_1+1}$. Otherwise, if this edge is visited an even number of times before, then we start $P^2$ at $i_{t_1}$, and proceed until we either (a) meet a new vertex in $\cup_{k\geq 2}\mathcal{P}_k$ for at least the second time counted by the visit of the original path, or (b) enter an edge $i_{t+1+r-1}\mapsto i_{t_1+r}$ that has been visited an odd number of times before, so this visit is the 2,4,6, etc. th visit. In case (a) we take this new vertex as the end of $P^2$ and in case (b) we take $i_{t_1+r-1}$ as the end of $P^2$.
 Inductively, suppose $P^r$ has been constructed with endpoint $i_{t_r}$, we find $s\in\mathbb{N}_+$ which is the smallest such that $i_{t_r+s-1}\mapsto i_{t_r+s}$ is visited an even number of times before $i_{t_r+s}$, and start $P^{r+1}$ at this index $i_{t_r+s-1}$ Proceed with $P^{r+1}$ all the way until we either meet a vertex in $\cup_{k\geq 2}\mathcal{P}_k$ (and visited at least once before in $P$) or meet a directed edge where this visit is the even number of time in $P$. Then stop $P^{r+1}$ at this vertex and proceed to define $P^{r+2}$. As an important note, for the second half of the path we are reversing the direction of edges as required by the non-Hermitian product rule, so the direction of the edge is $i_{p+s+1}\mapsto i_{p_s}$ for any $s\geq 0$.

For example, consider the path $7\mapsto 3\mapsto 5\mapsto 4\mapsto 3\mapsto 6\mapsto 3\mapsto 4\mapsto 5\mapsto 3\mapsto 7$ with a self intersection, then the path $P^1$ is $7\mapsto 3\mapsto 5\mapsto 4\mapsto 3$ and $P^2$ is $3\mapsto 6$ with the initial vertex of $P^2$ attached to $P^1$ (The second half of the edge repeats the first half, so they are not recorded). Consider a more complicated path $7\mapsto 3\mapsto 5\mapsto 4\mapsto 8\mapsto 5\mapsto 4\mapsto 8\mapsto 5\mapsto 6\mapsto 9\mapsto 6\mapsto 5\mapsto b\mapsto c\mapsto 5\mapsto b\mapsto c\mapsto 5\mapsto 3\mapsto 7$ (for some new integers $b,c$ unequal to $0,\cdots,9$) which has repeated loops, then $P^1$ is $7\mapsto 3\mapsto 5\mapsto 4\mapsto 8\mapsto 5$, and $P^2$ is $5\mapsto 6\mapsto 9$ and $P^3$ is $5\mapsto b\mapsto c\mapsto 5$. As a simple rule, suppose an index appears $x$ times in the path, then it will be the starting point of $x/2-1$ paths $P^i$ for $i\geq 2$. In the second example, 5 appears 6 times in the path and is the starting point of $P^2,P^3$.

By definition, we can prove that the starting point of each $P^\ell,\ell\geq 2$ has a one to one correspondence with a point in $\cup_{k\geq 2}\mathcal{P}_k$ in the path with one point removed from each multiset (this is because we do not consider the first visit, and do not consider the time when the edge is visited the even number of times, so we subtract by one here), and thus from a given path with fixed $n_1,n_2,\cdots$, this algorithm yields $J+1$  paths $P^1,\cdots, P^{J+1} $ where $J=\sum_{k\geq 2}(k-1)n_k$. The construction of $P^1,P^2,\cdots$ from the original path is unique. The endpoint of each $P^1,\cdots, P^J$ is fixed a-priori to repeat some earlier indices of the path, and the combinatorial factor of repetition is already determined in the $\frac{p!}{\prod_{k\geq 1}(k!)^{n_k}}$ factor of determining the exact order of appearance of an odd index. The final path $P^{J+1}$ has a fixed starting point, but the endpoint may or may not be fixed by a point in $\cup_{k\geq 2}\mathcal{P}_k$, see the above two examples. In this construction we only consider the odd visits of an directed edge, and these are marked by the odd indices and recorded in $\cup_{k\geq 1}\mathcal{P}_k$. The even visits (incidence of visiting an edge at an even number of time) are not considered in constructing $P^1,\cdots,P^{J+1}$ but we will compute the square of $b_{ij}$, i.e., $b_{ij}^2$, to account for even visits.

As in Theorem \ref{momenttheorem2.2} there are multiple choices of return when we visit an index $\cup_{k\geq 2}\mathcal{P}_k$ at an even instant, and this factor is absorbed in the combinatorial factor $(\text{const}\cdot k)^{kn_k}$. Finally, we have a factor $(\text{const}\cdot k)^{kn_k}$ to account for sub-Gaussian moments, or a factor $(\text{const}\cdot k)^{2kn_k}$ to account for sub-exponential moments but we can divide a factor $k!$ due to identical visits of an edge $k$ times, as in 
the proof of Theorem \ref{momenttheorem2.2}. 

We assume that each path $P^r$ has length $m_r\in\mathbb{N}$, and we denote by its vertices by $P^r=(P_0^r,P_1^r,\cdots, P_{m_r-1}^r,P_{m_r}^r)$. Since we keep the odd instance of the visit of each edge, we can check $\sum_{k\geq 1}m_k=p$.

 The sum over $P^r$ of coefficients $b_{ij}^2$ can be bounded very easily. Suppose its initial vertex is fixed to be $P_0^r$. Then we consider the contribution from summation of the coefficients $b_{ij}$, which is defined as the following expression for any $1\leq r\leq J$:  
\begin{equation}\label{line925fans}\begin{aligned}
&\sum_{P_1^r,\cdots,P_{m_r-1}^r\in[n],\quad P_{m_r}^r\text{ is fixed }}b^2_{P_0^rP_1^r}\cdots b^2_{P_{m_r-1}^rP_{m_r}^r}
\\&\leq (\sigma_*)^2\sum_{P_1^r,\cdots,P_{m_r-1}^r\in [n]}
b^2_{P_0^rP_1^r}\cdots b^2_{P_{m_r-2}^rP_{m_r-1}^r}
\leq C\sigma^{2m_r-2}(\sigma_*)^{2},\end{aligned}\end{equation}
where we simply bound $b^2_{P_{m_r-1}^rP_{m_r}^r}\leq(\sigma_*)^2$ and sum over other free indices $P_1^r,\cdots,P_{m_r-1}^r$. In the case $r=J+1$ we also bound, whether $P_{m_{J+1}}^{J+1}$ is fixed or not,
$$
\sum_{P_1^{J+1},\cdots,P_{m_{J+1}}^{J+1}\in[n]}b^2_{P_0^{J+1}P_1^{J+1}}\cdots b^2_{P_{m_{J+1}-1}^{J+1}P_{m_{J+1}}^{J+1}}
\leq C\sigma^{2m_{J+1}}.$$

\textbf{Comparison with the homogeneous model.}
Let $G$ be the (homogeneous variance profile) random matrix where we set $b_{ij}^2=\frac{1}{n}$ in the definition of $A$, we now wish to compare the computation of $\mathbb{E}[\operatorname{Tr}G^p(G^*)^p]$ to $\mathbb{E}[\operatorname{Tr}A^p(A^*)^p]$. This idea of comparison is similar to the idea in the proof of Theorem \ref{convergencespectral}, but with an essential difference: in this proof we wish to do the comparison in a pathwise manner, comparing the contribution to moments for each possible shape of the path, whereas Theorem \ref{convergencespectral} is essentially based on a global comparison of moments. We still use the notion of shape from the proof of Theorem \ref{convergencespectral}: for a path $\mathbf{u}=(u_1,\cdots,u_{2p})\in[n]^{2p}$ we denote by $\mathcal{S}(\mathbf{u})$ its shape. Then one can check that the shape of $\mathbf{u}$ is determined by both the numbers $n_1,\cdots,n_k$ and by the choice of sequence of appearance (the combinatorial factor $\frac{p!}{\prod_{k=2}^p (k!)^{n_k}}$). For each path $\mathbf{u}$, we can check that the expectation $\mathbb{E}[x_{u_1u_2}\cdots x_{u_{p}u_{p+1}}\bar{x}_{u_{p+2}u_{p+1}\cdots\bar{x}_{u_1u_{2p}}}]$ depends only on the shape $\mathcal{S}(\mathbf{u})$ of $\mathbf{u}$ and not on the specific choice of path, so the computation \eqref{line925fans}, after taking the product over $1\leq r\leq J+1$, gives the upper bound for the contribution of the inhomogeneous coefficients $b_{ij}$ over the summation of moments
to the homogeneous model. That is, for any fixed shape $\mathbf{s}\in \mathcal{S}_{2p}$, we have
$$\begin{aligned}&\sum_{\mathbf{u}:\mathcal{S}(\mathbf{u})=\mathbf{s}}\mathbb{E}[b_{u_1u_2}x_{u_1u_2}\cdots b_{u_{p}u_{p+1}}x_{u_{p}u_{p+1}}{b}_{u_{p+2}u_{p+1}}\bar{x}_{u_{p+2}u_{p+1}\cdots{b}_{u_1u_{2p}}\bar{x}_{u_1u_{2p}}}]\\&\leq C^{J+1}\sigma^{2p-2J}(\sigma_*)^{2J}\sum_{\mathbf{u}:\mathcal{S}(\mathbf{u})=\mathbf{s}}\mathbb{E}[x_{u_1u_2}\cdots x_{u_{p}u_{p+1}}\bar{x}_{u_{p+2}u_{p+1}\cdots\bar{x}_{u_1u_{2p}}}]\end{aligned}$$ where $J=\sum_{k\geq 2}(k-1)n_k$ is uniquely determined by $\mathbf{s}$.

\textbf{The final bound.} Now we can take the summation over all the shapes $\mathbf{s}\in\mathcal{S}_{2p}$ and upper bound in a similar way as in the case of equation \eqref{contributiongeneral}. The contribution to the expectation $\mathbb{E}[\operatorname{Tr}(A^p(A^*)^p)]$ by the summation of all the paths with specified $n_0,n_1,\cdots,n_p$ is upper bounded by 
\begin{equation}\label{***1234}
n\sigma^{2p}(\frac{\sigma_*}{\sigma})^{2\sum_{k\geq 2}(k-1)n_k}\frac{p!}{\prod_{k=2}^p(k!)^{n_k}}\prod_{k=2}^p(2k)^{kn_k}\prod_{k=2}^p(\text{const}\cdot k)^{kn_k} \cdot 4^{\sum_{k\geq 2}kn_k}.
\end{equation}
Summing the expression over all $0<\sum_{k\geq 2}kn_k$, the result is bounded from above by
$$
n\sigma^{2p}\left(\exp\left(\sum_{k\geq 2}\frac{(\text{const}\cdot k\cdot p)^k}{(\sigma^2/\sigma_*^2)^{k-1}}\right)-1\right)
$$ which is $O(1)n\sigma^{2p}$ whenever $p\leq C\sqrt{\sigma^2/\sigma_*^2}$ for some sufficiently small $C>0$.

Now we can conclude the proof of Theorem \ref{Theorem1.6a12} via the method of moments. Recalling that $\sigma=O(1)$, we have, for any $t>0$,
$$\begin{aligned}
\mathbb{P}(\rho(A)\geq \sigma(1+t
\sigma_*))&\leq \mathbb{P}(\operatorname{Tr}(A^p(A^*)^p) \geq \sigma^{2p}(1+t\sigma_*)^{2p})\\&\leq C_0n(1+t\sigma_*)^{-2p}\leq C_0ne^{-C\sigma t}
\end{aligned}$$
where we take $p=\lfloor\frac{C\sigma}
{\sigma_*}\rfloor$ and $C>0$ is a sufficiently small constant defined in the previous paragraph, depending on $\sigma$ and the sub-exponential distribution of $x_{ij}$. The constant $C_0>0$ is another fixed constant depending only on $\sigma$ and the sub-exponential tail \eqref{gaussianmoments} of $x_{ij}$. 
\end{proof}

The case of nonsymmetrically distributed entries can be handled in exactly the same way, where we modify the proof of Theorem \ref{nosymmetricdistribution} in Appendix \ref{appendixC}.

\begin{proof}[\proofname\ of Theorem \ref{Theorem1.6a12}, non-symmetric distribution]
Checking the proof of Theorem \ref{nosymmetricdistribution} in Appendix \ref{appendixC}, we only need to show that contributions of uneven paths (paths with at least one edge traveled an odd number of times) can be bounded by $O(n\sigma^{2p})$. This is easy: since each $b_{ij}$ is bounded by $\sigma_*$, we simply replace each factor $n^{-\frac{1}{2}}$
of the additional edge weight contribution in the proof of Theorem \ref{nosymmetricdistribution} by the factor $\sigma_*$, and choose $p\leq C'\sigma_*^{-1}$ for another fixed constant $C'>0$. The remaining computations are analogous. 

\end{proof}

\subsection{Small deviation bounds via spectral radius}

The proof is very similar to that of Theorem \ref{Theorem1.6a12}, but we take a different method to estimate the coefficient matrix $S$.

\begin{proof}[\proofname\ of Theorem \ref{spectralradiuslargedeviation}]
As in the proof of Theorem \ref{Theorem1.6a12}, we first assume that the entries $x_{ij}$ have a symmetric distribution. 
The idea of the proof is essentially the same as the proof of Theorem \ref{Theorem1.6a12}, but we need to take a different estimate on the variance profile $S$. After a rescaling we may assume for simplicity that $\rho(S)\geq 1$.

From the standard linear algebraic relation $\operatorname{Tr}(S^k)\leq n\rho(S)^k$ we expand 
$$
\sum_{i_0,\cdots,i_k\in[n]}s_{i_0i_1}s_{i_1i_2}\cdots s_{i_{k-1}i_k}s_{i_ki_0}\leq n\rho(S)^{k+1}.
$$ From the assumption $s_{i_ki_0}\geq\frac{c}{n}$, we get an obvious upper bound
\begin{equation}\label{definitionsk} S(k):=
\sum_{i_0,\cdots,i_k\in[n]}s_{i_0i_1}s_{i_1i_2}\cdots s_{i_{k-1}i_k}\leq c^{-1}n^2\rho(S)^{k+1}.
\end{equation}
This bound on $S(k)$ already gives the contribution to $\mathbb{E}[\operatorname{Tr}(A^k(A^*)^k)]$ generated by even closed paths with no self intersection, and the contribution is $O(n^2\rho(S)^{p+1})$ as one may compute similarly to \eqref{whatthezip}.

For general paths with self intersection, we may decompose the path into $J+1$ sub-paths $P^1,\cdots,P^{J+1}$ as in the proof of Theorem \ref{Theorem1.6a12}. Then we need to bound, for any fixed $1\leq r\leq J+1$ and such that the path $P^r$ has length $m_r$, the following summation
\begin{equation}\label{summationidentity2}
\sum_{p_s^r\in[n],1\leq s\leq m_r-1\forall r,P_0^1,p_{m_{J+1}}^{J+1}\in[n]}
\prod_{r=1}^{J+1}
\prod_{s=0}^{m_r-1}b^2_{P_s^rP_{s+1}^r},
\end{equation} (that is, we sum over the possible choices of the vertices of each path $P^j$ with the initial vertex and last vertex removed, and also sum over the first vertex of $P^1$ and the last vertex of $P^{J+1}$), under two constraints: 
\begin{enumerate}\item
We require that the endpoint of $P^r,1\leq r\leq J$ satisfy that $P_{m_r}^r=P_{m_t(s)}^t$
for some fixed $t\leq r$ and for some fixed $m_t(s)\in[0,m_t)\cap\mathbb{N}$. This is because by the way we construct the paths $P^1,\cdots,P^{J+1}$, the endpoint of $P^r$ is either some point in $
\cup_{k\geq 2}\mathcal{P}_k$ which is not the first visit, (and the first visit must be contained in the previously constructed paths), or the endpoint is a vertex of a previously existed directed edge, which also is contained in previously constructed paths.

\item A previously fixed correspondence between the starting point $P^r_0$ of the path $P^r,r=2,\cdots,J+1$ with $\cup_{k\geq 2}(k-1)\mathcal{P}_k$ (which is the multiset spanned by each element in $\mathcal{P}_k$ repeated $k-1$ times). \end{enumerate}
Both constraints are determined by the shape $\mathbf{s}$ of the path.

We will now show that \eqref{summationidentity2} can be bounded by a factor of the form $n^{-J}S(p-J)$, where $S(\cdot)$ was defined in \eqref{definitionsk} , and such that this bound is independent from the specific form of the constraints imposed. To see this, we take an iterative exploration of the paths by first revealing vertices in $P^1$, then vertices in $P^2$, then $P^3$, etc. Note that $P_{{m_{r}-1}}^r$ is a free vertex to be chosen in $[n]$ and $P_{{1}}^{r+1}$ is also a free vertex to choose, but the value of $P_{m_r}^r$ and $P^{r+1}_0$ have already been fixed. By the two-sided bound \eqref{twosidedbounds} on the variance $s_{ij}$, we have 
\begin{equation}\label{howlargeisproduct}
b^2_{P^r_{m_{r}-1},P^r_{m_r}}b^2_{P^{r+1}_0,P^{r+1}_1}\in [\frac{c^2}{Cn},\frac{C^2}{cn}]b^2_{P^r_{m_{r}-1},P^{r+1}_{1}}.
\end{equation}
This reduction removes one constraint previously set, but keeps the number of free vertices unchanged and brings about an additional $n^{-1}$ factor. This procedure can be applied altogether $J$ times, at the instance when we switch from paths $P^j$ to $P^{j+1}$ for each $j$. This leads to a reduction of $J$ edges.
After the reduction we have
\begin{equation}\label{intheformof}
\eqref{summationidentity2}\leq (\frac{\text{const}}{n})^{\sum_{k\geq 2}(k-1)n_k}S(p-J)
.\end{equation}
Therefore in this case the contribution of all terms with specified $n_0,n_1,\cdots,n_p$ to the trace expansion $\mathbb{E}[\operatorname{Tr}(A^p(A^*)^p)]$ is upper bounded by
\begin{equation}\label{radiusspect**1234}
n^2\rho(S)^{p+1}
(\frac{\text{const}}{n})^{\sum_{k\geq 2}(k-1)n_k}\frac{p!}{\prod_{k=2}^p(k!)^{n_k}}\prod_{k=2}^p(2k)^{kn_k}\prod_{k=2}^p(\text{const}\cdot k)^{kn_k} \cdot 4^{\sum_{k\geq 2}kn_k}.
\end{equation}
The summation over all $n_0,\cdots,n_p$ with $\sum_{k\geq 2}kn_k>0$ is bounded by 
$$
n^2\rho(S)^{p+1}
\left(\exp\left(\sum_{k\geq 2}\frac{(\text{const}k\cdot p)^k}{n^{k-1}}\right)-1\right)
$$ which is $O(1)n^2\rho(S)^{p+1}$ whenever $p\leq C_0\sqrt{n}$ for some fixed $C_0>0$ sufficiently small.

Finally, the contribution of terms with $n_0=n-p,n_1=p,n_2=\cdots=0$ is bounded by 
$$
c^{-1}n^2\rho(S)^{p+1}.
$$
 Combining the above, we can finish the proof of the small deviation estimate via a standard application of Markov's inequality, which parallels the argument in the proof of Theorem \ref{Theorem1.6a12} and is thus omitted.

Last we consider the case when $x_{ij}$ does not have a symmetric law. For this we just need to follow the proof of Theorem \ref{nosymmetricdistribution}, where we replace the edge weight from $n^{-1/2}$ to $Cn^{-1/2}$. There is also a change when we are in case (B) or (C) of the proof of Theorem \ref{nosymmetricdistribution}, when the path with an odd edge is generated from at least two even closed paths and some additional odd edges. For the contribution of coefficients from these even closed paths, we may just follow the computation in \eqref{summationidentity2} to stick together the contribution from two disjoint paths into contribution along one single path, as in \eqref{intheformof}, and finally conclude via this estimate.

\end{proof}

\subsection{Suppressed fluctuations: proofs}

This section is devoted to the proof of Theorem \ref{examplegrowthbounds}. The idea is again to compute the trace of high powers of $\mathcal{L}$.

\begin{proof}[\proofname\ of Theorem \ref{examplegrowthbounds}]
   Expanding the trace power of $\mathbb{E}[\operatorname{Tr}(\mathcal{L}^{p_n}(\mathcal{L}^*)^{p_n})]$, thanks to the special structure of $\mathcal{L}$, any monomial in the expansion should be of the form $$
X^1_{i_0i_1}X^2_{i_1i_2}\cdots X^{p_n}_{i_{p_{n}-1}i_{p_n}}\bar{X}^{p_n}_{i_{p_{n}+1}i_{p_n}}\cdots \bar{X}^1_{i_0i_{2p_n-1}},
    $$ where we use the notation $X^i_{j,k}$ to denote the $(j,k)$-th element of the matrix $X^i$.

    Since entries of $\mathcal{L}$ are independent with zero mean, for this term to have nonzero expectation we need each $X^r_{i_{r-1}i_{r}}$ to appear at least twice in the product. But as we consider exactly $p_n$ moments, this guarantees each element must appear exactly twice and therefore $i_{p_n+r}=i_{p_n-r}$ for all $r=1,2,\cdots,p_n-1$.

    Summing over all possibilities of $i_1,\cdots,i_{p_n}$, we get  
    $$
\mathbb{E}[\operatorname{Tr}\left(\mathcal{L}^{P_n}(\mathcal{L}^*)^{p_n}\right)]\leq q_n.
    $$
    Thus by Markov's inequality, whenever $p_n$ is large enough, for any $t>0$
$$\begin{aligned}
\mathbb{P}(\rho(\mathcal{L})\geq 1+\frac{t}{p_n})&\leq (1+\frac{t}{p_n})^{-2p_n}\mathbb{E}[\|\mathcal{L}^{p_n}\|^2]\\&\leq (1+\frac{t}{p_n})^{-2p_n}\mathbb{E}[\operatorname{Tr}(\mathcal{L}^{p_n}(\mathcal{L}^*)^{p_n}\|)]\leq e^{-2t}q_n.
\end{aligned}$$
    This proves claim (2) of Theorem \ref{examplegrowthbounds}. For claim (1), it suffices to take $e^{-2\frac{n}{q_n}}q_n\ll 1$ to ensure $\rho(\mathcal{L})\leq 1+o(1)$ with high probability: this is the same as $q_n\log q_n\ll n$.
\end{proof}

\section{Large deviation of spectral radius}

In this section we prove Theorem \ref{theorem1.5subgaussianconcentration}. Instead of using the moments computation, we will use the free probability framework of \cite{bandeira2023matrix}.
The basic idea is, to show that $\rho(A)\leq 1+t$, we first recall that $\|A\|\leq 4$ with very high probability, then we find small balls $B_i=B(\rho_i,\epsilon_i)$ to cover $\{z\in\mathbb{C}:1+t\leq |z|\leq 4\}$ and show that $\sigma_{min}(A-\rho_i I)>0$ with very high probability for each $i$. Finally we take a union bound over all such balls $B_i$.

Before outlining the proof we recall some free probability notations from \cite{bandeira2023matrix}. Consider a Gaussian random matrix $M\in\mathcal{M}_n(\mathbb{C})$ that can be written in the form 
$$
M=A_0+\sum_{i=1}^m A_ig_i
$$ where $A_0,A_1,\cdots,A_m\in \mathcal{M}_n(\mathbb{C})$ are fixed and $g_1,\cdots,g_m$ are an independent family of standard real Gaussian random variables. The idea is to show that the spectrum of the finite dimensional object $M$ is close to the spectrum of a limiting object, denoted $M_{\text{free}}$, that lives in an infinite dimensional space.

To properly define this limiting object $M_{\text{free}}$, we need to introduce some notations in free probability theory. A comprehensive introduction to free probability can be found in \cite{nica2006lectures}. Consider a unital $C^*$ algebra $(\mathcal{A},\tau)$ and a family of free semicircular elements $s_1,\cdots,s_m$ on $(\mathcal{A},\tau)$, where $(\mathcal{A},\tau)$ is the space where the the infinite dimensional objects lie in, and the semicircular elements $s_i$ can be thought as the limiting objects of GOE/GUE matrices in free probability theory. Then the free probability model $M_{\text{free}}$ of $M$ is defined on the $C^*$ algebra $(\mathcal{A}\otimes \mathcal{M}_n(\mathbb{C}),\tau\otimes\operatorname{tr})$ as 
\begin{equation}\label{line1007}
M_{\text{free}}=A_0\otimes 1+\sum_{i=1}^m A_i\otimes s_i.
\end{equation}Note that the definition of $M_{\text{free}}$ and $M$ applies to non-Hermitian matrices as well. For a random matrix $M$ with general (non-Gaussian) entries, we can consider its Gaussian model $M^{\text{gauss}}$ which is the Gaussian random matrix with the same mean and covariance profile as $M$, where the covariance profile is the following $n^2\times n^2$ matrix 
$$
(\operatorname{Cov}M)_{ij,kl}:=\mathbb{E}[(M-\mathbb{E}M)_{ij}\overline{(M-\mathbb{E}M)_{kl}}],
$$
and define its free probability version $M_{\text{free}}$ as the free model associated with $M^{\text{gauss}}$. While there are several ways to write out $M_{\text{free}}$ in terms of entries of $M$, they are equivalent in the following computations.

For a self-adjoint operator $M$ we denote by $\operatorname{Sp}(M)\subset\mathbb{R}$ the spectrum of $M$.
The main result in \cite{bandeira2023matrix} implies that, under appropriate conditions on the coefficients of $M$, the spectrum $\operatorname{Sp}(M)$ can be well captured by $\operatorname{Sp}(M_{\text{free}})$. We will not use the most general notation of \cite{bandeira2023matrix}  but adapt \cite{bandeira2023matrix}, Theorem 2.1 to the setting of Theorem \ref{theorem1.5subgaussianconcentration}.

\begin{Proposition}\label{proposition4.1first} Let $A$ be a random matrix with Gaussian entries defined in Theorem \ref{theorem1.5subgaussianconcentration}. For any $z\in\mathbb{C},$ we define $$\mathcal{Y}_z:=\begin{pmatrix}0&A+z\mathbf{1}\\A^*+\bar{z}\mathbf{1}&0\end{pmatrix},\quad \mathcal{Y}_{\text{free},z}:=\begin{pmatrix}0&A_{\text{free}}+z\mathbf{1}\\(A^*)_{\text{free}}+\bar{z}\mathbf{1}&0\end{pmatrix}.$$
   Then for some universal constant $C>0$, we have for all $t>0$ and $z\in\mathbb{C}$, 
$$\mathbb{P}\left(\operatorname{Sp}(\mathcal{Y}_z)\subset \operatorname{Sp}(\mathcal{Y}_{\text{free},z})+C\{\sigma_*t+\sqrt{\sigma_*}(\log n)^{\frac{3}{4}}\}[-1,1]\right)\geq 1-e^{-t^2}.$$
\end{Proposition}

\begin{proof} We shall use notations $\sigma(A),\sigma_*(A),\widetilde{v}(A)$ from \cite{bandeira2023matrix}.
The claim follows from applying \cite{bandeira2023matrix}, Theorem 2.1 with $\sigma_*(A)=\sigma_*$ and $\tilde{v}(A)\leq C\sqrt{\sigma_*}$, and the fact that the assumptions in Theorem \ref{theorem1.5subgaussianconcentration} implies $\sigma(A)\leq 3$.
\end{proof}

\begin{proof}[\proofname\ of Theorem \ref{theorem1.5subgaussianconcentration}.] 

We divide the proof into three regimes, that $t<\frac{1}{3},t\in[\frac{1}{3},3]$ and $t>3$.

\textbf{The large $t$ regime.} We can check that, say, $\|A_{\text{free}}\|\leq 4$ via Lehner's formula \cite{lehner1999computing} or Pisier's formula (see \cite{bandeira2023matrix}, Lemma 2.5), so that for any $t\geq 3$, by \cite{bandeira2023matrix}, Corollary 2.2, \begin{equation}\label{opertornorm}\mathbb{P}(\rho(A)\geq t+1)\leq \mathbb{P}(\|A\|\geq t+1)\leq e^{-C_0\frac{(t-2)^2}{\sigma_*^2}}\leq e^{-C_0\frac{t^2}{\sigma_*^2}},\end{equation}
where $C_0>0$ is a universal constant that changes from line to line. More specifically, the parameter $\sigma_*$ in this paper coincides with the parameter $\sigma_*(X)$ in \cite{bandeira2023matrix} by \cite{bandeira2023matrix}, Lemma 3.1, and we also use the assumption $(\log n)^\frac{3}{4}\sqrt{\sigma_*}\to 0$. 

\textbf{The small $t$ regime.}
Then we consider $0\leq t\leq \frac{1}{3}$. Consider a finite covering of 
$$\{z\in\mathbb{C}:1+\frac{t}{4}\leq |z|\leq 4\}$$ by finitely many balls of radius $\frac{t^{1.5}}{400}$: $$(B_i)_{i\in\mathcal{A}}=(B(\rho_i,\frac{t^{1.5}}{400}))_{i\in\mathcal{A}},$$ where $\mathcal{A}$ is a finite index set and $1+\frac{t}{4}\leq|\rho_i|\leq 4$ for each $i$. 

In the following, for any $z\in\mathbb{C}$ denote $A_z:=A-z\mathbf{1}$, where $\mathbf{1}$ is the identity matrix. Then we define $A_{\text{free},z}$ as the free probability model associated to $A_z$ as defined in \eqref{line1007}.

We claim that \begin{equation}\label{gaussianboundss}\sigma_{min}(A_{\text{free},\rho_i})\geq \frac{t^{1.5}}{100} \text{ for every } i\in\mathcal{A}.\end{equation} [The reason why we choose $t^{1.5}$ instead of $t$, and hence get the exponent $\exp(-t^3/\sigma_*^2)$ instead of $\exp(-t^2/\sigma_*^2)$, all arise from this estimate.] Once this is shown, we have that whenever $n$ satisfies $\sqrt{\sigma_*}(\log n)^\frac{3}{4}\leq \frac{t^{1.5}}{500}$, we can find a universal $C_0>0$ such that
$$\begin{aligned}
\mathbb{P}(A\text{ has an eigenvalue in } B_i)&\leq \mathbb{P}(\sigma_{min}(A_{\rho_i})\leq\frac{t^{1.5}}{200})\\&\leq \mathbb{P}(|\sigma_{min}(A_{\rho_i})-\sigma_{min}(A_{\text{free},\rho_i})|\geq\frac{t^{1.5}}{200})\\&\leq  e^{-C_0\frac{t^3}{\sigma_*^2}}
\end{aligned}$$ 
where the last inequality follows from Proposition \ref{proposition4.1first} (1) and our assumption on the smallness of $\sigma_*$. The first inequality follows from the fact that $\sigma_{min}(A_z)$ is Lipschitz continuous in $z$, i.e. for any $\lambda\in B_{i}$ we have $\sigma_{min}(A_\lambda )\geq \sigma_{min}(A_{
\rho_i})-|\rho_i-\lambda|$, so that $A$ has eigenvalue $\lambda\in B_i$ implies $\sigma_{min}(A_{\rho_i})\leq\frac{t^{1.5}}{200}.$
    Then we complete the proof via a union bound:
    \begin{equation}\label{howdoesunionwork}\begin{aligned}
\mathbb{P}(\rho(A)\geq t+1)&\leq \mathbb{P}(\|A\|\geq t+1)1_{t\geq 4}
+\sum_{i\in\mathcal{A}}\mathbb{P}(A\text{ has an eigenvalue in }B_i)    \\&\leq (1+|\mathcal{A}|)e^{-C_0\frac{t^3}{\sigma_*^2}}
\leq C_0(1+t^{-3})e^{-C_0\frac{t^3}{\sigma_*^2}}\end{aligned}\end{equation} where we can bound the entropy cost by $|\mathcal{A}|\leq C_0t^{-3}$ for some universal constant $C_0$, and this $C_0$ can be absorbed into the constant $C_0$ in the exponential term.
This completes the proof of Theorem \ref{theorem1.5subgaussianconcentration} in the Gaussian case, for all $t\leq \frac{1}{3}$. In fact, the value $\frac{1}{3}$ here is arbitrary, not used in the above computation (but used in the proof of \eqref{gaussianboundss}), and can be changed to any other small constant in $(0,1)$.

\textbf{Intermediate range of $t$.}
When $\frac{1}{3}\leq t\leq 3$, we can find a covering of by finitely many balls of radius $\frac{1}{6}$ and use the same argument. The entropy cost of covering is absorbed in the universal constant $C_0$.

\textbf{Verification for the claim on least singular values.}
We finally prove \eqref{gaussianboundss}. A similar computation has been done in \cite{han2024outliers}
but here we need a much more precise version. Define the resolvent of the free operator$$
G_{\text{free},z}(v):=((\tau\otimes \operatorname{id})\mathcal{Y}_{\text{free,z}}-v\mathbf{1})^{-1})
$$ for every $v\in\mathbb{C}_+:=\{z\in\mathbb{C}:\Im z>0\}$. We first assume that $\mathbb{E}[a_{ii}^2]=0$ for each $i\in[n]$. Then we solve the matrix Dyson equation for $G_{\text{free},z}$.
By our assumption, $G_{\text{free},z}(v)$ solves, by \cite{han2024outliers}, Section 3.1, the self consistency equation
\begin{equation}\label{matrixdysomequationsfag}
\mathbb{E}[\mathcal{Y}_0G_{\text{free},z}(v)\mathcal{Y}_0]+G_{\text{free},z}(v)^{-1}+\begin{pmatrix} v I_n\quad z I_n\\\bar{z}I_n\quad v I_n
\end{pmatrix}=0.\end{equation}
    The existence of a unique solution to \eqref{matrixdysomequationsfag} such that $\operatorname{Im}G_{\text{free},z}(v)>0$ is proven in \cite{helton2007operator}, and we claim that the solution is necessarily given by 

    $$G_{\text{free},z}(v)=\begin{pmatrix}
    a(z,v) \mathbf{1}& b(z,v)\mathbf{1}\\ \bar{b}(z,v)\mathbf{1}&a(z,v)\mathbf{1}
\end{pmatrix},$$  where $a(z,v),b(z,v)$ are scalar functions of $(z,v)$. To see this, plugging this expression back into the self-consistency equation \eqref{matrixdysomequationsfag} and multiplying by $\mathcal{G}_{\text{free},z}(v)$ to the left, we see that we are reduced to the following $2\times 2$ system of scalar equations:

    \begin{equation}\label{expandedfunction} \begin{pmatrix}a&b\\\bar{b}&a\end{pmatrix}\cdot \begin{pmatrix}
 a& 0
\\0 &a\end{pmatrix}+\begin{pmatrix}1&0\\0&1\end{pmatrix}
+\begin{pmatrix}a&b\\\bar{b}&a\end{pmatrix}\cdot\begin{pmatrix}
 v &z\\\bar{z}&v  \end{pmatrix}
 =0.   \end{equation} Solving the self-consistency equation \eqref{expandedfunction} we get 
$$
b=\frac{-az}{a+v}.
$$Taking this further back to \eqref{expandedfunction} we get a cubic equation for $a=a(z,v)$:
\begin{equation}\label{cubicsupports}
-\frac{1}{a}=a+v-\frac{|z|^2}{a+v}.
\end{equation} To determine the support of $\mathcal{Y}_z$, we will fix $|z|>1$, take $\Im v\to 0$ although we do not set $\Re v\to 0$. By a well-known criterion (see implication (vi)-(v) in Lemma D.1 of \cite{alt2020dyson}), we have
\begin{equation}\label{spectruminclusion}
\{t\in\mathbb{R}:\lim_{\eta\downarrow 0} \Im a(z,t+i\eta)
=0\}\subset\mathbb{C}\setminus\operatorname{Sp}(\mathcal{Y}_{\text{free},z}).\end{equation}

Now we fix $\tau:=|z|^2>1$ and rewrite the cubic equation \eqref{cubicsupports} as 
\begin{equation}\label{equationsolvedbya} \tau-1-v^2=a^2+2av+\frac{v}{a}.
\end{equation} We now assume $v$ is real and $v\in[-\frac{1}{3},\frac{1}{3}]$, and formulate a condition on $v$ which guarantees all three solutions to $a$ are real. By continuous dependence of solutions on the coefficients (or solution formula for cubic polynomials), we have that if $\Re v$ satisfies the given condition, then $\Im a(z,v)\to 0$ whenever $\Im v\to 0$. We can reformulate this implication as follows:
\begin{equation}\label{inclusiontype2}
\{v\in\mathbb{R}:\text{all three solutions $a(z,v)$ to \eqref{equationsolvedbya} are real\}}\subset\mathcal{C}\setminus\operatorname{Sp}(\mathcal{Y}_{\text{free},z}).
\end{equation}

Consider the function with real entries 
$$f(x)=x^2+2vx+\frac{v}{x},$$
computing $f'(x)=0$ yields $2x+2v-\frac{v}{x^2}=0$. We assume, say, $|v|<\frac{1}{3}$, then whenever $|x|>2 |v|^{1/3}$ or $|x|<\frac{|v|^{1/3}}{4}$ then $f'(x)\neq 0$, so all critical points of $f$ lie in $\{\theta v^{1/3}:\frac{1}{4}\leq|\theta|\leq 2\}$. Drawing the plots of $f$ we see that all the local minima of $f$ are smaller than $9|v|^{2/3}$, so that for any $\zeta>9|v|^{2/3}$, the line $y=\zeta$ has three intersection points with the plot of $f$ for fixed $v$. Thus whenever $\zeta>9|v|^{2/3}$, the equation $f(x)=\zeta$ has three real solutions for this fixed value of $v$. (Note that we are checking the inclusion \eqref{inclusiontype2} for each fixed $v\in\mathbb{R}$, and thus although $f$ depends on $v$, we may fix $v$ and plot the function $f$ for this value of $v$ ). Thus by the inclusion \eqref{spectruminclusion} one has
$$
\{t\in\mathbb{R}:|t|\leq\frac{1}{3},\quad 1+t^2+9|t|^{2/3}\leq|z|^2\}\subset\mathbb{C}\setminus\operatorname{Sp}(\mathcal{Y}_{\text{free},z}).
$$ That is, 
$$\sigma_{min}(\mathcal{Y}_{\text{free},z})\geq\min\left(\frac{1}{3},\left(\frac{|z|^2-1}{10}\right)^{3/2}\right).
$$ This bound is worse than the bound $\sigma_{min}(\mathcal{Y}_{\text{free},z})\geq |z|-1$ that one may expect, but is nonetheless workable. This estimate verifies \eqref{gaussianboundss} for $t\leq\frac{1}{3}$.
\end{proof}

\section{The heavy-tailed case}\label{sections5}
This section is devoted to the proof of Theorem \ref{upperboundsecondmoment}. We will use critically the arguments in \cite{WOS:000435416700013}, and take a further averaging step and careful counting.

We begin with some graph-theoretic notations from \cite{WOS:000435416700013}.

For each $n\in\mathbb{N}_+$ let $[n]:=\{1,2,\cdots,n\}$. A directed graph on $[n]$ is $G=(V,E)$ with $V\subset[n]$ the vertices set and $E\subset [n]\times [n]$ the set of directed edges. Edges in $E$ can appear multiple times, and we say $G$ is a \textbf{multi digraph}.
For each $m\in\mathbb{N}_+$, a path $P$ of length $m$ is some sequence $P=(i_1,\cdots,i_{m+1})\in [n]^{m+1}$, and we say $P$ is closed if it has coinciding start and end vertices, i.e., $i_1=i_{m+1}$. Each path $P$ generates a multigraph $G_P=(V,E)$ where $(i,j)\in E$ with multiplicity $n$ if the directed edge $(i,j)$ appears $n$ times in $P$.

Given a multi digraph, we call it a \textbf{double cycle} if the graph is formed by repeating a cycle twice. For a path $P$, we call it an \textbf{even path} if every adjacent pair $(i\to j)$ appears an even number of times. A multi digraph generated by an even path (by the method described in the last paragraph) is called an \textbf{even digraph}. Two multi-digraphs $G=(V,E)$, $G'=(V',E')$ on $[n]$ are \textbf{isomorphic} if there is a bijection $f:V\to V'$ such that $(i,j)\in E$ if and only if $(f(i),f(j))\in E'$, and the two edges have the same multiplicity in $G,G'$ respectively.

A \textbf{rooted digraph} is a digraph $G=(V,E,\rho)$ where $\rho\in E$ is a distinguished directed edge. Two rooted digraphs $G=(V,E,\rho),G'=(V',E',\rho')$ are \textbf{isomorphic} if there is an isomorphism preserving the multiplicity of edges and mapping $\rho$ to $\rho'$.

Given an even digraph $G=(V,E),$ we define its weight associated to the random matrix $A$ in Theorem \ref{upperboundsecondmoment} as 
$$
p(G):=\prod_{(i,j)\in E}|x_{ij}|^{n_{i,j}}, 
$$ where $n_{ij}\geq 2$ is the multiplicity of edge $(i,j)\in E$.

For a given unlabeled even digraph $\mathcal{U}$, we denote by $\{G: G\sim \mathcal{U}\}$ the equivalence class of even digraphs to $\mathcal{U}$, where the equivalence relation is modulo the isomorphism relation defined in the previous paragraph.

 We define the following statistical quantity, for each $h\in\mathbb{N}_+$:
\begin{equation}\label{definesu}
\mathcal{S}_h(\mathcal{U}):=\frac{2^h|\{G\sim\mathcal{U}:p(G)\geq 2^h\}|}{|\{G:G\sim \mathcal{U}\}|}
\end{equation} and 
\begin{equation}\label{defineshu}\mathcal{S}(\mathcal{U}):=\max(1,\max_{h\geq 0}\mathcal{S}_h(\mathcal{U})).
\end{equation} For rooted even digraphs $G=(V,E,\rho)$ we similarly define 
$$ p_r(G)=\prod_{(i,j)\in E} |x_{ij}|^{n_{ij}-2\mathbf{1}_{(i,j)=\rho}}.
$$
In the definition of $p(G),p_r(G)$ we have not included the coefficients $\{b_{ij}\}_{1\leq i,j\leq n}$. 

Let $\mathcal{C}_m$ denote a double cycle with $2m$ edges on $[n]$, and let $\mathcal{C}_m^*$ denote a rooted double cycle with $2m$ edges on $[n]$. Recall the following lemma on cycle statistics from \cite{WOS:000435416700013}, Lemma 3.1. For each $k\geq 1$ define $\mathcal{A}_k:=\mathcal{A}_k^1\cap\mathcal{A}_k^2\cap\mathcal{A}_k^3$, with 
$$
\mathcal{A}_k^1:=\cap_{m=1}^k \left\{\sum_{h=0}^\infty \mathcal{S}_h(\mathcal{C}_m)\leq k^2\right\},
$$
$$
\mathcal{A}_k^2:=\cap_{m=1}^k \left\{\sum_{h=0}^\infty \mathcal{S}_h(\mathcal{C}_m^*)\leq k^2\right\},
$$
$$
\mathcal{A}_k^3:=\cap_{m=1}^k \left\{\sum_{h=0}^\infty 2^{h\epsilon/2}\mathcal{S}_h(\mathcal{C}_m)\leq k^2B^m\right\}.
$$

\begin{lemma}(\cite{WOS:000435416700013}, Lemma 3.1)
    For every $k\geq 1$ we have $\mathbb{P}(\mathcal{A}_k)\geq 1-\frac{6}{k}$.

\end{lemma}

\begin{Proposition}\label{howdouweaga45.}(Main estimate, (\cite{WOS:000435416700013}, Proposition 3.3)
    Consider $n^{\epsilon/16}\geq 5ek^2$ and an unlabeled rooted even graph $\mathcal{U}$ having $2k$ edges and $x$ vertices. Set 
    $$
y_x:=\max\left(0,k-x-\frac{4k\log B}{\epsilon \log n}\right)
    ,$$ then on the event $\mathcal{A}_k$,
    $$
\mathcal{S}(\mathcal{U})\leq n^{k-x-\epsilon y_x/16} k^2(3ek^2)^\frac{4k\log B}{\epsilon\log n}.
    $$
\end{Proposition}

We also need a few combinatorial counting estimates. 
\begin{lemma}\label{lemma60678}
Let $\mathcal{UG}(k,\ell)$ be the isomorphism class of unlabeled
rooted digraphs with $k$ vertices and $2\ell$ edges. Then 
$$
|\mathcal{UG}(k,\ell)|\leq k^{2(k-\ell)+1}.
$$ \end{lemma}

\begin{proof}
    This is essentially \cite{WOS:000435416700013}, Lemma 2.3 but we do not label the vertices in $[n]$.
\end{proof}

\begin{lemma}\label{greatestlemma5.4}[\cite{WOS:000435416700013}, Lemma 2.2]
    Given $G=(V,E)$ an even digraph, $|E|=2k$, $|V|=\ell$. There are at most $\ell(4k-4\ell)!$ number of paths that generate $G$.
\end{lemma}

Now we are ready to prove Theorem \ref{upperboundsecondmoment}. We follow the lines of \cite{WOS:000435416700013} to a large extent, yet the main novelty of our proof compared to \cite{WOS:000435416700013} is that we add an additional permutation $\pi$ to the indices of $A_n$ in the estimate \eqref{mainproductestimate2} to take a further averaging, which is helpful when $\{b_{ij}\}$ do not have the same value.

\begin{proof}[\proofname\ of Theorem \ref{upperboundsecondmoment}, (1)] We work on the event $\mathcal{B}$ where $|x_{ij}|\leq n^2$ for all $(i,j)\in [n]^2$. By Markov, $\mathbb{P}(\mathcal{B})\geq 1-1/n^2$. Let $\mathcal{E}_k:=\mathcal{A}_k\cap\mathcal{B}$, then $\mathbb{P}(\mathcal{E}_k)\geq 1-n^{-2}-6k^{-1}$. 

\textbf{Averaging over a random permutation for heavy-tail entries.}
At this point we need to take one more averaging procedure. let $\pi\in S(n)$ be uniformly chosen permutation of $[n]$, and denote by $\pi A$ the matrix with entries $(b_{ij}x_{\pi(i),\pi(j)})_{i,j\in [n]}$, that is, we permute the entries of $\{x_{ij}\}$ via a uniformly chosen permutation $\pi$ but keep the coefficients $b_{ij}$ fixed at location $(i,j)$. Since $x_{ij}$ are i.i.d. entries we have $A\overset{\text{law}}{\equiv}\pi A$ for any $\pi\in S_n$. Moreover, from the definition of $\mathcal{E}_k$ we readily see that $\mathcal{E}_k$ is invariant under $\pi$, that is, $A\in\mathcal{E}_k$ if and only if $\pi A\in\mathcal{E}_k$. We use $\mathbb{E}_\pi$ to denote the expectation taken with respect to a uniformly chosen $\pi\in S_n$. For an arbitrary path $P_1=(i_0,i_1,\cdots,i_r)\in [n]^r$, the permutation $\pi$ naturally acts on $P_1$ and maps it to $\pi P_1=(\pi(i_0),\pi(i_1),\cdots,\pi(i_{r-1}),\pi(i_r))$, so that $\pi P_1$ permutes the coordinates of $P_1$. For this path $P_1$, we define its weight as 
$$
w(P_1)=\prod_{t=1}^r x_{i_{t-1}i_t}.
$$ Then $w(\pi P_1)$ is the weight of the path $\pi P_1$.

From this, we estimate the conditional moment of $\rho(A_n)$ via 

\begin{equation}\label{mainproductestimate2}\begin{aligned}
    \mathbb{E}[\rho(A_n)^{2k-2}\mid\mathcal{E}_k]&=
    \mathbb{E}_\pi \mathbb{E}[\rho(\pi A_n)^{2k-2}\mid\mathcal{E}_k]
     \\&\leq\sum_{i,j}\sum_{P_1,P_2:i\mapsto j} \mathbb{E}_\pi[\mathbb{E}[w(\pi P_1)\bar{w}(\pi P_2)V(P_1)\bar{V}(P_2)\mid\mathcal{E}_k]],
\end{aligned}\end{equation}
where the sum is over paths $P_1,P_2$ from $i$ to $j$ with length $k-1$ and $V(P_1),V(P_2)$ denote the product of entries of $b_{ij}$ along path $P_1$ (resp. $P_2$): that is, for $P_1$ is $(i_0,i_1,\cdots,i_r)$, then $$V(P_1):=\prod_{t=1}^{r}b_{i_{t-1},i_t}.$$

\textbf{Reduction to even digraphs}. 
Conditioning on $\mathcal{E}_k$ destroys independence of $x_{ij}$'s, but the sign of each $x_{ij}$ remain independent Bernoulli distributions because $x_{ij}$ has a symmetric distribution. Thus $$\mathbb{E}[w(\pi P_1)\bar{w}(\pi P_2)|\mathcal{E}_k]=0$$ unless each edge in $\pi P_1\cup \pi P_2$ has even multiplicity (equivalently, each edge in $P_1\cup P_2$ has even multiplicity).

    Let $P$ be the closed path starting from $i$, going through $P_1$, then take edge $(j\mapsto i)$, then follow $P_2$, and end again with $(j\mapsto i)$, so that $P$ is even closed of length $2k$. We use the notation $V_r(P)$ to denote the product of entries of $b_{ij}$ along the path $P$ but the multiple of entries on the root edge $r$ is reduced by multiplicity two: This is because we set $(i,j)$ the root and it contributes to no increase in the moments.

    Recall that for a path $P$, $G_P$ denotes the multigraph generated by $P$ (see the beginning of Section \ref{sections5}). Then we use the notation $V_r(G_P)$ to denote the product of entries $b_{ij}$ with $(i,j)$ ranging over directed edges in $G_P$, but the root is counted with multiplicity reduced by two. By injectivity of $(P_1,P_2)\to P$, we get 
    \begin{equation}\label{thisexpressions}
     \mathbb{E}[\rho(A_n)^{2k-2}|\mathcal{E}_k]\leq\sum_P\mathbb{E}_\pi\left[\mathbb{E}[p_r(\pi G_p)V_r(G_p)]|\mathcal{E}_k\right],
    \end{equation}
with the summation over all even closed paths of length $2k$: $P=(i_1,\cdots,i_{2k+1})$ with $(i_k,i_{k+1})$ the root. Via Lemma \ref{lemma60678} and \ref{greatestlemma5.4} the right hand side of \eqref{thisexpressions} is bounded by 
    \begin{equation}\label{whatisboundedby?}
        k\sum_{x=1}^k (4k)^{4(k-x)}\sum_{G\in\mathcal{G}_n(k,x)}\mathbb{E}_\pi[\mathbb{E}[|p_r(\pi G)V_r(G)|\mid\mathcal{E}_k]],
    \end{equation} with $\mathcal{G}_n(k,x)$ the set of rooted even directed graphs on $[n]$ having $2k$ edges and $x$ vertices.

\textbf{Upper bound the weights on heavy-tail edges.}
On the event $\mathcal{E}_k$ we have $\operatorname{p}_r(G)\leq n^{4k-4}$, so setting $H:=\lfloor 4k\log_2 n\rfloor$ we have 
\begin{equation}\label{decomposeprg} p_r(G)\leq 1+2\sum_{h=0}^H 2^h 1_{p_r(G)\geq 2^h}\text{ on }\mathcal{E}_k.
\end{equation}

For any fixed $G\in \mathcal{G}_n(k,x)$ we have on the event $\mathcal{E}_k$, $$\mathbb{E}_\pi[\mathbb{E}[p_r(\pi G)]|\mathcal{E}_k]\leq 3H\mathcal{S}(G)\leq 3H n^{k-x}n^{-\epsilon y_x/16}k^2(3ek^2)^\frac{4k\log B}{\epsilon \log n},$$
where the first equality is because $\{\pi G,\pi\in S_n\}$ exhausts all even digraphs on $[n]$ that are isomorphic to $G$,
and we recall the definition of $\mathcal{S}(\mathcal{U})$ and $\mathcal{S}_h(\mathcal{U})$ from \eqref{definesu}, \eqref{defineshu}: these two quantities are the average over the weights on paths isomorphic to the given path, and our expectation taken over $\pi$ also generates the uniform measure over the isomorphism class and thus gives $\mathcal{S}(\mathcal{U})$. The second inequality follows from Proposition \ref{howdouweaga45.}.   

\textbf{Contributions from the coefficients $\{b_{ij}\}$}. 
    Next we compute $\sum_{G\in\mathcal{G}_n(k,x)}V_r(G)$: this is the major technical ingredient in the current proof that we add upon to the proof in \cite{WOS:000435416700013}. For each $G\in\mathcal{G}_n(k,x)$ we choose a doubled path $P$ spanning $G$ (the choice of this path does not affect weight counting), where as each edge in $P$ is doubled, we can replace $P$ by a (possibly disconnected) path $\bar{P}$ starting from the same starting point of $P$ but where each edge visited at the even number of times is removed (so only edges visited for the first, third, fifth, etc. times are kept). Although $\bar{P}$ may be disconnected, it is uniquely determined by $P$ (we denote this by $P\mapsto\bar{P}$) and has $x$ vertices and $k$ edges. By Lemma \ref{lemma60678}, the isomorphism class of $\bar{P}$ has cardinality bounded by $k^{2(k-x)+1}$. Fix a candidate $\bar{P}$. We first assume $\bar{P}$ has no multiple edges (so the original path $P$ has no edges visited four times or more, and obviously these paths make the dominant contribution), and consider each connected component of $\bar{P}$ separately. Evidently if $k=x$ then $\bar{P}$ is simply a directed path with no self-intersection. Let $\bar{P}$ be $1\mapsto 2\mapsto\cdots s_1=k_1\mapsto k_1+1\cdots\mapsto s_2=k_2\mapsto\cdots k$ be a representation of $P$ where $s_1$ is the first self-intersection of $\bar{P}$ so that $s_1$ equals one of $1,\cdots,s_1-1$ (denoted $t_1$) and $s_2$ is the second self-intersection so $s_2=t_2$ for some $t_2<s_2$, etc. There are altogether $k-x$ self-intersections of $\bar{p}$, denoted $s_1,\cdots,s_{k-x}$:
    $$\sum_{G\in\mathcal{G}_n(k,x):\quad P\mapsto\bar{P}}V_r(G)=\sum_{n_0,\cdots,n_k\in [n]:\quad n_{s_i}=n_{t_i}\text{for each }i=1,\cdots,k-x} b^2_{n_0n_1}\cdots b^2_{n_{k-1}n_k}.
    $$ Since $|s_{ij}|\leq\frac{C}{n}$ by assumption, the above sum can be bounded by 
$$ 
\left(\frac{C}{n}\right)^{k-x}\sum_{n_i\in[n]\text{distinct},\quad i\in [k]\setminus \{s_1,\cdots,s_{k-x}\}} b^2_{n_0n_1}\cdots b^2_{n_{k-1}{n_k}}\leq \left(\frac{C}{n}\right)^{k-x} \sigma^{x},
$$ where we use the definition of $\sigma$ in \eqref{maximalrowcolumn}.

For disconnected $\bar{P}$ we treat each component individually, and for $\bar{P}$ with multiple edges we may first remove the multiple edges and estimate as above, and finally use $s_{ij}\leq\frac{C}{n}$ to see such terms are further negligible compared to cases where $\bar{P}$ has no repeated edges. Taking this back into \eqref{thisexpressions}, \eqref{whatisboundedby?} and using Lemma \ref{lemma60678} we see that 
$$\begin{aligned}
  \mathbb{E}[\rho(A_n)^{2k-2}|\mathcal{E}_k]&\leq 3H\sum_{x=1}^k  k^4(4k)^{6(k-x)}\sigma^xC^{k-x}n^{-\epsilon y_x/16}(3ek^2)^\frac{4k\log B}{\epsilon\log n}.
\end{aligned}
$$

\textbf{The eventual upper bound.} Taking $k\sim (\log n)^2,$ so that whenever $x\leq k-\frac{8k\log B}{\epsilon\log n}$ then $y_x\geq (k-x)/2$ and $$(4k)^{6(k-x)}(C\sigma^{-1})^{k-x}N^{-\epsilon y_x/16}\leq 1,$$
the following estimate also holds: $$\sum_{x=1}^k (4k)^{6(k-x)}(C\sigma^{-1})^{k-x}n^{-\sigma y_x/16}\leq k+\frac{8k\log B}{\epsilon\log n} \left(4k\max\left(\frac{C}{\sigma},\frac{\sigma}{C}\right)\right)^\frac{48k\log B}{\epsilon\log n}.$$ To conclude, we have shown, for some $C'$ depending only on  $C,\sigma$,
 $$
\mathbb{E}[(\rho(A_n)^{2k-2}|\mathcal{E}_k]\leq \sigma^k (\log n)^{C'\log n}
 ,$$ and the result follows from an application of Markov's inequality. 
\end{proof}

\begin{proof}[\proofname\ of Theorem \ref{upperboundsecondmoment}, (2)] For this we essentially follow the proof in case (1), but we use the relation \eqref{howlargeisproduct} thanks to the assumption $\frac{c}{n}\leq b_{ij}^2\leq\frac{C}{n}$ to reduce the number of constraints and, as a result, we only need to upper bound, as in \eqref{intheformof}, for any fixed sequence $s_1\leq \cdots\leq s_{k-x}\leq k$ and any fixed sequence $t_i\leq s_i,i=1,\cdots,k-x$:

$$\begin{aligned}\sum_{G\in\mathcal{G}_n(k,x):\quad P\mapsto\bar{P}}V_r(G)&=\sum_{n_0,\cdots,n_k\in [n]:\quad n_{s_i}=n_{t_i}\text{for each }i=1,\cdots,k-x} b^2_{n_0n_1}\cdots b^2_{n_{k-1}n_k}\\&\leq (\frac{\text{const}}{n})^{k-x} S(x),
\end{aligned} $$
    where we recall that 
    $$
S(x)=\sum_{i_0,\cdots,i_x\in[n]}b^2_{i_0i_1}b^2_{i_1i_2}\cdots b^2_{i_{x-1}i_x}\leq c^{-1}n^2\rho(S)^{x+1}.
    $$ The rest of the computation is almost identical to the proof of case (1).
\end{proof}

\appendix

\section{Unbounded spectral radius}\label{appendix1}
We now state and prove a theorem mentioned in the introduction, that $\sigma_*\ll(\log n)^{-1/2}$ is optimal for convergence of $\rho(A)$.
The statement and proof of this theorem is motivated by Theorem 1.3 of \cite{altschuler2024spectral} who considered a very similar problem for symmetric random matrices.

\begin{theorem}\label{unboundedspectralradius}(Unbounded spectral radius)
    Let $E$ be an $n\times n$ random matrix which has a block diagonal form. The main diagonal blocks are $\frac{n}{d_n}$ diagonal blocks $E_1,\cdots,E_{\frac{n}{d_n}}$ of size $d_n\times d_n$ each, and all other entries of $E$ are zero. We assume each diagonal block is independently distributed and follow the distribution of a complex Ginibre ensemble (i.e. each $E_i$ is a $d_n\times d_n$ square matrix with i.i.d. entries having distribution $\frac{1}{\sqrt{d_n}}\mathcal{N}_\mathbb{C}(0,1)$). Assume that $n/d_n$ is an integer. Then
    \begin{enumerate}
        \item   If   $d_n=o({\log n}),$ then for any $a>1$, $$\mathbb{P}(\lim\inf_{n\to\infty}\rho(E)\geq a)=1.$$
    \item If $d_n=O({\log n})$, then we have $$\mathbb{P}(\lim\inf_{n\to\infty}\rho(E)>1)=1.$$
     \end{enumerate}
\end{theorem}

 The proof of Theorem \ref{unboundedspectralradius} will rely on the following estimate on the complex Ginibre ensemble which has explicit eigenvalue correlation functions:

\begin{Proposition}\label{ginirbeboundasd}
    Let $G_n=(g_{ij})$ be sampled from the $n$-dimensional complex Ginibre ensemble, i.e. $G$ has independent entries, $g_{ij}\sim\frac{1}{\sqrt{n}}\xi$ and $\xi\sim \mathcal{N}_\mathbb{C}(0,1)$. Then
    \begin{enumerate}   
   \item For any $a>1$ we can find a constant $p_a>0$ such that whenever $n$ is sufficiently large,
    $$
\mathbb{P}(\rho(G_n)\geq a)
\geq \exp(-p_an).    $$
\item For any small $\epsilon>0$ we can find some $a_\epsilon>1$ such that whenever $n$ is sufficiently large, 
$$\mathbb{P}(\rho(G_n)\geq a_\epsilon)\geq \exp(-\epsilon n).
$$

\end{enumerate}\end{Proposition}
    
\begin{proof}
    We use the exact eigenvalue density function of complex Ginibre ensemble to deduce, by \cite{MR1986426}, equation (3), for any $a>0$:
    \begin{equation}\label{cramerbounds}
\mathbb{P}(\rho(G_n)\leq a)=\prod_{k=0}^{n-1}\mathbb{P}\left[\frac{1}{n}\sum_{\ell=1}^{n-k}X_\ell\leq a^2\right]
    ,\end{equation} where $\{X_k\}$ is a given sequence of independent rate one exponential random variables.

To prove part (1), for $k\geq \frac{n}{2}$ we trivially upper bound the probability of each factor on the right-hand side of \eqref{cramerbounds} by $1$. For $0\leq k\leq \frac{n}{2}$, we bound 
    $$
\mathbb{P}\left[\frac{1}{n}\sum_{\ell=1}^{n-k}X_\ell\leq a^2\right]\leq\mathbb{P}\left[\frac{1}{n-k}\sum_{\ell=1}^{n-k}X_\ell\leq 2a^2
\right]
    $$ We will apply Cramér's theorem in large deviations to the sum $$\bar{S}_{n-k}:=\frac{1}{n-k}\sum_{\ell=1}^{n-k}X_\ell,$$ where Cramér's theorem (see for example \cite{stroock2012introduction} or any standard large deviation textbook) shows that, for any $x>1$ we can find $c_x=x-1-\log x>0$ such that  
    $$\begin{aligned}
\lim_{n\to\infty}\frac{1}{n}\log\mathbb{P}(\bar{S}_n\geq x)&=-\sup_{t\geq 0}\{tx-\log\mathbb{E}[\exp(tX_1)] \}=-c_x<0.
\end{aligned} 
    $$ That is, whenever $n$ sufficiently large we may assume $\frac{1}{n}\log\mathbb{P}(\bar{S}_n\geq x)\geq -2c_x$. That is, whenever $n$ sufficiently large, 
    $$\begin{aligned}
\mathbb{P}(\rho(G_n)\leq a)&\leq \prod_{k=0}^{\frac{n}{2}}\mathbb{P}[\bar{S}_{n-k}\leq 2a^2]\\&\leq\prod_{k=0}^{\frac{n}{2}} (1-e^{-2c_{2a^2}(n-k)})\leq (1-e^{-2c_{2a^2}n})^\frac{n}{2}\leq 1-C_0 n e^{-2c_{2a^2}n}\end{aligned} $$
for some universal constant $C_0$. This completes the prof of part (1).

The proof of part (2) is similar, except that we need to keep track of the constants more carefully. 
For the given $\epsilon>0$ we find a constant $a_\epsilon>1$ such that $$c_{a_\epsilon^2}=(a_\epsilon)^2-1-\log ((a_{\epsilon})^2)<\frac{\epsilon}{10}.$$ 
Let $\theta_\epsilon:=1-(\frac{1+a_\epsilon}{2a_\epsilon})^2$. Then for $k\geq \theta_\epsilon n$, we bound each term on the right hand side of \eqref{cramerbounds} simply by $1$. For the smaller $k$, we bound as follows: using equation \eqref{cramerbounds},
whenever $n$ is sufficiently large, 
 $$\begin{aligned}
\mathbb{P}(\rho(G_n)\leq \frac{1+a_\epsilon}{2})&\leq \prod_{k=0}^{\theta_\epsilon n}\mathbb{P}[\bar{S}_{n-k}\leq (a_\epsilon)^2]\leq\prod_{k=0}^{\theta_\epsilon n} (1-e^{-2c_{a_\epsilon^2}(n-k)})\\&\leq (1-e^{-\frac{\epsilon}{5}n})^{\theta_\epsilon n}\leq 1-C_0\theta_\epsilon ne^{-\frac{\epsilon}{5}n}\end{aligned} $$
for some universal $C_0>0$. This completes the proof of part (2): it suffices to replace $\epsilon $ by $\frac{\epsilon}{5}$ and $a_\epsilon$ by $\frac{1+a_\epsilon}{2}$.

\end{proof}

    Now we can conclude the proof of Theorem \ref{unboundedspectralradius}.
    \begin{proof}[\proofname\ of Theorem \ref{unboundedspectralradius}]
        As $E$ is block diagonal, we have, whenever $n$ is sufficiently large and using Proposition \ref{ginirbeboundasd}(1), for any $a>0$:
$$\begin{aligned}\mathbb{P}((\rho(E)\leq a))&=\prod_{i=1}^{\frac{n}{d_n}}\mathbb{P}(\rho(E_i)\leq a)
        \end{aligned}.$$

First assume we are in case (1), i.e., $d_n=o(\log n)$, then we have $p_a d_n\leq \frac{\log n}{2}$ for any constant $p_a>0$ whenever $n$ is large. Then
$$
\mathbb{P}(\rho(E)\leq a)\leq (1-n^{-1/2})^\frac{n}{d_n}\leq e^{-n^{1/4}}. 
$$ Case (1) then follows from applying Borel-Cantelli lemma.

Then we assume we are in case (2), so that $d_n\leq C\log n$ for some $C>0$. Choose $\epsilon>0$ small enough so that $C\epsilon\leq \frac{1}{4}$, then by Proposition \ref{ginirbeboundasd} (2), we can find some $a_\epsilon>1$ such that  
$$
\mathbb{P}(\rho(E)\leq a_\epsilon)\leq (1-n^{-1/2})^\frac{n}{d_n}\leq e^{-n^{1/4}}.
  $$ The claim again follows from Borel-Cantelli lemma.      \end{proof}

\section{Further examples}
In this section we outline some further examples left out in the main text, to illustrate how the long-time control in Definition \ref{longtimeshorttime} can improve upon the quantity $\sigma$ in capturing $\rho(A)$, both on the $O(1)$ scale and on the small-deviation scale. The following example is notably more general than that in Example \ref{example206}.

\begin{example}\label{example12}
(Perturbed band matrix: improved upper bound for spectral radius)    
 Fix $c\in(0,1)$.
Consider an $n\times n$ matrix $S':=S+\Lambda$, where $S=\frac{1}{\lfloor n^c\rfloor}\operatorname{Adj}(G)$ and $\operatorname{Adj}(G)$ is the adjacency matrix of a $\lfloor n^c\rfloor$-regular directed graph $G$ on $n$ vertices.

Let $\Lambda$ be an $n\times n$ matrix with nonnegative entries, having at most $n^d$ $(d\leq c)$ nonzero entries $\Lambda_1,\cdots,\Lambda_{n^d},$ and we assume $0<\Lambda_i\leq Dn^{-c}$ for some fixed constant $D>0$ and for each $i=1,\cdots,n^d$.
Then the largest row sum of $S$ satisfies
$$
\sup_i\sum_j [S']_{ij}\leq 1+Dn^{-d-c}
,$$
and equality holds with the largest row sum equal to $1+Dn^{d-c}\leq 1+D$ in the extreme case where each $\Lambda_i$ lies on the same row and $\Lambda_i=Dn^{-c}$ for each $i$.

Let us assume that all the entries $\Lambda_i$ lie on the 1-th row of $S'$, and that none of the $\Lambda_i$'s lie on the diagonal entry of $S'$. In this case, we have the computation: for any fixed $i_0\in [n]\setminus\{1\}$:
\begin{equation}\label{adhoccomputation}\begin{aligned}\sum_{i_1,\cdots,i_k\in [n]}&(S')_{i_0i_1}(S')_{i_1i_2}\cdots (S')_{i_{k-1}i_k}\\&\leq 
\sum_{i=0}^k C_k^i (Dn^{-c})^in^{-c(k-i)}(n^c)^{k-2i}n^{di}\leq (1+Dn^{d-2c})^k.
\end{aligned}\end{equation}
where the label $i$ labels the number of terms in the expansion of $(S')^k$ coming from entries in $\Lambda$, $C_k^i$ is the combinatorial factor of the choice in $[k]$ for terms come from $\Lambda$. The term $n^{di}$ arises as there are at most $n^d$ elements in $\Lambda$, so we have at most $n^d$ choices of column label for an element from $\Lambda$. Finally the $(n^c)^{k-2i}$ factor arises because we have fixed $i$ elements from $\Lambda$, but the fact that each $\Lambda_i$ lies on the same row and none of the $\Lambda_i$ is on the diagonal imply that they cannot be consecutive terms in the product of $(S')^k$, so we have already fixed $2i$ free vertices. (The assumption that $i_0\neq 1$ ensures that all terms from $\Lambda'$ appear as the form $(S')_{i_ri_{r+1}}$ for some $r\geq 1$, so we have a loss of two free vertices.). Thus, there are only $k-2i$ remaining free vertices to be chosen, each having at most $n^c$ choices thanks to the regular digraph structure imposed on $S$. This also implies $i\leq k/2$, but we still take the sum over $i\leq k$ as each term is non-negative. 

From this computation, and the fact that none of the $\Lambda_i$ lies on the diagonal, we see that 
\begin{equation}\label{sumsumasagae}\sum_{i_1\cdots,i_k\in [n]} (S')_{1i_1}\cdots (S')_{i_{k-1}i_k}\leq (1+Dn^{d-c})(1+Dn^{d-2c})^{k-1}\leq (1+D)(1+Dn^{d-2c})^{k-1}.
\end{equation} A similar computation also applies to the long-time control of $(S^t)^k$.
From the above computations 
\eqref{adhoccomputation}, \eqref{sumsumasagae} we see that $1+Dn^{d-2c}$, instead of $1+Dn^{d-c}$, should be the better estimate for $\rho(X)$. 

When $d=c$, by taking $D$ large we create a family of variance profiles with arbitrarily large row sum but still $\rho(X)\leq 1+o(1)$ thanks to Theorem \ref{Theorem1.6a12}, as well as a small deviation estimate $\rho(X)\leq (1+Dn^{-c})(1+\frac{\log n}{n^{c/2}})$ with very high probability. When $d<c$, our refined bound on $\sigma$ also provides much better small deviation bound for $\rho(X)$ when $d$ is very close to $c$ and $c$ is close to 1. 
\end{example}

Then we provide an example showing the optimality of condition \eqref{howdoesupperbounds}.
\begin{example}\label{examplebandmatrix}

Consider $d_n=n^{1-c}$ for some $c>0$, and $A=(a_{ij})$ where $a_{ij}\sim \frac{1}{\sqrt{d_n}}g$ for a standard Gaussian variable $g$ whenever $0<|i-j|<d_n$, $a_{ij}=0$ when $|i-j|>d_n$, and $a_{ii}=\frac{1}{\sqrt{d_n}}\xi$, where $\xi$ is a symmetric, mean 0, variance 1 random variable with $\mathbb{P}(|\xi|^\frac{1-c}{2}\geq t)\sim t^{-1}$. Denote by $A_o$ the matrix from $A$ setting all diagonal entries 0, and $A_d$ the diagonal part of $A$. By the main result of this paper, we have $\rho(A_o)\leq 1+o(1)$ with high probability. However, we also have $\|A_d\|>1$ with non-vanishing probability, and since $A_d$ is diagonal, the eigenvalues of $A_d$ away from $[-1,1]$ can be proven to converge to a Poisson process in the large $n$ limit. We can write $A=A_o+A_d$ and regard $A$ as a (growing) rank perturbation of the matrix $A_o$. Recently \cite{han2024outliers} characterized the outliers of perturbed non-Hermitian band matrices, and (although that result does not immediately apply here as $A_d$ has growing rank, we expect an analogous result can be proven with additional effort), we can formally deduce that outlying eigenvalues of $A_0+A_d$ correspond with eigenvalues of $A_d$ outside $[-1,1]$, and hence with non-vanishing probability, $\rho(A)\to 1+o(1)$ does not hold.

\end{example}

\section{Non-symmetric entry distributions}\label{appendixC}

In this section we prove Theorem \ref{nosymmetricdistribution} on non-symmetric distributions. The notations and constructions in the proof will be used repeatedly in the proof of Theorem \ref{Theorem1.6a12} for non-symmetric entry distributions.

\begin{proof}[\proofname\ of Theorem \ref{nosymmetricdistribution}] We again consider the summation of monomial expansions $$\mathbb{E}[x_{i_0i_1}x_{i_1i_2}\cdots x_{i_{p-1}i_p}\bar{x}_{i_{p+1}i_p}\cdots\bar{x}_{i_0i_{2p-1}}]$$ over a closed path $\mathcal{P}=(i_0,i_1,i_2,\cdots i_p,i_{p+1},\cdots i_{2p-1},i_0)$ of length $2p$. Since $x_{ij}$ are centered each directed edge $(i,j)$ must exist at least twice in $\mathcal{P}$ in the sense specified in the proof of Theorem \ref{momenttheorem2.2}, but as $x_{ij}$ may not be symmetric, the edge $i\mapsto j$ can appear an odd number of times: if this is the case then we call this edge an odd directed edge. Since $2p$ is even, the number of odd directed edges must be even.
    We will adapt the argument in \cite{peche2007wigner} for Wigner matrices, to define a gluing procedure which takes as an input a path $\mathcal{P}$ with length $2p$ and $2\ell$ odd edges, and outputs a math $\mathcal{P}'$ of length $2p-2\ell$ having the same set of edges as $\mathcal{P}$ except that the last time of occurrence of each odd edge in $\mathcal{P}$ is removed. Thus each directed edge in $\mathcal{P}'$ is traveled an even number of times, but $\mathcal{P}$ is not necessarily a closed path: it can be disjoint union of connected paths.

For such an input math $\mathcal{P}$, the moment of the last occurrence of the odd edge is a subset of $\{1,2,\cdots,2p\}$, which can be written in a union of $J$ ($1\leq J\leq 2\ell$) disjoint intervals on $\mathbb{Z}$. Then the set of odd edges is split into $J$ subsequences $S_i,i=1,\cdots,J$. We denote by $e_i$ ($f_i$) the left (right) endpoints of $S_i$ and $f_0=e_{J+1}=i_0$ with $i_0$ the initial point of $\mathcal{P}$.   Finally, for $i=0,\cdots,J$ let $\mathcal{P}_i$ be the subpath starting at $f_i$ and ending at $e_{i+1}$. 

From these subpaths $\mathcal{P}_i$ we construct a graph $G$ with vertices $\{e_i,f_i:i=0,\cdots,J\}$, where an edge is present from $v_i$ to $v_j$ if there is a path $\mathcal{P}_k$ admitting $v_i,v_j$ as endpoints. We denote $\mathcal{L}:=\{e_i,f_i;0\leq i\leq J\}$. Let $I'\leq J$ be the number of connected components of $G$. From this graph we can reconstruct the paths as follows: first read $\mathcal{P}_0$ until we reach the endpoint $e_1$, then choose a subpath $\mathcal{P}_{i_1}$ with $e_1$ as endpoint: we read $\mathcal{P}_{i_1}$ in the forward direction if $e_1$ is the left endpoint of $\mathcal{P}_{i_1}$ and in the backward direction otherwise. Then we find $\mathcal{P}_{i_2}$ whose endpoint coincides with the right endpoint of $\mathcal{P}_{0}\cup\mathcal{P}_{i_1}$, until the constructed subpath $\mathcal{P}_0\cup\mathcal{P}_{i_1}\cdots\cup\mathcal{P}_{i_k}$ is a closed path. Then if $i_0$ is the endpoint of some $\mathcal{P}_j$ which hasn't been glued we glue $\mathcal{P}_j$ in the direction where $i_0$ is the starting point, and restart the gluing procedure. After this, we find the first path $\mathcal{P}_i$ not glued before (whose left end point must have occurred in $\mathcal{P}_0\cup\cdots\cup\mathcal{P}_{i_k}$), and we continue the gluing procedure starting from $\mathcal{P}_{i}$. This produces a sequence of $I_0\geq I'$ closed paths $\widetilde{W}_i,0\leq i\leq I_0-1$ with origins $v_{i_j}\in\mathcal{L},j=1,\cdots,I_0-1,v_{i_0}=i_0.$ Then we reconstruct the path $\mathcal{P}'$ from the paths $\widetilde{W}_i$: we can order the $\widetilde{W}_i$ in an arbitrary way so long as we start from $\mathcal{P}_0,$ read all the edges with origin $v_1$ and concatenate them to get the path $W_1$, etc. Then we get a sequence of paths $W_1,\cdots,W_{I-1}$ and finally define $\mathcal{P}'$ by concatenating the paths $W_i,0\leq i\leq I-1$, where $\mathcal{P}'$ is possibly disconnected. The order of reading $W_i's$ in $\mathcal{P}'$ will not be relevant: what matters is that the origin of $W_i$ is a marked vertex of $\mathcal{P}'$ and pairwise distinct.

Then there are three possible cases for the path $\mathcal{P}$: (A), the gluing yields a \textit{genuine} closed even path $\mathcal{P}'$ (a closed path where each directed edge is traveled an even number of times); (B), the gluing yields a sequence of $I\geq 2$ closed even paths with respective origins $\{i_0,v_i,1\leq i\leq I-1\},$ each $v_i$ an odd index of $\mathcal{P}'$ (recall the proof of Theorem \ref{momenttheorem2.2} for definition of an odd index, which is an analogue of marked vertex in \cite{peche2007wigner}); or (C) the gluing yields $I\geq 2$ paths, with some of them having an odd directed edge.

We now estimate the contribution of case (A) and (B), and show that case (C) is controlled by case (B) or (A). For this we need to recall the insertion procedure defined in \cite{peche2007wigner}, Section 3.1. First consider the simple case (A), and assume given a closed even path $\mathcal{P}'$ of length $2m=2p-2\ell$, so we know edges in $\mathcal{P}'$ and the orders they are read. First we choose $J$ vertices with at most $C_{2m}^J$ choices. This splits $\mathcal{P}'$ into $J+1$ subpaths $\mathcal{R}_i,0\leq i\leq J$. $\mathcal{R}_i$ and $\mathcal{P}_i$ differ only by the order they are read (and possibly the directions they are read): so the number of ways to construct $\mathcal{P}_i$ from $\mathcal{R}_i$ is $J!2^J$. We find the number of ways to write $2\ell$ as the sum of $J$ positive integers, which is the way we assign the unreturned edges $S_i$, giving a factor $C_{2\ell}^J$. Finally, take an ordered collection of $2\ell-J$ edges from the edges in $\mathcal{P}'$: this is given in $\frac{(2m)!}{(2m-2\ell+J)!}$ methods: we only need $2\ell-J$ instead of $2\ell$ edges as the endpoints of each $S_i$ are known a priori. Finally we need to take into account the weight $\mathbb{E}\left[\prod_{j=0}^{p-1}\frac{x_{i_ji_{j+1}}}{\sqrt{n}}\prod_{j=p}^{2p-1}\frac{\overline{x}_{i_{j+1}i_j}}{\sqrt{n}}\right]$, which evaluated along path $\mathcal{P}$ will differ from that along $\mathcal{P}'$ by a factor depending on the sub-exponential moments of $x_{ij}$. This contribution cannot be bounded separately but has to be bounded taken account of weights on the double path $\mathcal{P}'$. As the gluing procedure adds multiplicity at most one to the edges in $\mathcal{P}'$, we only need to replace in equation \eqref{contributiongeneral} each term $(\text{const}\cdot k)^k$ by $(\text{const}\cdot k)^{k+1}$ to account for the additional weight, where the constant in the equation may change from line to line. In the worst case scenario we would have to consider (recall $m=p-\ell$, and $\sum_{k\geq 1}kn_k=m$) \begin{equation}\label{**contribution234} \hat{T}_m(n_1,\cdots,n_p):=
n\frac{1}{n^m}\frac{n!}{n_0!n_1!\cdots n_p!}\frac{m!}{\prod_{k=2}^m(k!)^{n_k}}\prod_{k=2}^p(2k)^{kn_k}\prod_{k=2}^s(\text{const}\cdot k)^{(k+1)n_k} \cdot 4^{\sum_{k\geq 2}kn_k},
\end{equation} which amounts to increasing the multiplicity of each edge in $\mathcal{P}'$ by one.

Via a similar computation as in \eqref{lastboundsssa}  we have $\sum_{k= 2}^p\frac{(\text{const}\cdot p)^kk^{k+1}}{n^{k-1}}\leq\frac{1}{4}$ whenever $p\leq C_0 \sqrt{n}$ for some fixed small constant $C_0>0$ depending only on the sub-exponential moments of $x_{ij}$. We get the following estimate, where $C_1<\frac{1}{4}$: 
$$\sum_{\sum_{k\geq 2}n_k>0}\widehat{T}_m(n_1,\cdots,n_p)\leq C_1n,\quad \widehat{T}_m(m,0,\cdots,0)\leq C_1n,
$$and that 
$$
\widehat{T}_m:=\sum_{n_1,\cdots,n_k\geq 0:\quad \sum_{k\geq 0}n_k=n,\quad \sum_{k\geq 0} kn_k=m}\widehat{T}_m(n_1,n_2,\cdots,n_p)\leq 2C_1n.
$$
Combining with the gluing procedure, we see that the contribution by paths $\mathcal{P}$ with $I=1$, $\ell>0$ is given by
$$\begin{aligned}
&\sum_{\ell=1}^m \widehat{T}_{p-\ell}\sum_{J=1}^{2\ell} C_{2p-2\ell}^J J!2^JC_{2\ell}^J \frac{(2p-2\ell)!}{(2p-4\ell+J)!} n^{-\ell}\\&\leq 2C_1n\cdot \sum_{\ell=1}^{p-1}\left(\frac{\text{const}(p-\ell)}{\sqrt{n}}
\right)^{2\ell}\leq 4C1C_0^2n
\end{aligned}
$$ provided that $p\leq C_0\sqrt{n}$. Thus contributions of $\mathcal{P}$ of class (A) is $O(n)$.

For paths $\mathcal{P}$ of class (B), we follow the ideas of \cite{peche2007wigner}, Section 3.1.2. Assume the closed even paths corresponding to the $I\geq 2$ clusters have lengths $2s_i,i=1,\cdots,I-1$. Each $s_i>0$ and $\sum_{i=0}^{I-1}2s_i=2p-2\ell$. Again we write these closed even paths by $W_i,i=0,\cdots,I-1$. The first closed even path makes a contribution computable by Theorem \ref{momenttheorem2.2}. As previously noted, the starting points of $W_1,\cdots,W_{i-1}$ are odd indices of $\mathcal{P}'$, so that their origins can be chosen in $C_{2p-2\ell}^{I-1}$ ways and the order of choice is irrelevant to the insertion procedure. Then we need to choose $J-(I-1)$ vertices from vertices in $\mathcal{P}'$: they give the endpoints of the odd edges $S_1,\cdots,S_J$ and the number of choice is $\binom{2p-2\ell-I+1}{J-I+1}$. The choice of all these $J$ vertices splits the path into subpaths $\mathcal{R}_i,i=1,\cdots,J$ and $\mathcal{P}_i$ differ from $\mathcal{R}_i$ only by ordering and directions. The contribution of such paths is upper bounded by 
\begin{equation}\label{contributionB}\begin{aligned}
&\sum_{\ell=1}^{p-1}\sum_{J=1}^{2\ell} 2^J J!\binom{2\ell}{J} \frac{(2p-2\ell)!}{(2p-4\ell+J)!}n^{-\ell}\sum_{I=2}^J\binom{2p-2\ell}{I-1}\binom{2p-2\ell-I+1}{J-I+1}\\&\times\sum_{s_0,\cdots,s_{I-1}:\sum_i s_i=2p-2\ell}n\prod_{i=0}^{\ell-1}(\frac{\widehat{T}_n}{n}),
\end{aligned}\end{equation}
where we consider $\frac{\widehat{T}}{n}$ because the origin of the path has been fixed. In the above computation we assume the $W_i$'s do not share an edge (if they share an edge, then we have one fewer vertex resulting in a factor $n^{-1}$, which compensates the higher moment factor $k$ we gain, so all such contributions are negligible). Noting that 
$$
\binom{2p-2\ell}{I-1}\binom{2p-2\ell-I+1}{J-I+1}\leq 2^J \binom{2p-2\ell}{J},
$$
    we see that \eqref{contributionB} is $O(n)$, so contributions from case (B) is $O(n)$.

    Finally consider case (C). For this we need an additional gluing procedure that can be found in the beginning of \cite{peche2007wigner}, Section 2.2. Denote by $\widetilde{i}$ the smallest index such that $W_{\widetilde{i}}$ has an odd edge, and $\widetilde{e}$ (resp. $t_{\widetilde{e}}$) the first occurrence (resp. the instance of first occurrence) of an odd edge in $W_{\widetilde{i}}$. Let $\widetilde{j}$ be the first instant where $W_{\widetilde{j}}$ also has the edge $\widetilde{e}$, and $t'_{\widetilde{e}}$ is the instance of first appearance of $\widetilde{e}$ in $W_{\widetilde{j}}$. If the selected occurrence of $\widetilde{e}$ has opposite directions in $W_{\widetilde{i}},W_{\widetilde{j}}$, then we can glue the two paths to define $W_{\tilde{i}}\vee W_{\widetilde{j}}$ via first reading the $t_{\widetilde{e}}-1$ edges of $W_{\widetilde{i}}$, switch to $W_{\widetilde{j}}$, read edges of $W_{\widetilde{j}}$ from the instant $t_{\widetilde{e}}+1$ until end of $W_{\widetilde{j}}$. Then we restart at the origin of $W_{\widetilde{j}}$ until (not including) the occurrence of $\widetilde{e}$, switch to $W_{\widetilde{i}}$, and read its edges until the end. The formed path $W_{\tilde{i}}\vee W_{\widetilde{j}}$ is obtained via removing occurrence of $\widetilde{e}$ twice.  If the occurrence of $\widetilde{e}$ has same direction then we proceed similarly with the gluing, but we read edges in $W_{\widetilde{e}}$ in reverse direction. We continue the procedure until we end up with $I-I_1$ closed even paths, and $I_1$ iterations of the gluing procedure result in a loss of $2I_1$ occurrences of edges. The final length of the path $\mathcal{P}'$ is $2p-2\ell-2I_1.$ Let $D_j,j=0,\cdots,I-I_1-1$ denote the closed even paths thus defined, so it has total length $2p-2\ell-2I_1$. To reconstruct the paths we record (i) the $I_1$ moments we remove the $I_1$ marked edges; (ii) the length and (iii) origins of all these $I_1$ paths when the switch occurs. The number of preimages, times the multiplicative factor $K/\sqrt{n}$ contributed by the erased edges, is bounded by (for sufficiently small $C_0>0$)
    $$\binom{2p}{I_1}\times (\frac{\text{const}
\times s^2}{n})^{I_1}\leq 10^{-2} \binom{2p}{I_1},\quad s\leq C_0\sqrt{n}, 
    $$and this is enough to show the contribution from case (C) is bounded by contribution from case (A) or case (B). More precisely, assume that $I-I_1>1$ and we reduce to case (B) (reduction to case (A) when $I-I_1=1$ is analogous). Suppose that $D_i$ is obtained by gluing $I_i'+1$ paths, we only need to bound
\begin{equation}\label{upperboundforthegames}\begin{aligned}
&\sum_{\ell=1}^{p-1}\sum_{J=1}^{2\ell} 2^JJ! \binom{2\ell}{J}(\frac{\text{const}}{ n})^\ell \frac{(2p-2\ell)!}{(2p-4\ell+J)!} \sum_{I=2}^J \binom{2p-2\ell-I+1}{J-I+1} \sum_{I_1=1}^{I-1} \binom{2p-2\ell}{I-I_1-1}\\&\times \sum_{\sum_i s_i'=2p-2\ell-2I_1} n\prod_{i=0}^{I-I_1-1}\binom{2s_i'}{I_i'}\frac{\widehat{T}_{2s_i'}}{n} (\frac{s_i^2}{n})^{I_i'}
.\end{aligned}\end{equation}
This expression comes from replacing the terms $\binom{2p-2\ell}{I-1}$ in expressions in case (B) (or (A)) \eqref{contributionB} by the left hand side of the following expression 
$$
\binom{2p-2\ell}{I-I_1-1}\sum_{I_0',I_1',\cdots} \prod_{i=0}^{I-I_1-1} (\frac{s_i^2}{n})^{I_1'}
\binom{2s_i'}{I_i'}\leq \binom{2p-2\ell}{I-1}(\frac{p^2}{n})^{I_1}\text{const}^\ell,
$$
we see that whenever we choose $C_0>0$ sufficiently small, then for all $p\leq C_0\sqrt{n}$, the contribution from case (C) as in the expression \eqref{upperboundforthegames} is bounded by a constant multiple of the contribution from case (A) or (B). This completes the proof of the theorem.
    
\end{proof}

\section*{Funding}
The author is supported by a Simons Foundation Grant (601948, DJ)

\printbibliography

\end{document}